\documentclass[11pt,reqno]{amsart}

\usepackage{amsfonts}
\usepackage{eurosym}
\usepackage{amssymb}
\usepackage{amsthm}
\usepackage{amsmath}
\usepackage{bm}
\usepackage{cite}
\usepackage{mathrsfs}
\usepackage{xcolor}
\usepackage[OT1]{fontenc}
\usepackage[left=2cm, right=2cm, top=2.5cm, bottom=2.5cm]{geometry}
\usepackage{hyperref}
\usepackage{times}
\usepackage{tikz}
\usepackage{tikzpagenodes}
\usepackage[shortlabels]{enumitem}
\usepackage{enumitem}

\hypersetup{colorlinks=true, linkcolor=blue, citecolor=red, urlcolor=blue}
\newcommand{\subscript}[2]{$#1 _ #2$}
\allowdisplaybreaks
\newtheorem{theorem}{Theorem}[section]

\newtheorem{lemma}[theorem]{Lemma}

\theoremstyle{remark}
\newtheorem{remark}[theorem]{Remark}

\def\d{{\rm d}}

\def\n{\mathbf{n}}

\newcommand{\numberset}{\mathbb}

\newcommand{\R}{\numberset{R}}
\everymath{\displaystyle}
\def \au {\rm}
\def \ti {\it}
\def \jou {\rm}
\def \bk {\it}
\def \no#1#2#3 {{\bf #1} (#3), #2.}
\def \eds#1#2#3 {#1, #2, #3.}

\begin{document}
\title[Global well-posedness and asymptotic behavior for the nonlocal AGG
model]{Global well-posedness and convergence to equilibrium for the
Abels-Garcke-Gr\"{u}n model with nonlocal free energy}
\author{\textsc{Ciprian G. Gal}$^{\dagger }$, \textsc{Andrea Giorgini}$%
^{\ast}$, \textsc{Maurizio Grasselli}$^{\ast }$, \textsc{Andrea Poiatti}$%
^{\ast }$}
\address{$^\dagger$Department of Mathematics,\\
Florida International University\\
Miami, FL 33199, USA
}
\email{cgal@fiu.edu}
\address{$^\ast$Dipartimento di Matematica\\
Politecnico di Milano\\
Milano 20133, Italy}
\email{andrea.giorgini@polimi.it,\! maurizio.grasselli@polimi.it,\! andrea.poiatti@polimi.it}
\maketitle

\begin{abstract}
We investigate the nonlocal version of the Abels-Garcke-Gr\"{u}n (AGG)
system, which describes the motion of a mixture of two
viscous incompressible fluids. This consists of the incompressible
Navier-Stokes-Cahn-Hilliard system characterized by concentration-dependent
density and viscosity, and an additional flux term due to interface
diffusion. In particular, the Cahn-Hilliard dynamics of the concentration
(phase-field) is governed by the aggregation/diffusion competition of the
nonlocal Helmholtz free energy with singular (logarithmic) potential and
constant mobility. We first prove the existence of global \textit{strong}
solutions in general two-dimensional bounded domains and their uniqueness
when the initial datum is strictly separated from the pure phases. The key
points are a novel well-posedness result of strong solutions to the nonlocal
convective Cahn-Hilliard equation with singular potential and constant
mobility under minimal integral assumption on the incompressible velocity
field, and a new two-dimensional interpolation estimate for the $L^4(\Omega)$
control of the pressure in the stationary Stokes problem. Secondly, we show
that any weak solution, whose existence was already known, is globally
defined, enjoys the propagation of regularity and converges towards an
equilibrium (i.e., a stationary solution) as $t\rightarrow \infty$. Furthermore, we
demonstrate the uniqueness of strong solutions and their continuous
dependence with respect to general (not necessarily separated) initial data
in the case of matched densities and unmatched viscosities (i.e., the
nonlocal model H with variable viscosity, singular potential and constant
mobility). Finally, we provide a stability estimate between the strong
solutions to the nonlocal AGG model and the nonlocal Model H in terms of the
difference of densities.
\end{abstract}

\tableofcontents





\section{Introduction and main results}

\label{1}

In the Diffuse Interface theory, the motion of a mixture of two
incompressible viscous Newtonian fluids and the evolution of the interface
separating the bulk phases have been originally modeled by the so-called Model
H (see, e.g., \cite{gurtin,halperin}). This leads to the following
Navier-Stokes-Cahn-Hilliard system:
\begin{equation}
\begin{cases}
\rho \partial _{t} \mathbf{u} +\rho (\mathbf{u}\cdot \nabla )\mathbf{u}-%
\mathrm{div}\,(\nu (\phi )D\mathbf{u})+\nabla \Pi =-\mathrm{div}\,\left(
\nabla \phi \otimes \nabla \phi \right) , \\
\mathrm{div}\,\mathbf{u}=0, \\
\partial _{t}\phi +\mathbf{u}\cdot \nabla \phi =\mathrm{div}\,(m(\phi
)\nabla \mu ), \\
\mu =-\Delta \phi +\Psi ^{\prime }(\phi ),%
\end{cases}
\label{modelH}
\end{equation}%
in $\Omega \times (0,\infty )$. Here, $\Omega $ is a bounded domain in $%
\mathbb{R}^{d}$, $d=2,3$, $\mathbf{u}$ represents the (volume
averaged) velocity, $D\mathbf{u}=(\nabla \mathbf{u}+(\nabla \mathbf{u}%
)^{T})/2$ is the symmetric strain tensor, $\Pi $ denotes the pressure, $\nu
(\cdot )>0$ is the viscosity of the mixture, $\rho $ is the \textit{constant}
mixture density, $m(\cdot )\geq 0$ is the mobility function, and $\Psi $ is
the Flory-Huggins double-well potential defined by
\begin{equation}
\Psi (s)=\frac{\alpha }{2}\left( (1+s)\ln (1+s)+(1-s)\ln (1-s)\right) -\frac{%
\alpha _{0}}{2}s^{2}=F(s)-\frac{\alpha _{0}}{2}s^{2},\quad \forall \,s\in
\lbrack -1,1],  \label{potential}
\end{equation}%
where the two positive parameters $\alpha ,\alpha _{0}$ satisfy the
relations $0<\alpha <\alpha _{0}$. This potential, in particular, ensures the
existence of \textit{physical solutions}, that is, solutions such that $\phi
\in \lbrack -1,1]$. One of the fundamental modeling assumptions of %
\eqref{modelH} is that the densities of both components match and thereby
the density of the mixture $\rho $ is constant. This restricts the
applicability of the model to those fluid mixtures having a negligible
difference between the two densities. To overcome it, the so-called
Abels-Garcke-Gr\"{u}n (AGG) system has been introduced in the seminal work
\cite{agg} as a thermodynamically consistent generalization of the Model H,
allowing to treat fluids with unmatched densities. The AGG model
reads as follows
\begin{equation}
\begin{cases}
\partial _{t}(\rho (\phi )\mathbf{u})+\mathrm{div}\,\left(\mathbf{u}\otimes
\left( \rho (\phi )\mathbf{u}+\mathbf{J}\right) \right) -\mathrm{div}\,(\nu
(\phi )D\mathbf{u})+\nabla \Pi =-\mathrm{div}\,\left( \nabla \phi \otimes
\nabla \phi \right) , \\
\mathrm{div}\,\mathbf{u}=0, \\
\partial _{t}\phi +\mathbf{u}\cdot \nabla \phi =\mathrm{div}\,(m(\phi
)\nabla \mu ), \\
\mu =-\Delta \phi +\Psi ^{\prime }(\phi ),%
\end{cases}
\label{AGGlocal2d}
\end{equation}%
in $\Omega \times (0,\infty )$, where
\begin{equation}
\rho (\phi )=\rho _{1}\frac{1+\phi }{2}+\rho _{2}\frac{1-\phi }{2},\quad
\mathbf{J}=-\frac{\rho _{1}-\rho _{2}}{2}m(\phi )\nabla \mu .
\label{RHO-Jtilde}
\end{equation}%
System \eqref{AGGlocal2d} is usually supplemented with the boundary and
initial conditions
\begin{equation}
\begin{cases}
\mathbf{u}=0,\quad \partial _{\n}\phi =\partial _{\mathbf{n}}\mu =0,\quad &
\text{on }\partial \Omega \times (0,\infty), \\
\mathbf{u}|_{t=0}=\mathbf{u}_{0},\quad \phi |_{t=0}=\phi _{0},\quad & \text{%
in }{\Omega },%
\end{cases}
\label{locAGG-bc}
\end{equation}%
where $\n$ is the unit outward normal vector to $\partial
\Omega $.

In both Model H and AGG model, the fluid mixture is driven by the capillary forces
$-\mathrm{div}\,(\nabla \phi \otimes \nabla \phi )$, accounting for the
surface tension effect, together with a partial diffusive mixing. The latter
is assumed in the interfacial region and it is modeled by $\Delta \mu $. The
specificity of the AGG model compared to the Model H lies in the presence of
the flux term $\mathbf{J}$. In contrast to the one-phase flow, the (average)
density $\rho (\phi )$ in \eqref{AGGlocal2d} does not satisfy the continuity
equation with respect to the flux associated with the velocity $\mathbf{u}$.
Instead, the density satisfies the continuity equation with a flux given by
the sum of the transport term $\rho (\phi )\mathbf{u}$ and the term $\mathbf{%
J}$, which is due to the diffusion of the concentration in the unmatched
densities case\footnote{%
Indeed, $\rho =\rho \left( \phi \right) $ satisfies the continuity equation $%
\partial _{t}\rho +\mathrm{div}\,\left( \rho \mathbf{u}+\mathbf{J}\right)
=0. $} (see also \cite{AGG2d} and the references therein). Notice that we recover \eqref{modelH} when $\rho _{1}=\rho _{2}$ in \eqref{AGGlocal2d}-\eqref{RHO-Jtilde}.

Concerning the mathematical analysis of the AGG model \eqref{AGGlocal2d}-%
\eqref{locAGG-bc}, the existence of global weak solutions were proven in
\cite{abels} and \cite{ADG2} in the case of strictly positive and degenerate
mobility $m(\cdot)$, respectively. The existence of global weak solutions
has been extended in \cite{AB} to viscous non-Newtonian binary fluids (with
constant mobility) and in \cite{GalGW} to the case of dynamic boundary
conditions (with strictly positive mobility). The convergence of a fully
discrete numerical scheme to weak solutions was shown in \cite{GGM2016}.
More advanced issues related to the well-posedness, regularity and
longtime behavior have obtained a renewed interest in the last years. In
\cite{AW2021}, the local well-posedness of strong solutions has been proven in
three dimensions for polynomial-like potentials $\Psi_{\mathrm{%
pol}}$ and strictly positive mobility. We point out that the solution in
\cite{AW2021} may not satisfy $\vert\phi(x,t)\vert\leq 1$ in the space-time domain.
In \cite{AGG2d}, the well-posedness of (local-in-time) strong
solutions in two dimensional bounded domains has been obtained for the
logarithmic potential \eqref{potential} and constant mobility. In this case,
$\phi$ takes its values in the physical range $[-1,1]$.
If the boundary conditions are periodic then the strong solutions
are globally defined in time (see \cite{AGG2d}). The case of
bounded three-dimensional domains has been investigated in \cite{AGG3d},
where the well-posedness of local strong solutions is shown. More recently,
the propagation of regularity in time of any weak solutions in three
dimensions and its stabilization towards an equilibrium state as $t \to
\infty$ have been achieved in \cite{AGG2022} in the case of constant
mobility (see also \cite{A2009,GMT2019} for the matched density case). In the latter, the authors also discussed the global
well-posedness in bounded two-dimensional domains and the deep quench limit. 
We conclude this part by mentioning that the realm of Diffuse Interface
(phase field) models for fluid mixtures has been widely deepened in the past
decades. Several models have been proposed to describe binary mixtures with
non-constant density relying on different assumptions. We refer the
interested reader to the models derived, e.g., in \cite{B2002, DSS2007, GKL2018,
HMR2012, LT1998, SSBZ2017} and the analysis carried out in \cite{A2009-2,
A2012, B2001, GT2020, KZ2015}.

The evolution of the phase-field variable in both Model H and AGG model is
modeled by the local Cahn-Hilliard equation driven by divergence-free drift.
The chemical potential $\mu $ in \eqref{AGGlocal2d} is defined as the first
variation of the Ginzburg-Landau free energy
\begin{equation}
{\mathcal{E}}_{\mathrm{loc}}(\phi )=\int_{\Omega }\frac{1}{2}|\nabla \phi
|^{2}+\Psi (\phi )\,\mathrm{d}x.  \label{energyphi}
\end{equation}%
The free energy $\mathcal{E}_{\mathrm{loc}}(\phi )$ only focuses on short
range interactions between particles. Indeed, the gradient square term
accounts for the fact that the local interaction energy is spatially
dependent and varies across the interfacial surface due to spatial
inhomogeneities in the concentration. Going back to the general approach of
statistical mechanics, the mutual short and long-range interactions between
particles is described through convolution integrals weighted by
interactions kernels. Based on this ancient approach (see \cite{vanderwaals}),
Giacomin and Lebowitz (%
\cite{GL1996, GL1997, GL1998}) observed that a physically more rigorous
derivation leads to a nonlocal dynamics, which is the nonlocal Cahn-Hilliard
equation. In particular, this equation is rigorously justified as a
macroscopic limit of microscopic phase segregation models with particles
conserving dynamics. In this case, the gradient term is replaced by a
nonlocal spatial interaction integral, namely
\begin{equation}
{\mathcal{E}}_{\mathrm{nloc}}(\phi )=\int_{\Omega }F(\phi (x))\,\mathrm{d}x-%
\frac{1}{2}\int_{\Omega }\int_{\Omega }J(x-y)\phi (x)\phi (y)\,\mathrm{d}x\,%
\mathrm{d}y,  \label{nonloc-en}
\end{equation}%
where $J$ is a sufficiently smooth symmetric interaction kernel. As shown in
\cite{GL1997} (see also \cite{GGG2022} and the references therein), the
energy $\mathcal{E}_{\mathrm{loc}}$ can be seen as an approximation of $%
\mathcal{E}_{\mathrm{nloc}}$, as long as we suitably redefine $\Psi $ as $%
\widetilde{\Psi }(x,s)=F(s)-(J\ast 1)(x)s^{2}/2$. The physical relevance of
nonlocal interactions was already pointed out in the pioneering paper \cite%
{vanderwaals} (see also \cite[4.2]{emmerich} and references therein) and
studied for different kind of evolution equations, mainly Cahn-Hilliard and
phase field systems, see, e.g., \cite{BEG2007, CKRS2007, zach, GGG2017, GG,
KRS2007, LP02011,LP2011}. In this context, the nonlocal AGG model we want to analyze
reads as
\begin{equation}
\begin{cases}
\partial _{t}(\rho (\phi )\mathbf{u})+\mathrm{div}\,\left( \mathbf{u}
\otimes \left( \rho (\phi )\mathbf{u}+\mathbf{J}\right) \right) -\mathrm{div}%
\,(\nu (\phi )D\mathbf{u})+\nabla \Pi =\mu \nabla \phi , \\
\mathrm{div}\,\mathbf{u}=0, \\
\partial _{t}\phi +\mathbf{u}\cdot \nabla \phi =\mathrm{div}\,(m(\phi
)\nabla \mu ), \\
\mu =F^{\prime }(\phi )-J\ast \phi ,%
\end{cases}
\label{syst2}
\end{equation}%
in $\Omega \times (0,\infty )$, where  $F$, $\rho $ and $\mathbf{J}$  are
given in \eqref{potential} and \eqref{RHO-Jtilde}. System %
\eqref{syst2} is endowed with the boundary and initial conditions
\begin{equation}
\begin{cases}
\mathbf{u}=0,\quad \partial _{\mathbf{n}}\mu =0\quad & \text{on }\partial
\Omega \times (0,\infty), \\
\mathbf{u}|_{t=0}=\mathbf{u}_{0},\quad \phi |_{t=0}=\phi _{0}\quad & \text{%
in }{\Omega }.%
\end{cases}
\label{bic}
\end{equation}

Regarding the mathematical analysis of the nonlocal AGG system, there are
much fewer contributions than in the local case. A first nonlocal variant has been investigated in \cite{Abelsterasawa}, where the authors proved the existence
of weak solutions in two and three dimensional bounded domains in the case
of a singular nonlocal free energy and strictly positive mobility. More
precisely, the ``diffusive'' term $|\nabla \phi |^{2}$ in $E_{\mathrm{loc}}$
is replaced by singular nonlocal operator which controls the $H^{\frac{%
\alpha }{2}}$ norm of the concentration for $\alpha \in (0,2)$ (as a
consequence, $\Delta \phi $ in \eqref{AGGlocal2d}$_{4}$ is replaced by a
regional fractional Laplacian). The nonlocal AGG system \eqref{syst2}-\eqref{bic} has only been analyzed so far
in \cite{Frig15} and in \cite{Frigeri}. In the former, the
existence of weak solutions is shown in the case of singular (logarithmic)
potential and strictly positive mobility in two and three dimensional
bounded domains. In the latter, the existence of weak solutions is proven in
the case of singular potential and degenerate mobility. More recently, the authors in \cite{Abelsterasawa2} have shown that such weak solutions converge to those of \eqref{AGGlocal2d} in the setting introduced in \cite{DST2021}-\cite{DST2021-2}. In addition, in \cite%
[Theorem 3.5]{Frigeri} the author proved that, under suitable assumptions on
the initial datum, the kernel $J$, the potential $F$ and the degenerate
mobility, the concentration function $\phi $ enjoys the regularity $%
L^{2}(0,T;H^{2}(\Omega ))\cap H^{1}(0,T;L^{2}(\Omega ))$ in two dimensions
(cf. Theorem \ref{R:weak-sol} for weak solutions reported below). In particular, it is
worth pointing out that the assumptions $\mathbf{(A4)}$ and $\mathbf{(A1b)}$ in
\cite{Frigeri}, i.e. $m(\cdot )F^{\prime \prime }(\cdot )\in C^{1}([-1,1])$
such that $m(\cdot )F^{\prime \prime }(\cdot )$ is strictly positive, allows
to rewrite \eqref{syst2}$_{3,4}$ as
\begin{equation}
\partial _{t}\phi +\mathbf{u}\cdot \nabla \phi =\mathrm{div}\left( m(\phi
)F^{\prime \prime }(\phi )\nabla \phi -\nabla J\ast \phi \right) ,\quad
\text{in}\ \Omega \times (0,\infty ),  \label{div-formCH}
\end{equation}%
where the main diffusion operator is a second-order divergence form operator
with positive and bounded coefficients. On the other hand, if $m$ is
constant (or even non-degenerate) at the endpoints $\pm 1,$ the product $%
mF^{\prime \prime }(\cdot )$ no longer enjoys the typical cancellation
effect seen with phase segregation phenomena when $m\left( \pm 1\right) =0$
(see \cite{FGG2021}). In particular, the (variable) diffusion coefficient $%
mF^{\prime \prime }(\cdot )$ in (\ref{div-formCH}) is highly singular and
unbounded since $mF^{\prime \prime }(s)=\alpha m(1-s^{2})^{-1}$ , owing to
(\ref{potential}).
To the best of our knowledge, there are yet no results concerning the
global well-posedness for the nonlocal AGG model \eqref{syst2}-\eqref{bic}
in dimension two in the case of singular potentials and constant mobility.
As a matter of fact, the same open questions remain still unresolved even
for the nonlocal Model H with unmatched viscosities and logarithmic-like
potential, which corresponds to \eqref{syst2}-\eqref{bic} with $\rho
_{1}=\rho _{2}$. In fact, beyond the global existence of weak solution for
the nonlocal Model H established in \cite{FG}, the uniqueness of weak
solutions, their propagation of regularity and their longtime behavior in
two dimensions has been discussed in \cite{FGG} and \cite{GGG2017} in the
case of constant viscosity only.

The aim of the present paper is to present the first well-posedness result
concerning the nonlocal AGG model (with unmatched densities and viscosities)
in presence of singular-like potentials and constant mobility. In fact,
following \cite{GGG2022}, we consider a general class of entropy potentials,
commonly employed for complex binary particle systems experiencing long
range interactions. This class generalizes the classical logarithmic density
function (\ref{potential}). Recall that the latter is uniquely generated by
Boltzmann-Gibbs statistics of macroscopic mixing of the fluid constituents.
To this end, let us first state the main assumptions, which will be adopted
throughout our analysis:

\begin{enumerate}[label=(\subscript{H}{{\arabic*}})]
\item \label{Omega} $\Omega $ is a bounded domain in $\mathbb{R}^{2}$ with
boundary $\partial \Omega $ of class $C^{2}$.

\item   \label{nu} The viscosity $\nu \in W^{1,\infty }(\mathbb{R})$
satisfies
\begin{equation*}
0<\nu _{\ast }\leq \nu (s)\leq \nu ^{\ast },\quad \forall \,s\in \mathbb{R},
\end{equation*}%
for some positive values $\nu _{\ast },\nu ^{\ast }$.

\item   \label{J-ass} The interaction kernel $J\in W^{1,1}(\mathbb{R}^{2})$ is such that $J(x)=J(-x)$
for all $x\in \mathbb{R}^{2}$.

\item   \label{F-ass-1} $F\in C([-1,1])\cap C^{2}(-1,1)$ fulfills
\begin{equation*}
\lim_{s\rightarrow -1}F^{\prime }(s)=-\infty ,\quad \lim_{s\rightarrow
1}F^{\prime }(s)=+\infty ,\quad F^{\prime \prime }(s)\geq {\alpha },\quad
\forall \,s\in (-1,1).
\end{equation*}%
We extend $F(s)=+\infty $ for any $s\notin \lbrack -1,1]$. Without loss of
generality, $F(0)=0$ and $F^{\prime }(0)=0$. In particular, this entails
that $F(s)\geq 0$ for any $s\in \lbrack -1,1]$.

\item  \label{h3} We assume\footnote{%
Without loss of generality we also assume that $F^{\prime \prime }$ is
symmetric on $\left( -1,1\right) ,$ see \cite{GGG2022}. Other statistical
entropy density functionals (that can be more singular at $\pm 1$ than the
logarithmic density (\ref{potential})) from information theory are included
in this study.} that $F^{\prime \prime }$ is monotone non-decreasing on $%
[1-\varepsilon _{0},1)$, for some $\varepsilon _{0}>0$ and there exist $p\in
\lbrack 2,\infty )$ and a continuous function $h:\left( 0,1\right)
\rightarrow \mathbb{R}_{+},$ $h\left( \delta \right) =o(\delta ^{4/p}),$ as $%
\delta \rightarrow 0^{+},$ such that%
\begin{equation}
\begin{cases}
F^{\prime \prime }(s)\leq C\mathrm{e}^{C|F^{\prime }(s)|^{\beta }},
\quad &\text{for all }\,s\in (-1,1),  \\
F^{\prime \prime }\left( 1-2\delta \right) h\left( \delta \right) \geq 1,%
\quad &\text{for all }\delta \leq \frac{\varepsilon_{0}}{2}, 
\end{cases}  \label{sep}
\end{equation}%
for some $C>0$ and $\beta \in \lbrack 1,2).$

\item   \label{m-ass} The mobility is constant and equal to unity, i.e., $m\equiv 1$.
\end{enumerate}

\begin{remark}
The logarithmic convex function $F$ in \eqref{potential} fulfills \ref%
{F-ass-1} and \ref{h3} (cf. \cite{GGG2022}). A common form for the
viscosity is the following
\begin{equation*}
\nu (s)=\nu _{1}\frac{1+s}{2}+\nu _{2}\frac{1-s}{2},\quad s\in \lbrack -1,1],
\end{equation*}%
which can be easily extended on the whole $\mathbb{R}$ in such way to comply
with \ref{nu}. Many other examples of entropy densities satisfying \ref%
{F-ass-1} and \ref{h3} (including the Tsallis entropy) can be found in
\cite[Sections 6.2, 6.3]{GGG2022}.
\end{remark}

Before proceeding with the statements of the main results, we report the
only available result for the nonlocal AGG system \eqref{syst2}-\eqref{bic},
which concerns the existence of weak solutions on any fixed time interval $%
(0,T)$ proven in \cite{Frig15}. For the sake of completeness,
we report it in the original form with the non-degenerate
mobility. We refer the reader to Section \ref{prel} for functional space notation.

\begin{theorem}
\label{R:weak-sol} Let \ref{Omega}-\ref{h3} hold and let $m\in C_{%
\mathrm{loc}}^{1,1}(\mathbb{R})$ such that $0<m_{\ast }\leq m(s)\leq m^{\ast
}$ for all $s\in \mathbb{R}$ for some $m_{\ast }$ and $m^{\ast }$. Assume
that $\mathbf{u}_{0}\in L_{\sigma }^{2}(\Omega )$ and $\phi _{0}\in
L^{\infty }(\Omega )$ with $F(\phi _{0})\in L^{1}(\Omega )$ and $|\overline{%
\phi _{0}}|<1$. Then, for any $T>0$, there exists a weak solution to %
\eqref{syst2}-\eqref{bic} in $(0,T)$ such that

\begin{itemize}
\item[(i)] \label{E11} The pair $(\mathbf{u},\phi )$ satisfies the
properties
\begin{equation}
\begin{cases}
\mathbf{u}\in C_{\mathrm{w}}([0,T];L_{\sigma }^{2}(\Omega ))\cap
L^{2}(0,T;H_{0,\sigma }^{1}(\Omega )), \\
\phi \in L^{\infty }(\Omega \times (0,T))\cap L^{2}(0,T;H^{1}(\Omega ))%
\text{ with }|\phi|<1\ \text{a.e. in }\Omega \times (0,T), \\
\partial _{t}\left( \rho (\phi )\mathbf{u}\right) \in L^{\frac{4}{3}%
}(0,T;H_{0,\sigma }^{2}(\Omega )^{\prime }),\quad \partial _{t}\phi \in
L^{2}(0,T;H^{1}(\Omega )^{\prime }), \\
\mu =F^{\prime }(\phi )-J\ast \phi \in L^{2}(0,T;H^{1}(\Omega )).
\end{cases}
\label{regg-weak}
\end{equation}

\item[(ii)] \label{E2} The solution $(\mathbf{u},\phi )$ fulfills the system
in weak sense:%
\begin{align}
& \langle \partial _{t} (\rho (\phi )\mathbf{u}) ,\mathbf{w}%
\rangle _{H_{0,\sigma }^{2}(\Omega )}-\left( \rho (\phi )\mathbf{u}\otimes
\mathbf{u},D\mathbf{w}\right) +\left( \nu (\phi )D\mathbf{u},D\mathbf{w}%
\right) -\left( \mathbf{u},\left( \mathbf{J}\cdot \nabla \right) \mathbf{w}%
\right) =-\left( \phi \nabla \mu ,\mathbf{w}\right) , \\
& \langle \partial _{t}\phi ,v\rangle _{H^{1}(\Omega )}-\left( \phi \,%
\mathbf{u},\nabla v\right) +\left( m(\phi )\nabla \mu ,\nabla v\right) =0,
\label{weak-NS}
\end{align}%
for any $\mathbf{w}\in H_{0,\sigma }^{2}(\Omega )$, $v\in H^{1}(\Omega )$
and for almost any $t\in (0,T)$.

\item[(iii)] \label{E3} The initial conditions $\mathbf{u}(\cdot ,0)=\mathbf{%
u}_{0}$ and $\phi (\cdot ,0)=\phi _{0}$ hold in $\Omega $.

\item[(iv)] \label{E4} The energy inequality
\begin{equation}
E(\mathbf{u}(t),\phi (t))+\int_{s}^{t}\left\Vert \sqrt{\nu (\phi (\tau ))}D%
\mathbf{u}(\tau )\right\Vert _{L^{2}(\Omega )}^{2}+\left\Vert \sqrt{m(\phi )}%
\nabla \mu (\tau )\right\Vert _{L^{2}(\Omega )}^{2}\,\mathrm{d}\tau \leq E(%
\mathbf{u}(s),\phi (s))  \label{energy-ineq}
\end{equation}%
holds for all $t\in \lbrack s,\infty )$ and almost all $s\in \lbrack
0,\infty )$ (including $s=0$), where the total energy is defined as
\begin{equation}
E(\mathbf{u},\phi ):=\frac{1}{2}\int_{\Omega }\rho (\phi )|\mathbf{u}|^{2}%
\, \mathrm{d}x+\int_{\Omega }F(\phi ) \, \mathrm{d}x-\frac{1}{2}\int_{\Omega
}(J\ast \phi )\phi \,  \mathrm{d}x.
\label{TOTALEN}
\end{equation}
\end{itemize}
\end{theorem}

\begin{remark}
The result of \cite[Theorem 1]{Frig15} actually holds for a slightly
different model than \eqref{syst2}. Indeed, the nonlocal Helmholtz free
energy considered in \cite{Frig15} is
\begin{equation*}
E_{\mathrm{nloc}}^{\star }(\phi )=\frac{1}{4}\int_{\Omega }\int_{\Omega
}J(x-y)\left( \phi (x)-\phi (y)\right) ^{2}\mathrm{d}x\,\mathrm{d}%
y+\int_{\Omega }F(\phi )-\frac{\alpha _{0}}{2}\phi ^{2} \, \d x.
\end{equation*}%
As a consequence, the chemical potential is $\mu =a\phi -J\ast \phi
+F^{\prime }(\phi )-\alpha _{0}\phi $. On the other, as explained in \cite%
{GGG2017} the two problems are strictly related and \cite[Theorem 1]{Frig15}
can be extended also to the problem in our analysis. More precisely, the two
models are equivalent as long as we suppose $\alpha _{0}(x)=a(x)$, where $%
a(x)=J\ast 1$, and the differences in the analysis provided in \cite{Frig15}
are related only to lower order terms.
\end{remark}

Our first main result concerns the global existence and uniqueness of strong
solutions to \eqref{syst2}-\eqref{bic}.

\begin{theorem}
\label{MR:strong} Let the assumptions \ref{Omega}-\ref{m-ass} hold.
Assume that $\mathbf{u}_{0}\in H_{0,\sigma }^{1}(\Omega )$, $\phi _{0}\in
H^{1}(\Omega )$, with $|\overline{\phi _{0}}|<1$, $F^{\prime }(\phi _{0})\in
L^{2}(\Omega )$ and $F^{\prime \prime }(\phi _{0})\nabla \phi _{0}\in
L^{2}(\Omega ;\mathbb{R}^{2})$. Then, there exists a global strong solution $%
(\mathbf{u},\Pi ,\phi ):\Omega \times \lbrack 0,\infty )\rightarrow \mathbb{R%
}^{2}\times \mathbb{R}\times \mathbb{R}$ to \eqref{syst2}-\eqref{bic} such
that:

\begin{itemize}
\item[(i)] \label{K1} The solution $(\mathbf{u},\Pi ,\phi )$ satisfies the
properties
\begin{equation}
\begin{cases}
\mathbf{u}\in BC([0,\infty );H_{0,\sigma }^{1}(\Omega ))\cap L_{\mathrm{%
uloc}}^{2}([0,\infty );H_{0,\sigma }^{2}(\Omega ))\cap H_{\mathrm{uloc}%
}^{1}([0,\infty );L_{\sigma }^{2}(\Omega )), \\
\Pi \in L_{\mathrm{uloc}}^{2}([0,\infty );H_{(0)}^{1}(\Omega )), \\
\phi \in BC_{\mathrm{w}}([0,\infty );H^{1}(\Omega ))\cap L_{\mathrm{uloc}%
}^{q}([0,\infty );W^{1,p}(\Omega )),\quad q=\frac{2p}{p-2},\quad \forall
\,p\in (2,\infty ), \\
\phi \in L^{\infty }(0,\infty ;L^{\infty }(\Omega )):\quad |\phi (x,t)|<1\
\text{ for a.a. } x\in\Omega ,\,\forall \,t\in \lbrack 0,\infty ), \\
\partial _{t}\phi \in L^{\infty }(0,\infty ;H^{1}(\Omega )^{\prime })\cap
L^{2}(0,\infty ;L^{2}(\Omega )),\quad F^{\prime }(\phi )\in L^{\infty
}(0,\infty ;H^{1}(\Omega )), \\
\mu \in BC_{\mathrm{w}}([0,\infty );H^{1}(\Omega ))\cap L_{\mathrm{uloc}%
}^{2}([0,\infty );H^{2}(\Omega ))\cap H_{\mathrm{uloc}}^{1}([0,\infty
);H^{1}(\Omega )^{\prime }).
\end{cases}
\label{regg}
\end{equation}

\item[(ii)] \label{K2} $(\mathbf{u},\Pi ,\phi )$ fulfills the
system \eqref{syst2} almost everywhere in $\Omega \times (0,\infty )$ and
the boundary condition $\partial _{\mathbf{n}}\mu =0$ almost everywhere on $%
\partial \Omega \times (0,\infty )$.

\item[(iii)] \label{K3}   $(\mathbf{u},\Pi ,\phi )$ is such that $%
\mathbf{u}(\cdot ,0)=\mathbf{u}_{0}$ and $\phi (\cdot ,0)=\phi _{0}$ in $%
\Omega $.

\item[(iv)] \label{K4} For any $\tau >0$, there exists $\delta =\delta (\tau
)\in (0,1)$ (depending on the norms of the initial datum) such that
\begin{equation}
\sup_{t\in \lbrack \tau ,\infty )}\Vert \phi (t)\Vert _{L^{\infty }(\Omega
)}\leq 1-\delta .  \label{delt}
\end{equation}
\end{itemize}

Moreover, if we additionally assume that $\Vert \phi _{0}\Vert _{L^{\infty
}(\Omega )}\leq 1-\delta _{0}$, for some $\delta _{0}\in (0,1)$, then there
exists $\delta ^{\star }>0$ such that
\begin{equation}
\label{SP-Star}
\sup_{t\in \lbrack 0,\infty )}\Vert \phi (t)\Vert _{L^{\infty }(\Omega
)}\leq 1-\delta ^{\star }.
\end{equation}%
As a consequence, $\partial _{t}\mu \in L_{\mathrm{uloc}}^{2}([0,\infty
);L^{2}(\Omega ))$, and 
$(\mathbf{u},\phi )$ is unique and depends continuously on the initial data
in $L_{\sigma }^{2}(\Omega )\times L^{2}(\Omega )$ on $[0,T]$, for any $T>0$.
More precisely, if $(\mathbf{u}_{j},\Pi _{j},\phi _{j})$ is the strong solutions
to \eqref{syst2}-\eqref{bic}
originating from the initial datum ($\mathbf{u}^j_{0},\phi^j_{0}$), $j=1,2$, then
\begin{equation}
\label{contdep}
\Vert \mathbf{u}_1(t) - \mathbf{u}_2(t)\Vert _{L^{2}(\Omega )}^{2}
+\Vert (\phi_1(t) - \phi_2(t)\Vert_{L^{2}(\Omega )}^{2}\leq C\left( \Vert \mathbf{u}^1_{0} - \mathbf{u}^2_{0}\Vert _{L^{2}(\Omega)}^{2}
+\Vert \phi^1_{0} - \phi^2_{0}\Vert _{L^{2}(\Omega )}^{2}\right) \mathrm{e}^{\int_{0}^{T}K(\tau )\,\mathrm{d}\tau },
\end{equation}%
where $C>0$ is a constant depending on the norms of both the initial data and
\begin{equation}
\label{contdepconst}
K(t)=C\left( 1+\Vert \partial _{t}\mathbf{u}_{2}(t)\Vert _{L^{2}(\Omega
)}^{2}+\Vert \nabla \mathbf{u}_{2}(t)\Vert _{L^{4}(\Omega )}^{4}+\Vert
\mathbf{u}_{2}(t)\Vert _{H^{2}(\Omega )}^{2}+\Vert \nabla \phi _{2}(t)\Vert
_{L^{4}(\Omega )}^{4}\right) .
\end{equation}
 \end{theorem}

\begin{remark}[Unique continuation property]
Let us consider two strong solutions according Theorem \ref{MR:strong} which
depart from the same initial data with $\phi _{0}$ not necessarily strictly
separated. If there exists $\widetilde{\tau }>0$ such that the two solutions
coincide at $t=\widetilde{\tau }$, they coincide over the entire interval $[%
\widetilde{\tau },\infty )$. In fact, since both solutions are strictly
separated in $[\widetilde{\tau },\infty )$ by \eqref{delt}, the claim
follows from \eqref{contdep}.
\end{remark}

The strategy of our proof relies on two new tools. First, we show a novel
well-posedness result of strong solutions to the nonlocal Cahn-Hilliard
system driven by an incompressible velocity field:%
\begin{align}
&\partial _{t}\phi +\mathbf{u}\cdot \nabla \phi  =\Delta \mu ,\quad \text{ }\mu
=F^{\prime }(\phi )-J\ast \phi, \quad \text{ in }\Omega \times (0,\infty),
\label{CH-intr} \\
&\partial _{\mathbf{n}}\mu  =0, \quad\text{ on }\partial \Omega \times (0,\infty),\quad\text{
}\phi (0)={\phi }_{0}, \quad\text{ in }{\Omega }.  \notag
\end{align}%
We recall that the regularity achieved in \cite[Lemma 6.1]{GGG2017} has been
obtained by assuming that
$$
\mathbf{u}\in L^{2}(0,T; L_{\sigma
}^{\infty }(\Omega ))\cap H^{1}(0,T; L_{\sigma }^{2}(\Omega )).
$$
To the best of our knowledge, no other results in the case of singular
potential and constant mobility are available. In Theorem \ref{ExistCahn},
we prove the existence and uniqueness of the strong solutions \eqref{CH-intr}
under the solely assumption that $\mathbf{u}\in L^{4}(0,T;L_{\sigma
}^{4}(\Omega ))$, which holds when $\mathbf{u}$ just belongs to the Leray-Hopf class.
In doing so, we exploit the specific form of the chemical
potential $\mu $ to rewrite the apparently unmanageable term $\int_{\Omega }%
\mathbf{u}\cdot \nabla \phi \,\partial _{t}\mu \,\mathrm{d}x$ (cf. %
\eqref{drift}). The second tool is a new interpolation estimate for the
pressure of the Stokes operator given in Lemma \ref{press} which improves
\cite[Lemma B.2]{GMT2019}. Once these two preliminary results are proven,
the strong couplings in \eqref{syst2} are handled through a
suitable approximation scheme to obtain global-in-time higher-order
Sobolev/energy estimates. Finally, we mention that it is unlikely
to study the three-dimensional case in the functional framework
considered in Theorem \ref{MR:strong}. Indeed, assuming that $\mu $ and $%
\mathbf{u}$ satisfy the properties \eqref{regg}, in three-dimensions the
nonlinear term $(\nabla \mu \cdot \nabla )\mathbf{u}$ does not even belong
to $L^{2}(0,T;L^{2}(\Omega ;\mathbb{R}^{3}))$.

%

Our second main result regards the global behavior and propagation of
regularity of any weak solution. Weak uniqueness or weak-strong uniqueness
results are not available in this context, therefore we need to exploit a
different proof to obtain that \textit{any} weak solution enjoys an
instantaneous propagation of regularity and converges to an
equilibrium, i.e., to a stationary state as time goes to $+\infty$.

\begin{theorem}
\label{MR:weak} Let the assumptions \ref{Omega}-\ref{m-ass} hold. Assume
that $\mathbf{u}_{0}\in L_{\sigma }^{2}(\Omega )$ and $\phi _{0}\in
L^{\infty }(\Omega )$ with $F(\phi _{0})\in L^{1}(\Omega )$ and $|\overline{%
\phi _{0}}|<1$. For any $T>0$, we consider the weak solution $(\mathbf{u}%
,\phi )$ to \eqref{syst2}-\eqref{bic} defined on $\Omega \times [0,T)$ given
by Theorem \ref{R:weak-sol}. Then, $(\mathbf{u},\phi )$ is uniquely extended
on $\Omega \times [0,\infty )$ and, for any $\tau >0$, $(\mathbf{u},\phi )$
is a strong solution on $[\tau ,+\infty )$. More precisely, for any $t\geq
\tau >0$, $(\mathbf{u},\phi )$ satisfies
\begin{equation}
\begin{cases}
 \mathbf{u}\in BC([\tau ,\infty );H_{0,\sigma }^{1}(\Omega ))\cap L_{%
\mathrm{uloc}}^{2}([\tau ,\infty );H_{0,\sigma }^{2}(\Omega ))\cap H_{%
\mathrm{uloc}}^{1}([\tau ,\infty );L_{\sigma }^{2}(\Omega )), \\
\Pi \in L_{\mathrm{uloc}}^{2}([\tau ,\infty );H_{(0)}^{1}(\Omega )), \\
\phi \in L^{\infty }(\tau ,\infty ;H^{1}(\Omega ))\cap L_{\mathrm{uloc}%
}^{q}([\tau ,\infty );W^{1,p}(\Omega )),\quad q=\frac{2p}{p-2},\quad p\in
(2,\infty ), \\
\phi \in L^{\infty }(\tau ,\infty ;L^{\infty }(\Omega )):\quad |\phi
(x,t)|<1 \ \text{ for a.a. }x\in\Omega ,\,\forall \,t\in \lbrack \tau ,\infty ), \\
\partial _{t}\phi \in L^{\infty }(\tau ,\infty ;H^{1}(\Omega )^{\prime
})\cap L^{2}(\tau,\infty ;L^{2}(\Omega )),\quad F^{\prime }(\phi )\in
L^{\infty }(\tau ,\infty ;H^{1}(\Omega )), \\
\mu \in BC_{\mathrm{w}}([\tau ,\infty );H^{1}(\Omega ))\cap L_{\mathrm{uloc}%
}^{2}([\tau ,\infty );H^{2}(\Omega ))\cap H_{\mathrm{uloc}}^{1}([\tau
,\infty );H^{1}(\Omega )^{\prime }),
\end{cases}
\label{reggbis}
\end{equation}%
and the energy identity
\begin{equation}
E(\mathbf{u}(t),\phi (t))+\int_{\tau }^{t}\left\Vert \sqrt{\nu (\phi (s))}D%
\mathbf{u}(s)\right\Vert _{L^{2}(\Omega )}^{2}+\Vert \nabla \mu (s)\Vert
_{L^{2}(\Omega ))}^{2}\,\mathrm{d}s=E(\mathbf{u}(\tau ),\phi (\tau ))
\label{ENERGYID}
\end{equation}%
holds, for every $0<\tau \leq t<\infty $. Moreover, there exists a constant $%
\delta =\delta (\tau )\in (0,1)$ such that
\begin{equation*}
\sup_{t\in \lbrack \tau ,+\infty )}\Vert \phi (t)\Vert _{L^{\infty }(\Omega
)}\leq 1-\delta .
\end{equation*}%
If, in addition, $F$ is also real analytic, then $(\mathbf{u}(t),\phi (t))\rightarrow (\mathbf{0},\phi _{\infty
})$ in $L_{\sigma }^{2}(\Omega )\times L^{\infty }(\Omega )$ as $%
t\rightarrow +\infty $, where $\phi _{\infty }\in L^{\infty }(\Omega )\cap
H^{1}(\Omega )$ is a solution to the stationary nonlocal Cahn-Hilliard
equation
\begin{equation}
\begin{cases}
F^{\prime }(\phi _{\infty })-J\ast \phi_\infty=\mu _{\infty },\quad \text{in }%
\Omega , \\
\frac{1}{|\Omega |}\int_{\Omega }\phi _{\infty }(x)\mathrm{d}x =\overline{%
\phi _{0}},\quad \mu _{\infty }\in \mathbb{R}.
\end{cases}
\label{Stat-CH}
\end{equation}
\end{theorem}

\begin{remark}
The convergence to a single equilibrium stated in Theorem \ref{MR:weak} also
holds for the global strong solutions constructed in Theorem \ref{MR:strong}.
\end{remark}

Let us now consider the matched densities case, i.e. $\rho=\rho_1=\rho_2$.
Both Theorems \ref{MR:strong} and \ref{MR:weak} remains true for the
nonlocal Navier-Stokes-Cahn-Hilliard (NSCH) system with unmatched
viscosities, singular potential and constant mobility (also called nonlocal Model H, see \cite{FGG, GGG2017}). Thanks to the aforementioned novel
interpolation result (see Lemma \ref{press} below),
we prove the following continuous dependence estimate for strong solutions
to the nonlocal NSCH. In particular, this guarantees the uniqueness of
strong solutions (not necessarily ``separated" as in Theorem \ref{MR:strong}%
) to the nonlocal NSCH system. This result has been an open question since
\cite{FGG} where only the constant viscosity case was considered.

\begin{theorem}
\label{Uniq:strong:H} Let the assumptions \ref{Omega}-\ref{m-ass} hold.
Suppose that $(\mathbf{u}_{1},\Pi _{1},\phi _{1})$ and $(\mathbf{u}_{2},\Pi _{2},\phi
_{2})$ are two strong solutions given by Theorem \ref{MR:strong} with
constant density $\rho =\rho _{1}=\rho _{2}>0$ corresponding to the initial
data $(\mathbf{u}_{0}^{1},\phi _{0}^{1})$ and $(\mathbf{u}_{0}^{2},\phi
_{0}^{2})$, respectively. Then, there holds
\begin{equation}
\begin{split}
& \Vert \mathbf{u}_{1}(t)-\mathbf{u}_{2}(t)\Vert _{H_{\sigma }^{-1}(\Omega
)}^{2}+\Vert \phi _{1}(t)-\phi _{2}(t)\Vert _{H^{1}(\Omega )^{\prime }}^{2}
\\
& \quad \leq C\left( \Vert \mathbf{u}_{0}^{1}-\mathbf{u}_{0}^{2}\Vert
_{H_{\sigma }^{-1}(\Omega )}^{2}+\Vert \phi _{0}^{1}-\phi _{0}^{2}\Vert
_{H^{1}(\Omega )^{\prime }}^{2}\right) \mathrm{{e}}^{S(t)}+%
R(t)\left\vert \overline{\phi }_{0}^{1}-\overline{\phi }%
_{0}^{2}\right\vert \mathrm{e}^{{S}(t)}+Ct\left\vert \overline{%
\phi }_{0}^{1}-\overline{\phi }_{0}^{2}\right\vert ^{2}\mathrm{e}^{S(t)},
\label{weakstrong1}
\end{split}
\end{equation}
for all $t>0$, where
\begin{equation*}
S(t)=C\int_{0}^{t}\left( 1+\Vert \mathbf{u}_{1}(s)\Vert
_{L^{4}(\Omega )}^{4}+\Vert \mathbf{u}_{2}(s)\Vert _{L^{4}(\Omega
)}^{4}+\Vert D\mathbf{u}_{2}(s)\Vert _{L^{4}(\Omega )}^{4}+\Vert \nabla \phi
_{1}(s)\Vert _{L^{4}(\Omega )}^{4}\right) \,\mathrm{d}s
\end{equation*}%
and
\begin{equation*}
R(t)=C\int_{0}^{t}\left( \Vert F^{\prime }(\phi _{1}(s))\Vert
_{L^{1}(\Omega )}+\Vert F^{\prime }(\phi _{2}(s))\Vert _{L^{1}(\Omega
)}\right) \,\mathrm{d}s,
\end{equation*}%
as well as some positive constant $C$ depending on the norm of the initial
data.
\end{theorem}

Finally, we estimate the difference between the strong solutions to the
nonlocal AGG model and the nonlocal Model H originating from the same initial
datum, in terms of the density values. In particular, the following result
gives a rigorous justification of the nonlocal Model H as the \emph{constant density
approximation} of the nonlocal AGG model:

\begin{theorem}
\label{stability} Let the assumptions \ref{Omega}-\ref{m-ass} hold.
Given an initial datum $(\mathbf{u}_{0},\phi _{0})$ as in Theorem \ref%
{MR:strong}, we consider a strong solution $(\mathbf{u},\Pi ,\phi )$ to the
AGG model and the strong solution $(u_{H},\Pi _{H},\phi _{H}) $ to the AGG
model with constant density $\overline{\rho }>0$ (nonlocal Model H). Then,
for any given $T>0$, there exists a constant $C>0$, that depends on the norm of
the initial data, the time $T$ and the parameters of the systems, such that
\begin{equation}
\sup_{t\in \lbrack 0,T]}\Vert \mathbf{u}(t)-\mathbf{u}_{H}(t)\Vert
_{H_{\sigma }^{-1}(\Omega )}+\sup_{t\in \lbrack 0,T]}\Vert \phi (t)-\phi
_{H}(t)\Vert _{H^{1}(\Omega )^{\prime }}\leq C\left( \left\vert \frac{\rho
_{1}-\rho _{2}}{2}\right\vert +\left\vert \frac{\rho _{1}+\rho _{2}}{2}-%
\overline{\rho }\right\vert \right) .  \label{stab}
\end{equation}
\end{theorem}

\begin{remark}
Assuming that $\rho_1=\overline{\rho}$ and $\rho_2=\overline{\rho}%
+\varepsilon$, for some (small) $\varepsilon>0$, estimate
\eqref{stab} reads
\begin{align*}
\sup_{t\in[0,T]}\Vert\mathbf{u}(t)-\mathbf{u}_H(t)\Vert_{H^{-1}_{\sigma}(%
\Omega)}+\sup_{t\in[0,T]}\Vert\phi(t)-\phi_H(t)\Vert_{H^1(\Omega)^{\prime}}
\leq C\varepsilon.
\end{align*}
\end{remark}


\section{Mathematical Setting}

\label{prel} \setcounter{equation}{0}

Let $\Omega $ be a bounded domain of class $C^{2}$ in $\mathbb{R}^{2}$. The
Sobolev spaces of functions $u:\Omega \rightarrow \mathbb{R}$ and of vector
fields $\mathbf{u}:\Omega \rightarrow \mathbb{R}^{d}$ ($d\in \mathbb{N}$)
are denoted by $W^{k,p}(\Omega )$ and $W^{k,p}(\Omega ;\mathbb{R}^{d})$,
respectively, where $k\in \mathbb{N}$ and $1\leq p\leq \infty $. For
simplicity of notation, we will denote their norms by $\Vert \cdot \Vert
_{W^{k,p}(\Omega )}$ in both cases. If $p=2$, the Hilbert spaces $%
W^{k,2}(\Omega )$ and $W^{k,2}(\Omega ;\mathbb{R}^{d})$ are denoted by $%
H^{k}(\Omega )$ and $H^{k}(\Omega ;\mathbb{R}^{d})$, respectively, with norm
$\Vert \cdot \Vert _{H^{k}(\Omega )}$.
We will adopt the notation $(\cdot ,\cdot )$ for the inner product in $%
L^{2}(\Omega )$ and in $L^{2}(\Omega ;\mathbb{R}^{d})$. The dual spaces of $%
W^{k,p}(\Omega )$ and $W^{k,p}(\Omega ;\mathbb{R}^{d})$ (as well as $%
H^{k}(\Omega )$ and $H^{k}(\Omega ;\mathbb{R}^{d})$) are denoted by $%
W^{k,p}(\Omega )^{\prime }$ and $W^{k,p}(\Omega ;\mathbb{R}^{d})^{\prime }$,
respectively, and the duality product by $\langle \cdot ,\cdot \rangle
_{W^{k,p}(\Omega )}$ and $\langle \cdot ,\cdot \rangle _{W^{k,p}(\Omega ;%
\mathbb{R}^{d})}$.
In addition, we define the linear subspaces
\begin{equation*}
L_{(0)}^{2}(\Omega )=\left\{ u\in L^{2}(\Omega ):\overline{u}=\frac{(u,1)}{%
|\Omega |}=0\right\} ,\quad H_{(0)}^{1}(\Omega )=H^{1}(\Omega )\cap
L_{(0)}^{2}(\Omega )
\end{equation*}%
and
\begin{equation*}
H_{(0)}^{-1}(\Omega )=\left\{ u\in H^{1}(\Omega )^{\prime }:\overline{u}%
=\left\vert \Omega \right\vert ^{-1}\left\langle u,1\right\rangle =0\right\}
,
\end{equation*}%
endowed with the norms of $L^{2}(\Omega )$, $H^{1}(\Omega )$ and $%
H^{1}(\Omega )^{\prime }$, respectively. By the Poincar\'{e}-Wirtinger
inequality, it follows that $\big(\Vert \nabla u\Vert _{L^{2}(\Omega )}^{2}+|%
\overline{u}|^{2}\big)^{\frac{1}{2}}$ is a norm in $H^{1}(\Omega )$, that is
equivalent to $\Vert u\Vert _{H^{1}(\Omega )}$. The Laplace operator $%
A_{0}:H_{(0)}^{1}(\Omega )\rightarrow H_{(0)}^{-1}(\Omega )$ defined by $%
\left\langle A_{0}u,v\right\rangle_{H^1_{(0)}(\Omega)} =(\nabla u,\nabla v)$%
, for any $v\in H_{(0)}^{1}(\Omega )$, is a bijective map between $%
H_{(0)}^{1}(\Omega )$ and $H_{(0)}^{-1}(\Omega )$. We denote its inverse by $%
\mathcal{N}=A_{0}^{-1}:H_{(0)}^{-1}(\Omega )\rightarrow H_{(0)}^{1}(\Omega )$%
, namely for any $u\in H_{(0)}^{-1}(\Omega )$, $\mathcal{N}u$ is the unique
function in $H_{(0)}^{1}(\Omega )$ such that $(\nabla \mathcal{N}u,\nabla
v)=\langle u,v\rangle _{H_{(0)}^{1}(\Omega )}$ for any $v\in
H_{(0)}^{1}(\Omega )$. As a consequence, for any $u\in H_{(0)}^{-1}(\Omega )$%
, we set $\Vert u\Vert _{\ast }:=\Vert \nabla \mathcal{N}u\Vert
_{L^{2}(\Omega )}$, which is a norm in $H_{(0)}^{-1}(\Omega )$, that is
equivalent to the canonical dual norm. In turn, $\big(\Vert u-\overline{u}%
\Vert _{\ast }^{2}+|\overline{u}|^{2}\big)^{\frac{1}{2}}$ is a norm $%
H^{1}(\Omega )^{\prime }$, that is equivalent to $\Vert u\Vert
_{H^{1}(\Omega )^{\prime }}$. Moreover, by the regularity theory for the
Laplace operator with homogeneous Neumann boundary conditions, there is a constant $C>0$
such that
\begin{equation}
\Vert \nabla \mathcal{N}u\Vert _{H^{1}(\Omega )}\leq C\Vert u\Vert
_{L^{2}(\Omega )},\quad \forall \,u\in L_{(0)}^{2}(\Omega ).  \label{H_2}
\end{equation}%
Lastly, we report the following Gagliardo-Nirenberg and Agmon inequalities
\begin{align}
& \Vert u\Vert _{L^{4}(\Omega )}\leq C\Vert u\Vert _{L^{2}(\Omega )}^{\frac{1%
}{2}}\Vert u\Vert _{H^{1}(\Omega )}^{\frac{1}{2}},\quad & & \forall \,u\in
H^{1}(\Omega ),  \label{LADY} \\
& \Vert u\Vert _{L^{\infty }(\Omega )}\leq C\Vert u\Vert _{L^{2}(\Omega )}^{%
\frac{1}{2}}\Vert u\Vert _{H^{2}(\Omega )}^{\frac{1}{2}},\quad & & \forall
\,u\in H^{2}(\Omega ),  \label{W14bis} \\
& \Vert u\Vert _{W^{1,4}(\Omega )}\leq C\Vert u\Vert _{L^{\infty }(\Omega
)}^{\frac{1}{2}}\Vert u\Vert _{H^{2}(\Omega )}^{\frac{1}{2}},\quad & &
\forall \,u\in H^{2}(\Omega ),  \label{W14tris}
\end{align}
for some suitable constants $C>0$ depending only on $\Omega$.

Next, we introduce the solenoidal function spaces
\begin{align*}
& L_{\sigma }^{p}(\Omega )=\{\mathbf{u}\in L^{p}(\Omega ;\mathbb{R}^{2}):%
\mathrm{div}\,\mathbf{u}=0\ \text{in }\Omega ,\,\mathbf{u}\cdot \mathbf{n}%
=0\ \text{on}\ \partial \Omega \},\quad  \forall\, p \in (1,\infty)  , \\
& W_{0,\sigma }^{1,p}(\Omega )=\{\mathbf{u}\in W^{1,p}(\Omega ;\mathbb{R}%
^{2}):\mathrm{div}\,\mathbf{u}=0\ \text{in }\Omega ,\,\mathbf{u}=\mathbf{0}\
\text{on}\ \partial \Omega \},\quad \forall\, p \in (1,\infty) .
\end{align*}%
We recall that $L_{\sigma }^{p}(\Omega )$ and $W_{0,\sigma }^{1,p}(\Omega )$
corresponds to the completion of $C_{0,\sigma }^{\infty }(\Omega ;\mathbb{R}%
^{2})$, namely the space of divergence-free vector fields in $C_{0}^{\infty
}(\Omega ;\mathbb{R}^{2})$, in the norm of $L^{p}(\Omega ;\mathbb{R}^{2})$
and $W^{1,p}(\Omega ;\mathbb{R}^{2})$, respectively. For simplicity of
notation, we will also use $\Vert \cdot \Vert _{L^{p}(\Omega )}$ and $\Vert
\cdot \Vert _{W^{1,p}(\Omega )}$ to denote the norms in $L_{\sigma
}^{p}(\Omega )$ and $W_{0,\sigma }^{1,p}(\Omega )$, respectively.
The space $H_{0,\sigma }^{1}(\Omega )=W_{0,\sigma }^{1,2}(\Omega )$ is
endowed with the inner product and norm $(\mathbf{u},\mathbf{v}%
)_{H_{0,\sigma }^{1}(\Omega )}=(\nabla \mathbf{u},\nabla \mathbf{v}) $ and $%
\Vert \mathbf{u}\Vert _{H_{0,\sigma }^{1}(\Omega )}=\Vert \nabla \mathbf{u}%
\Vert _{L^{2}(\Omega )}$, respectively. By the Korn inequality,
\begin{equation}
\Vert \nabla \mathbf{u}\Vert _{L^{2}(\Omega )}\leq \sqrt{2}\Vert D\mathbf{u}%
\Vert _{L^{2}(\Omega )}\leq \sqrt{2}\Vert \nabla \mathbf{u}\Vert
_{L^{2}(\Omega )},\quad \forall \,\mathbf{u}\in H_{0,\sigma }^{1}(\Omega ).
\label{korn}
\end{equation}%
%
%
%
%
%
%
%
%
%
%
%
%
%
%
%
%
%
%
%
%
%
%
%
%
%
%
%
%
%
%
%
%
%
%
%
%
%
%
%
%
%
%
%
%
%
%
%
%
%
The Stokes operator is $\mathbf{A}=-\mathbb{P}\Delta $, with domain $D(%
\mathbf{A})=H_{0,\sigma }^{2}(\Omega )$, where $H_{0,\sigma }^{2}(\Omega
)=H^{2}(\Omega ;\mathbb{R}^{2})\cap H_{0,\sigma }^{1}(\Omega )$ and $\mathbb{%
P}$ is the Leray orthogonal projector from $L^{2}(\Omega )$ onto $L_{\sigma
}^{2}(\Omega )$. We denote by $\mathbf{A}^{-1}$ the inverse map of the
Stokes operator. In particular, $\mathbf{A}^{-1}$ is an isomorphism between $%
H_{\sigma }^{-1}(\Omega )=H_{0,\sigma }^{1}(\Omega )^{\prime }$ and $%
H_{0,\sigma }^{1}(\Omega )$ such that, for any $\mathbf{u}\in H_{\sigma
}^{-1}(\Omega )$, $\mathbf{A}^{-1}\mathbf{u}$ is the unique function in $%
H_{0,\sigma }^{1}(\Omega )$ such that $(\nabla \mathbf{A}^{-1}\mathbf{u}%
,\nabla \mathbf{v})=\langle \mathbf{u},\mathbf{v}\rangle _{H_{0,\sigma
}^{1}(\Omega )}$ for any $\mathbf{v}\in H_{0,\sigma }^{1}(\Omega )$. Then,
it follows that $\Vert \mathbf{u}\Vert _{\sharp }:=\Vert \nabla \mathbf{A}%
^{-1}\mathbf{u}\Vert _{L^{2}(\Omega )}$ is an equivalent norm on $H_{\sigma
}^{-1}(\Omega )$. By the regularity theory of the Stokes operator, there
exists a constant $C$ such that
\begin{equation}
\Vert \mathbf{u}\Vert _{H^{2}(\Omega )}\leq C\Vert \mathbf{A}\mathbf{u}\Vert
_{L^{2}(\Omega )},\quad \forall \,\mathbf{u}\in H_{0,\sigma }^{2}(\Omega ).
\label{stoke}
\end{equation}%
Then, we define $\Vert \mathbf{u}\Vert _{H_{0,\sigma }^{2}(\Omega )}=\Vert
\mathbf{A}\mathbf{u}\Vert _{L^{2}(\Omega )}$, which is a norm in $%
H_{0,\sigma }^{2}(\Omega )$, that is equivalent to $\Vert \mathbf{u}\Vert
_{H^{2}(\Omega )}$.

Let $X$ be a real Banach space and consider an interval $I\subseteq[0,\infty)$. The
Banach space $BC(I;X)$ consists of all bounded and continuous $f: I\to X$
equipped with the supremum norm. The subspace $BUC(I;X)$ denotes the set of
all bounded and uniformly continuous functions $f: I\to X$. We denote by $%
BC_{\mathrm{w}}(I;X)$ the topological vector space of all bounded and weakly
continuous functions $f: I\to X$. If $I$ is a compact interval, then we simply use the notation $C(I; X)$ or $C_{\mathrm{w}}(I;X)$. 
The set $C^\infty_0(I;X)$ denotes the
vector space of all smooth functions $f: I\to X$ whose support is compactly
embedded in $I$. Given $p\in[1,\infty]$, the Lebesgue space $L^p(I;X)$
denotes the set of all strongly measurable $f: I\to X$ that are $p$%
-integrable/essentially bounded. In particular, $L^p_{\mathrm{uloc}%
}([0,\infty);X)$ is the set of all strongly measurable $f: [0,\infty) \to X$
such that
\begin{equation*}
\|f\|_{L^p_{\mathrm{uloc}}([0,\infty); X)}= \sup_{t\geq
0}\|f\|_{L^q(t,t+1;X)} <\infty.
\end{equation*}
The Bochner spaces $W^{1,p}(I;X)$ consists of all $f\in L^p(I;X)$ with $%
\partial_t f \in L^p(I;X)$. If $p=2$, then $H^1(I;X)=W^{1,2}(I;X)$. In a
similar way, we also define $H^1_{\mathrm{uloc}}([0,\infty);X)$.

In the following sections, we will denote by $C$ a generic positive constant, which may even vary within the same line, possibly depending on $\Omega$ as well as on the parameters of the system.



\section{Interpolation estimate for the $L^4$-norm of the pressure in the
Stokes problem}

\label{tools} \setcounter{equation}{0}

The regularity theory of the Stokes operator ensures that, for any $\mathbf{f%
}\in L_{\sigma }^{2}(\Omega )\subset L^{2}(\Omega ;\mathbb{R}^{2})$, there
exist $\mathbf{u}=\mathbf{A}^{-1}\mathbf{f}\in H_{0,\sigma }^{2}(\Omega )$
and $P\in H_{(0)}^{1}(\Omega )$ such that
\begin{equation}
-\Delta \mathbf{u}+\nabla P=\mathbf{f},\quad \text{a.e. in }\Omega .
\label{st}
\end{equation}%
We refer the reader to \cite{Galdi} and the references therein for the
comprehensive theory. The $L^{p}$-norms of the pressure $P$ are usually
controlled by the norms of negative Sobolev spaces of the external force $%
\mathbf{f}$ (see, for instance, \cite[Lemma IV.2.1]{Galdi}). Notwithstanding
these results are sharp from the viewpoint of the regularity theory of
steady Stokes flows, an estimate of the $L^{p}$-norms of $P$ in terms of the
norms in $H_{\sigma }^{-1}(\Omega )=(H_{0,\sigma }^{1}(\Omega ))^{\prime }$
and in $L_{\sigma }^{2}(\Omega )$ of $\mathbf{f}$ is more effective for some
purposes in the context of evolutionary Navier-Stokes flows. A first
interpolation result for the $L^{2}$-norm of the pressure $P$ was
established in \cite[Lemma B.2]{GMT2019}. We present a novel interpolation
result for the $L^{4}$-norm of the pressure $P$, which improves the one in
\cite[Lemma B.2]{GMT2019}. This is essential to perform some crucial
estimates in the sequel in order to deal with the low regularity
guaranteed by the nonlocal setting.

\begin{lemma}
\label{press} Let $\Omega$ be a bounded domain in $\mathbb{R}^2$ of class $%
C^2$. There exists $C$ such that
\begin{align}
\| P \|_{L^4(\Omega)} \leq C\| \nabla \mathbf{A}^{-1}\mathbf{f}
\|_{L^2(\Omega)}^{\frac12} \| \mathbf{f} \|_{L^2(\Omega)}^{\frac12}, \quad
\forall \, \mathbf{f} \in L^2_\sigma(\Omega),
\end{align}
where $P$ is the pressure in \eqref{st}.
\end{lemma}

\begin{proof}
We know from \cite[Lemma IV.2.1]{Galdi}) that there exists $C$ such that
\begin{align}  \label{pp}
\Vert P \Vert_{L^4(\Omega)}\leq C \Vert \mathbf{f} \Vert_{W^{-1,4}(\Omega)},
\end{align}
where $W^{-1,4}(\Omega;\mathbb{R}^2)= (W_0^{1,\frac43}(\Omega;\mathbb{R}%
^2))^{\prime}$ and $W_0^{1,\frac43}(\Omega;\mathbb{R}^2)$ is the completion
of $C_0^{\infty}(\Omega;\mathbb{R}^2)$ with respect to the norm of $%
W^{1,\frac43}(\Omega;\mathbb{R}^2)$. In order to estimate the right-hand
side in \eqref{pp}, let us consider $\mathbf{v}\in W^{1,\frac43}_0(\Omega;%
\mathbb{R}^2)$. By using the integration by parts and the properties of $\mathbb{%
P}$, we have
\begin{align}  \label{fv}
(\mathbf{f},\mathbf{v}) =(\mathbb{P} (-\Delta) \mathbf{A}^{-1} \mathbf{f} ,%
\mathbf{v}) =((-\Delta) \mathbf{A}^{-1}\mathbf{f},\mathbb{P}\mathbf{v})
=(\nabla \mathbf{A}^{-1} \mathbf{f}, \nabla \mathbb{P}\mathbf{v})
-\int_{\partial\Omega} \left(\nabla\mathbf{A}^{-1} \mathbf{f} \, \mathbf{n}
\right) \cdot \mathbb{P} \mathbf{v} \, \mathrm{d} \sigma.
\end{align}
Since $\mathbf{f} \in L^2_\sigma(\Omega)$, we observe from \eqref{stoke}
that $\| \mathbf{A}^{-1} \mathbf{f}\|_{H^2(\Omega)}\leq C \| \mathbf{f}%
\|_{L^2(\Omega)}$. By using this fact, together with \eqref{LADY} and the
continuity of the projection operator $\mathbb{P}$ from $W_0^{1,\frac43}(%
\Omega;\mathbb{R}^2)$ onto $W^{1,\frac43}(\Omega;\mathbb{R}^2) \cap
L^2_\sigma(\Omega)$, we easily infer that
\begin{align*}
(\nabla \mathbf{A}^{-1} \mathbf{f},\nabla\mathbb{P} \mathbf{v}) &\leq \|
\nabla \mathbf{A}^{-1} \mathbf{f} \|_{L^4(\Omega)} \| \mathbb{P} \mathbf{v}
\|_{W^{1,\frac43}(\Omega)} \\
&\leq C\| \nabla\mathbf{A}^{-1} \mathbf{f} \|_{L^2(\Omega)}^{\frac12} \|
\mathbf{f} \|_{L^2(\Omega)}^{\frac12} \| \mathbf{v}\|_{W_0^{1,\frac43}(%
\Omega)}.
\end{align*}
Next, by the interpolation trace estimate, there exists $C$ such
that
\begin{equation}  \label{Trace1}
\Vert f \Vert_{L^2(\partial\Omega)} \leq C\Vert f
\Vert_{L^2(\Omega)}^{\frac12} \Vert f \Vert_{H^1(\Omega)}^{\frac12}, \quad
\forall \, f \in H^1(\Omega).
\end{equation}
Furthermore, we also report the following trace estimate (see \cite[Thm
II.4.1]{Galdi}, with $n=2$, $m=1$, $q=4/3$ and $r=2$): there exists $C$ such that
\begin{equation}  \label{Trace2}
\Vert f \Vert_{L^2(\partial\Omega)} \leq C\Vert f
\Vert_{W^{1,\frac43}(\Omega)}, \quad \forall \, f \in W^{1,\frac43}(\Omega).
\end{equation}
Thus, exploiting \eqref{Trace1} and \eqref{Trace2}, together with %
\eqref{stoke}, we obtain
\begin{align*}
\left| \int_{\partial\Omega} \left( \nabla \mathbf{A}^{-1} \mathbf{f} \,
\mathbf{n}\right) \cdot \mathbb{P} \mathbf{v} \, \mathrm{d}\sigma \right|
&\leq C \Vert \nabla\mathbf{A}^{-1} \mathbf{f} \Vert_{L^2(\partial\Omega)}
\Vert\mathbb{P} \mathbf{v} \Vert_{L^2(\partial\Omega)} \\
&\leq C\Vert \nabla \mathbf{A}^{-1} \mathbf{f} \Vert_{L^2(\Omega)}^{\frac12}
\Vert \nabla \mathbf{A}^{-1} \mathbf{f} \Vert_{H^1(\Omega)}^{\frac12} \Vert
\mathbb{P} \mathbf{v} \Vert_{W^{1,\frac43}(\Omega)} \\
&\leq C\Vert\nabla \mathbf{A}^{-1} \mathbf{f} \Vert_{L^2(\Omega)}^{\frac12}
\Vert \mathbf{f} \Vert_{L^2(\Omega)}^{\frac12} \Vert\mathbf{v}%
\Vert_{W_0^{1,\frac43}(\Omega)}.
\end{align*}
Therefore, we deduce that
\begin{align}  \label{fondd}
\Vert \mathbf{f} \Vert_{W^{-1,4}(\Omega)} =\sup_{0\neq \mathbf{v} \in
W_0^{1,\frac43}(\Omega; \mathbb{R}^2)} \frac{\left| \left( \mathbf{f},
\mathbf{v} \right) \right|}{\| \mathbf{v} \|_{W_0^{1,\frac43}(\Omega)} }
\leq C\Vert\nabla\mathbf{A}^{-1} \mathbf{f} \Vert_{L^2(\Omega)}^{\frac12}
\Vert\mathbf{f} \Vert_{L^2(\Omega)}^{\frac12},
\end{align}
which entails the desired conclusion in light of \eqref{pp}.
\end{proof}

\begin{remark}
A similar result can be obtained in the three-dimensional case. Replacing
the $L^4$-norm with the $L^3$-norm of $P$ and exploiting the corresponding
Gagliardo-Nirenberg and trace estimates in dimension three, one can repeat
word by word the arguments of the above proof to obtain
\begin{equation*}
\| P \|_{L^3(\Omega)} \leq C \| \nabla\mathbf{A}^{-1} \mathbf{f}
\|_{L^2(\Omega)}^{\frac12} \| \mathbf{f} \|_{L^2(\Omega)}^{\frac12}, \quad \forall
\, \mathbf{f} \in L^2_\sigma(\Omega).
\end{equation*}
\end{remark}

\section{The nonlocal Cahn-Hilliard equation with divergence-free drift}

\label{convective} \setcounter{equation}{0}

Let $\mathbf{u}$ be a divergence-free vector field. We consider the
initial-boundary value problem for the nonlocal Cahn-Hilliard equation with
divergence-free drift
\begin{equation}  \label{nCH}
\begin{cases}
\partial_t\phi+\mathbf{u}\cdot \nabla\phi=\Delta\mu,\quad \mu=F^\prime(\phi)-J\ast
\phi,\quad & \text{in } \Omega\times(0,T), \\
\partial_\textbf{n}\mu=0,\quad & \text{on }\partial\Omega\times(0,T), \\
\phi(\cdot, 0)=\phi_0,\quad & \text{in }\Omega.%
\end{cases}%
\end{equation}
We present herein a novel well-posedness and regularity result under minimal
assumptions on the velocity field for the system \eqref{nCH} (cf. with the
analysis in \cite{GGG2017}). Beyond its own interest \textit{per se}, this
will play an essential role in the proof of Theorem \ref{MR:strong}.

\begin{theorem}
\label{ExistCahn} Let the assumptions \ref{Omega}-\ref{m-ass} hold and
let $T>0$. Assume that $\mathbf{u}\in L^{4}(0,T;L_{\sigma }^{4}(\Omega ))$
and $\phi _{0}\in L^{\infty }(\Omega )$ with $\Vert \phi _{0}\Vert
_{L^{\infty }(\Omega )}\leq 1$ and $|\overline{\phi _{0}}|<1$. Then, there
exists a unique weak solution to \eqref{nCH} such that
\begin{equation}
\begin{cases}
\phi \in C([0,T];L^{2}(\Omega ))\cap L^{2}(0,T;H^{1}(\Omega ))\cap
H^{1}(0,T;H^{1}(\Omega )^{\prime }), \\
\phi \in L^{\infty }(\Omega \times (0,T)):\quad |\phi|<1 \ \text{a.e.
in }\Omega \times (0,T), \\
\mu \in L^{2}(0,T;H^{1}(\Omega )),\quad F^{\prime }(\phi )\in
L^{2}(0,T;H^{1}(\Omega )),
\end{cases}
\label{0}
\end{equation}%
which satisfies%
\begin{equation}
\begin{split}
\left\langle \partial _{t}\phi ,v\right\rangle_{H^1(\Omega)} -(\phi \,%
\mathbf{u},\nabla v)+(\nabla \mu ,\nabla v)=0, \quad & \forall \,v\in H^{1}(\Omega
),\ \text{a.e. in }(0,T), \\
\mu =F^{\prime }(\phi )-J\ast \phi , \quad & \text{a.e. in }\Omega \times (0,T),%
\end{split}
\label{weak-nCH-2}
\end{equation}%
and $\phi (\cdot ,0)=\phi _{0}$ almost everywhere in $\Omega $. The weak
solution fulfills the energy identity
\begin{equation}
\mathcal{E}_{\mathrm{nloc}}(\phi (t))+\int_{0}^{t}\Vert \nabla \mu (\tau
)\Vert _{L^{2}(\Omega )}^{2}\,\mathrm{d}\tau +\int_{0}^{t}\int_{\Omega }\phi
\,\mathbf{u}\cdot \nabla \mu \,\mathrm{d}x\,\mathrm{d}\tau =\mathcal{E}_{%
\mathrm{nloc}}(\phi _{0}),\quad \forall \,t\in \lbrack 0,T].  \label{EE-nCH}
\end{equation}%
In addition, given two weak solutions $\phi ^{1}$ and $\phi ^{2}$
corresponding to the initial data $\phi _{0}^{1}$ and $\phi _{0}^{2}$, respectively, we
have
\begin{equation}
\begin{split}
& \left\Vert \phi ^{1}-\phi ^{2}\right\Vert _{C([0,T];H^{1}(\Omega )^{\prime })} \\
& \quad \leq \left( \left\Vert \phi _{0}^{1}-\phi _{0}^{2}\right\Vert _{H^{1}(\Omega
)^{\prime }}+\left\vert \overline{\phi _{0}^{1}}-\overline{\phi _{0}^{2}}%
\right\vert ^{\frac{1}{2}}\Vert \Lambda \Vert _{L^{1}(0,T)}^{\frac{1}{2}%
}+CT^{\frac{1}{2}}\left\vert \overline{\phi _{0}^{1}}-\overline{\phi _{0}^{2}%
}\right\vert \right) \mathrm{exp} \left(C\left( T+\Vert \mathbf{u}\Vert
_{L^{4}(0,T;L^{4}(\Omega ))}^{4}\right) \right),
\end{split}
\label{cd}
\end{equation}%
for all $t\in \lbrack 0,T]$, where $\Lambda =2\left\Vert F^{\prime }(\phi
^{1})\right\Vert _{L^{1}(\Omega )}+2\left\Vert F^{\prime }\left(\phi ^{2}\right)\right\Vert
_{L^{1}(\Omega )}$ and $C$ only depends on $\alpha $, $J$ and $\Omega $.
\medskip

\noindent Furthermore, the following regularity results hold:

\begin{itemize}
\item[(i)] \label{i1} If additionally $\phi _{0}\in H^{1}(\Omega )$ such
that $F^{\prime }(\phi _{0})\in L^{2}(\Omega )$ and $F^{\prime \prime }(\phi
_{0})\nabla \phi _{0}\in L^{2}(\Omega ;\mathbb{R}^{2})$, then
\begin{equation}
\begin{cases}
\phi \in L^{\infty }(0,T;L^{\infty }(\Omega )):\quad |\phi (x,t)|<1\ \text{ 
for a.a. } x\in\Omega ,\,\forall \,t\in \lbrack 0,T], \\
\phi \in L^{\infty }(0,T;H^{1}(\Omega ))\cap L^{q}(0,T;W^{1,p}(\Omega
)),\quad q=\frac{2p}{p-2},\quad \forall \,p\in (2,\infty ), \\
\partial _{t}\phi \in L^{4}(0,T;H^{1}(\Omega )^{\prime })\cap
L^{2}(0,T;L^{2}(\Omega )), \\
\mu \in L^{\infty }(0,T;H^{1}(\Omega ))\cap L^{2}(0,T;H^{2}(\Omega ))\cap
H^{1}(0,T;H^{1}(\Omega )^{\prime }), \\
F^{\prime }(\phi )\in L^{\infty }(0,T;H^{1}(\Omega )),\quad F^{\prime
\prime }(\phi )\in L^{\infty }(0,T;L^{p}(\Omega )),\quad \forall \,p\in
\lbrack 2,\infty ).
\end{cases}
\label{3}
\end{equation}%
In particular, we have the estimates
\begin{equation}
\begin{split}
& \Vert \nabla \mu \Vert _{L^{\infty }(0,T;L^{2}(\Omega ))} \\
& \leq \left( \left\Vert F^{\prime \prime }(\phi _{0})\nabla \phi _{0}-\nabla
J\ast \phi _{0}\right\Vert _{L^{2}(\Omega )}+C\left( \int_{0}^{T}\Vert \mathbf{u}%
(\tau )\Vert _{L^{2}(\Omega )}^{2}+\Vert \nabla \mu (\tau )\Vert
_{L^{2}(\Omega )}^{2}\,\mathrm{d}\tau \right) ^{\frac{1}{2}}\right) \\
& \quad \times \mathrm{exp} \left(C\int_{0}^{T}\Vert \mathbf{u}(\tau )\Vert
_{L^{4}(\Omega )}^{4}\,\mathrm{d}\tau \right)=:\Xi _{1},
\end{split}
\label{ST1}
\end{equation}%
\begin{equation}
\begin{split}
& \int_{0}^{T}\Vert \partial _{t}\phi (\tau )\Vert _{L^{2}(\Omega
)}^{2}+\Vert \nabla \mu (\tau )\Vert _{H^{1}(\Omega )}^{2}\,\mathrm{d}\tau \\
& \quad \leq C\left( \left\Vert F^{\prime \prime }(\phi _{0})\nabla \phi
_{0}-\nabla J\ast \phi _{0}\right\Vert _{L^{2}(\Omega )}^{2}+C\int_{0}^{T}\Vert
\mathbf{u}(\tau )\Vert _{L^{2}(\Omega )}^{2}+\Vert \nabla \mu (\tau )\Vert
_{L^{2}(\Omega )}^{2}\,\mathrm{d}\tau \right) \\
& \qquad \times \left( 1+\int_{0}^{T}\Vert \mathbf{u}(\tau )\Vert
_{L^{4}(\Omega )}^{4}\,\mathrm{d}\tau \right) \mathrm{exp}\left( C\int_{0}^{t}\Vert
\mathbf{u}(\tau )\Vert _{L^{4}(\Omega )}^{4}\,\mathrm{d}\tau \right)=:\Xi _{2},
\end{split}
\label{ST2}
\end{equation}%
and the bounds
\begin{equation}
\begin{cases}
\Vert \mu \Vert _{L^{\infty }(0,T;H^{1}(\Omega ))}+\Vert \phi \Vert
_{L^{\infty }(0,T;H^{1}(\Omega ))}+\Vert F^{\prime }(\phi )\Vert _{L^{\infty
}(0,T;H^{1}(\Omega ))}\leq \mathcal{Q}\left( \overline{\phi _{0}},\Xi
_{1},\alpha ,\Vert J\Vert _{W^{1,1}(\mathbb{R}^{2})},\Omega \right), \\
\Vert F^{\prime \prime }(\phi )\Vert _{L^{\infty }(0,T;L^{p}(\Omega
))}\leq \mathcal{Q}\left( p,\overline{\phi _{0}},\Xi _{1},\alpha ,\Vert
J\Vert _{W^{1,1}(\mathbb{R}^{2})},\Omega \right) ,\quad \forall \,p\in
\lbrack 2,\infty ), \\
\Vert \phi \Vert _{L^{q}(0,T;L^{p}(\Omega ))}\leq \mathcal{Q}\left(
\overline{\phi _{0}},\Xi _{1},\Xi _{2},\alpha ,\Vert J\Vert _{W^{1,1}(%
\mathbb{R}^{2})},\Omega ,T\right) ,\quad q=\frac{2p}{p-2},\quad \forall
\,p\in (2,\infty ), \\
\Vert \mu \Vert _{L^{2}(0,T;H^{2}(\Omega ))}+\Vert \partial _{t}\mu \Vert
_{L^{2}(0,T;H^{1}(\Omega )^{\prime })}\leq \mathcal{Q}\left( \overline{\phi
_{0}},\Xi _{1},\Xi _{2},\alpha ,\Vert J\Vert _{W^{1,1}(\mathbb{R}%
^{2})},\Omega ,T\right) ,
\end{cases}
\label{ST3}
\end{equation}%
where $C$ only depends on $\alpha $, $J$ and $\Omega $ and $%
\mathcal{Q}$ is a generic increasing and continuous function of its
arguments. Moreover, if $\mathbf{u}\in L^{\infty }(0,T;L_{\sigma
}^{2}(\Omega ))$, we also have $\partial _{t}\phi \in L^{\infty
}(0,T;H^{1}(\Omega )^{\prime })$. \smallskip

\item[(ii)] \label{i2} Let the assumptions of (i) hold. Suppose also $\Vert
\phi _{0}\Vert _{L^{\infty }(\Omega )}\leq 1-\delta _{0},$ for some $\delta
_{0}\in (0,1)$. Then, there exists $\delta >0$ such that
\begin{equation}
\sup_{t\in \lbrack 0,T]}\Vert \phi (t)\Vert _{L^{\infty }(\Omega )}\leq
1-\delta .  \label{SP}
\end{equation}%
As a consequence, $\partial _{t}\mu \in L^{2}(0,T;L^{2}(\Omega ))$ and $\mu
\in C([0,T];H^{1}(\Omega ))$.
\end{itemize}
\end{theorem}


\begin{remark}
The existence of (at least) one weak solution to \eqref{nCH} satisfying %
\eqref{0}, \eqref{weak-nCH-2} as well as \eqref{EE-nCH} holds under the
milder regularity $\mathbf{u}\in L^{2}(0,T;L_{\sigma }^{2}(\Omega )) $. We
refer the reader to the proof of Theorem \ref{ExistCahn} (see below).
\end{remark}

\begin{remark}
\label{R3}
In the case (i) above, the assumption $\phi _{0}\in H^{1}(\Omega )$ can be
relaxed by only requiring that $\nabla \phi _{0}\in L_{\text{loc}%
}^{1}(\Omega )$. Indeed, it implies that $\nabla \phi _{0}$ is measurable.
Then, $\Vert \nabla \phi _{0}\Vert _{L^{2}(\Omega )}<\infty $ follows from $%
F^{\prime \prime }(\phi _{0})\nabla \phi _{0}\in L^{2}(\Omega ;\mathbb{R}%
^{2})$ due to the strict convexity of $F$. Besides, our set of assumptions
in the case (i) entails that $F^{\prime }(\phi _{0})\in H^{1}(\Omega )$
(similarly to \cite[Theorem 4.1]{DPGG2018}). In fact, we first observe that
the chain rule $\nabla F^{\prime }(\phi _{0})=F^{\prime \prime }(\phi
_{0})\nabla \phi _{0}$ holds almost everywhere in $\Omega $. More precisely,
by exploiting the approximation $\phi _{0}^{k}$ of the initial datum
provided in the proof of Theorem \ref{ExistCahn}, it can be shown (cf. also
\cite[Lemma 3.2]{He}) that
\begin{align*}
\int_{\Omega }F^{\prime }(\phi _{0})\,\partial _{i}\varphi \mathrm{d}x&
=\lim_{k\rightarrow \infty }\int_{\Omega }F^{\prime }(\phi
_{0}^{k})\,\partial _{i}\varphi \, \mathrm{d}x =-\lim_{k\rightarrow \infty
}\int_{\Omega }F^{\prime \prime }(\phi _{0}^{k})\partial _{i}\phi
_{0}^{k}\,\varphi \, \mathrm{d}x =-\lim_{k\rightarrow \infty }\int_{\Omega
}F^{\prime \prime }(\phi _{0})\partial _{i}\phi _{0}\varphi \, \mathrm{d}x
\end{align*}%
for any $i=1,2$ and $\varphi \in C_{0}^{\infty }(\Omega )$. Then, owing to
this, we immediately infer that $F^{\prime }(\phi _{0})\in H^{1}(\Omega )$.
On the other hand, by the previous reasoning together with the Fatou lemma,
it is possible to show that $\phi _{0}\in H^{1}(\Omega )$ with $F^{\prime
}(\phi _{0})\in H^{1}(\Omega )$ guarantees that $F^{\prime \prime }(\phi
_{0})\nabla \phi _{0}\in L^{2}(\Omega ;\mathbb{R}^{2})$.
\end{remark}


\begin{remark}
\label{Rem-SP} In the case (i), the separation property holds for positive
times. More precisely, for any $0<\tau \leq T$ there exists $\delta =\delta
(\tau )\in (0,1)$ such that it holds
\begin{equation}
\sup_{t\in \lbrack \tau ,T]}\Vert \phi (t)\Vert _{L^{\infty }(\Omega )}\leq
1-\delta .  \label{del-SP}
\end{equation}%
Although we only have $\phi _{0}\in H^{1}(\Omega )$, thereby $\phi _{0}$ is
not strictly separated. Nevertheless, since $\mu \in L^{2}(0,T;H^{2}(\Omega
))$ and $J\ast \phi \in L^{\infty }(\Omega \times (0,T))$, we observe that $%
F^{\prime }(\phi )\in L^{2}(0,T;L^{\infty }(\Omega ))$. In turn, this
implies for any $\tau >0$ there exists $\tau ^{\ast }\in (0,\tau )$ such
that $F^{\prime }(\phi (\tau ^{\star }))\in L^{\infty }(\Omega )$, which
gives $\Vert \phi _{0}\Vert _{L^{\infty }(\Omega )}<1$. Thus, \eqref{del-SP}
also follows from (ii).
\end{remark}

\begin{remark}
In light of the assumptions in Theorem \ref{ExistCahn}, while in the case
(i), the assumption $\Vert \phi _{0}\Vert _{L^{\infty }(\Omega )}<1$ in \ref%
{i2} is ensured if (additionally) $\phi _{0}\in W^{1,p}(\Omega )$ for some $%
p>2$. In fact, by Remark \ref{R3}, $F^{\prime }(\phi _{0})\in H^{1}(\Omega )$%
. Then, the Trudinger-Moser inequality and the assumption \ref{h3}
(exactly as in \cite[Theorem 5.2]{GGG2017}) ensure that $F^{\prime \prime
}(\phi _{0})\in L^{r}(\Omega )$ for every $r\in \lbrack 2,\infty )$. Thus,
by the chain rule in Remark \ref{R3}, we conclude
\begin{equation*}
\left\Vert \nabla F^{\prime }(\phi _{0})\right\Vert _{L^{s}(\Omega )}\leq 
\left\Vert
F^{\prime \prime }(\phi _{0}) \right\Vert _{L^{\frac{sp}{p-s}}(\Omega )}
\left\Vert \nabla \phi _{0}\right\Vert _{L^{p}(\Omega )}<\infty ,\quad \text{where }s\in
(2,p).
\end{equation*}%
Since $s>2$, $F^{\prime }(\phi _{0})\in W^{1,s}(\Omega )\hookrightarrow
L^{\infty }(\Omega )$ implies that the initial datum $\phi _{0}$ is strictly
separated from $\pm 1$. 
\end{remark}

\begin{proof}[Proof of Theorem \protect\ref{ExistCahn}]
The proof is divided into several steps. \medskip

\noindent \textbf{Uniqueness and continuous dependence estimate.}
Let us consider two weak solutions $\phi ^{1}$ and $\phi ^{2}$
satisfying \eqref{0} and \eqref{weak-nCH-2}, and originating from two
initial data $\phi _{0}^{1}$ and $\phi _{0}^{2}$ (where possibly $\overline{%
\phi _{0}^{1}}\neq \overline{\phi _{0}^{2}}$). Setting $\phi =\phi ^{1}-\phi
^{2}$ and $\mu =F^{\prime }(\phi ^{1})-F^{\prime }(\phi ^{2})-J\ast \phi $,
it is easy to realize that
\begin{equation}
\langle \partial _{t}\phi ,v\rangle _{H^{1}(\Omega )}-(\phi \,\mathbf{u}%
,\nabla v)+(\nabla \mu ,\nabla v)=0,\quad \forall \, v \in H^1(\Omega),\ \text{a.e. in }%
(0,T).  \label{weaky}
\end{equation}%
Taking $v=\mathcal{N}(\phi -\overline{\phi })$, we find
\begin{equation*}
\frac{1}{2}\frac{\mathrm{d}}{\mathrm{d}t}
\left\Vert \phi -\overline{\phi }\right\Vert
_{\ast }^{2}+
\left(\mathbf{u}\cdot \nabla \phi ,\mathcal{N}\left(\phi -\overline{\phi }%
\right)\right)+\left(\mu ,\phi -\overline{\phi }\right)=0.
\end{equation*}%
Exploiting \ref{J-ass} and \ref{F-ass-1}, together with Young
inequality, we have
\begin{equation}
\begin{split}
\left(\mu ,\phi -\overline{\phi }\right)& \geq
\alpha \Vert \phi \Vert _{L^{2}(\Omega )}^{2}-
\left(F^{\prime }\left(\phi
^{1}\right)-F^{\prime }\left(\phi ^{2}\right),\overline{\phi }\right)  
-\left(J\ast \phi ,\phi -\overline{\phi }\right)
\\
& =\alpha
\Vert \phi \Vert _{L^{2}(\Omega )}^{2}-
\left(F^{\prime }\left(\phi ^{1}\right)-
F^{\prime}\left(\phi ^{2}\right),\overline{\phi }\right) -
\left(\nabla J\ast \phi ,\nabla \mathcal{N}\left(\phi -\overline{\phi }\right)\right)
\\
& \geq 
\alpha \Vert
\phi \Vert _{L^{2}(\Omega )}^{2}-
\left(F^{\prime }\left(\phi ^{1}\right)-F^{\prime }\left(\phi
^{2}\right),\overline{\phi }\right) -
\Vert J\Vert _{W^{1,1}(\mathbb{R}^{2})}\Vert \phi \Vert
_{L^{2}(\Omega )}
\left\Vert \phi -\overline{\phi }\right\Vert _{\ast }\\
& \geq \frac{3\alpha }{4}\Vert \phi \Vert _{L^{2}(\Omega )}^{2}-C\Vert \phi -%
\overline{\phi }\Vert _{\ast }^{2}-\left\vert \overline{\phi ^{1}}-\overline{%
\phi ^{2}}\right\vert \left( \Vert F^{\prime }(\phi ^{1})\Vert
_{L^{1}(\Omega )}+\Vert F^{\prime }(\phi ^{2})\Vert _{L^{1}(\Omega )}\right)
.
\end{split}
\label{mu-diff}
\end{equation}%
Concerning the convective term, by \eqref{H_2} and \eqref{LADY}, we obtain
\begin{equation}
\begin{split}
\left| \left(\mathbf{u}\cdot \phi ,\nabla \mathcal{N}\left(\phi -\overline{\phi }
\right) \right) \right|& \leq
C\Vert \mathbf{u}\Vert _{L^{4}(\Omega )}\Vert \phi \Vert _{L^{2}(\Omega
)}
\left\Vert \nabla \mathcal{N}\left(\phi -\overline{\phi }\right)\right\Vert _{L^{2}(\Omega )}^{%
\frac{1}{2}}
\left\Vert \phi -\overline{\phi }\right\Vert _{L^{2}(\Omega )}^{\frac{1}{2}}
\\
& \leq \frac{\alpha }{8}\Vert \phi \Vert _{L^{2}(\Omega )}^{2}+C\Vert
\mathbf{u}\Vert _{L^{4}(\Omega )}^{2}
\left(\Vert \phi \Vert +C\left|\overline{\phi }\right| \right)
\left\Vert \phi -\overline{\phi }\right\Vert _{\ast } \\
& \leq \frac{\alpha }{4}\Vert \phi \Vert _{L^{2}(\Omega )}^{2}
+C\Vert
\mathbf{u}\Vert _{L^{4}(\Omega )}^{4}
\left\Vert \phi -\overline{\phi } \right\Vert
_{\ast }^{2}+C\left\vert \overline{\phi }\right\vert ^{2}.
\end{split}
\label{U:weak-u}
\end{equation}%
Then, recalling the conservation of mass, i.e. $\overline{\phi ^{i}}(t)=%
\overline{\phi _{0}^{i}}$ for all $t\in \lbrack 0,T]$ and $i=1,2$, we are
led to
\begin{equation*}
\frac{\mathrm{d}}{\mathrm{d}t}\Vert \phi \Vert _{H^{1}(\Omega )^{\prime
}}^{2}+\alpha \Vert \phi \Vert _{L^{2}(\Omega )}^{2}\leq C\left( 1+\Vert
\mathbf{u}\Vert _{L^{4}(\Omega )}^{4}\right) \Vert \phi \Vert _{H^{1}(\Omega
)^{\prime }}^{2}+\Lambda \left\vert \overline{\phi }(0)\right\vert
+C\left\vert \overline{\phi }(0)\right\vert ^{2},
\end{equation*}%
where $\Lambda =2\left\Vert F^{\prime }\left(\phi ^{1} \right)\right\Vert _{L^{1}(\Omega )}+2\left\Vert
F^{\prime }\left(\phi ^{2}\right) \right\Vert _{L^{1}(\Omega )}$. Therefore, an application of Gronwall's Lemma implies \eqref{cd},
which, in particular, entails the
uniqueness of the weak solutions. \medskip

\noindent \textbf{Definition of the regularized problem.} Let us consider a
sequence $\{ \mathbf{u}^{k} \}_{k\in \mathbb{N}}\subset C_0^\infty(
(0,T);C_{0,\sigma }^{\infty }(\Omega ;\mathbb{R}^{2}))$ such that $\mathbf{u}%
^{k}\rightarrow \mathbf{u}$ in $L^{4}(0,T;L_{\sigma }^{4})$. We assume first
that $\phi _{0}\in H^{1}(\Omega )\cap L^{\infty }(\Omega )$ with $\Vert \phi
_{0}\Vert _{L^{\infty }(\Omega )}\leq 1$ and $|\overline{\phi _{0}}|<1$. For
any $k\in \mathbb{N}$, we introduce the Lipschitz function $h_{k}:\mathbb{R}%
\rightarrow \mathbb{R},\ k\in \mathbb{N}$ such that
\begin{equation*}
h_{k}(s)=%
\begin{cases}
-1+\frac{1}{k},\quad & s<-1+\frac{1}{k}, \\
s,\quad & s\in \left[ -1+\frac{1}{k},1-\frac{1}{k}\right] , \\
1-\frac{1}{k},\quad & s>1-\frac{1}{k}.%
\end{cases}%
\end{equation*}%
We define $\phi _{0}^{k}:=h_{k}(\phi _{0})$. It follows from the Stampacchia
superposition principle \cite{MM} that $\phi _{0}^{k}\in H^{1}(\Omega )\cap
L^{\infty }(\Omega )$ such that $\nabla \phi _{0}^{k}=\nabla \phi _{0}\cdot
\chi _{\lbrack -1+\frac{1}{k},1-\frac{1}{k}]}(\phi _{0})$ almost everywhere
in $\Omega $, where $\chi _{A}(\cdot )$ is the indicator function of the set
$A$. By definition, we have
\begin{equation}
\left|\phi _{0}^{k}\right| 
\leq \left|\phi _{0}\right|,\quad \left|\nabla \phi _{0}^{k}\right|
\leq \left|\nabla
\phi _{0}\right|,\quad \text{a.e. in }\Omega .  \label{stamp}
\end{equation}%
As a consequence, we infer that $\phi _{0}^{k}\rightarrow \phi _{0}$ in $%
H^{1}(\Omega )$ as $k\rightarrow \infty $. Observe also that $\left\vert
\overline{\phi _{0}^{k}}\right\vert \rightarrow \left\vert \overline{\phi
_{0}}\right\vert $. Then, there exist $\varpi >0$ and $\overline{k}>0$ such
that
\begin{equation}
\left\vert \overline{\phi _{0}^{k}}\right\vert \leq 1-\varpi ,\quad \forall
\,k>\overline{k}.  \label{kk}
\end{equation}%
Thanks to Theorem \ref{REG-APPR-nCH}, there exists a sequence of functions $%
\{\phi ^{k},\mu ^{k}\}_{k\in \mathbb{N}}$ satisfying
\begin{equation}
\begin{split}
& \phi ^{k}\in L^{\infty }(0,T;H^{1}(\Omega )\cap L^{\infty }(\Omega
)):\quad \sup_{t\in \lbrack 0,T]}\Vert \phi ^{k}(t)\Vert _{L^{\infty
}(\Omega )}\leq 1-\delta _{k}, \\
& \phi ^{k}\in L^{q}(0,T;W^{1,p}(\Omega )),\quad q=\frac{2p}{p-2},\quad
\forall \,p\in (2,\infty ), \\
& \partial _{t}\phi ^{k}\in L^{\infty }(0,T;H^{1}(\Omega )^{\prime })\cap
L^{2}(0,T;L^{2}(\Omega )), \\
& \mu ^{k}\in C([0,T];H^{1}(\Omega ))\cap L^{2}(0,T;H^{2}(\Omega ))\cap
H^{1}(0,T;L^{2}(\Omega )),
\end{split}
\label{Reg-Appr-nch}
\end{equation}%
where $\delta _{k}\in (0,1)$ depends on $k$. The solutions satisfies
\begin{equation}
\partial _{t}\phi ^{k}+\mathbf{u}^{k}\cdot \nabla \phi ^{k}=\Delta \mu ^{k},%
\text{ }\mu ^{k}=F^{\prime } (\phi ^{k})-J\ast \phi ^{k},\quad\text{ a.e. in }%
\Omega \times (0,T).  \label{K-nCH}
\end{equation}%
In addition, $\partial _{\mathbf{n}}\mu ^{k}=0$ almost everywhere on $%
\partial \Omega \times (0,T)$ and $\phi ^{k}(\cdot ,0)=\phi _{0}^{k}$ almost
everywhere in $\Omega $. \medskip

\noindent \textbf{Energy estimates.} Integrating \eqref{K-nCH}$_{1}$ over $%
\Omega \times (0,t)$ for any $t\in (0,T]$, we obtain the conservation of
mass
\begin{equation}
\overline{\phi ^{k}}(t)=\overline{\phi _{0}^{k}},\quad \forall \,t\in
\lbrack 0,T].  \label{cons-mass}
\end{equation}%
We multiply \eqref{K-nCH}$_{1}$ by $\mu ^{k}$ and integrate over $\Omega $.
By exploiting the convexity of $F$, the regularity \eqref{Reg-Appr-nch} and
\cite[Proposition 4.2]{CKRS2007}, we obtain
\begin{equation}
\frac{\mathrm{d}}{\mathrm{d}t}\mathcal{E}_{\mathrm{nloc}}(\phi
^{k})+\int_{\Omega }\mathbf{u}^{k}\cdot \nabla \phi ^{k}\,\mu ^{k}\,\mathrm{d%
}x+\int_{\Omega }|\nabla \mu ^{k}|^{2}\,\mathrm{d}x=0.  \label{Energy-1}
\end{equation}%
Since $\mathrm{div}\,\mathbf{u}=0$ in $\Omega \times (0,T)$ and $\mathbf{u}%
\cdot \n=0$ on $\partial \Omega \times (0,T)$, by using the
uniform bound $\Vert \phi ^{k}(t)\Vert _{L^{\infty }(\Omega )}\leq 1$, we
find
\begin{equation*}
\left\vert \int_{\Omega }\mathbf{u}^{k}\cdot \nabla \phi ^{k}\,\mu ^{k}\,%
\mathrm{d}x\right\vert =\left\vert \int_{\Omega }\mathbf{u}^{k}\cdot \nabla
\mu ^{k}\,\phi ^{k}\,\mathrm{d}x\right\vert \leq \frac{1}{2}\int_{\Omega
}|\nabla \mu ^{k}|^{2}\,\mathrm{d}x+\frac{1}{2}\int_{\Omega }|\mathbf{u}%
^{k}|^{2}\,\mathrm{d}x.
\end{equation*}%
Then, we easily infer from \eqref{Energy-1} that
\begin{equation}
\mathcal{E}_{\mathrm{nloc}} (\phi ^{k}(t) )+\frac{1}{2}\int_{0}^{t}\Vert
\nabla \mu ^{k}(\tau )\Vert _{L^{2}(\Omega )}^{2}\,\mathrm{d}\tau \leq
\mathcal{E}_{\mathrm{nloc}}(\phi _{0}^{k})+\frac{1}{2}\int_{0}^{t}\Vert
\mathbf{u}^{k}(\tau )\Vert _{L^{2}(\Omega )}^{2}\,\mathrm{d}\tau ,\quad
\forall \,t\in \lbrack 0,T].  \label{Energy-2}
\end{equation}%
We observe that $0\leq F(s)\leq C$, for $s\in \lbrack -1,1]$. By the
Young inequality, we also have that $|(J\ast u,u)|\leq \Vert J\Vert _{L^{1}(%
\mathbb{R}^{2})}\Vert u\Vert _{L^{1}(\Omega )}\Vert u\Vert _{L^{\infty
}(\Omega )}$, for any $u\in L^{\infty }(\Omega )$. Therefore, we simply
deduce that
\begin{equation}
\int_{0}^{T}\Vert \nabla \mu ^{k}(\tau )\Vert _{L^{2}(\Omega )}^{2}\,\mathrm{%
d}\tau \leq C+\int_{0}^{T}\Vert \mathbf{u}^{k}(\tau )\Vert _{L^{2}(\Omega
)}^{2}\,\mathrm{d}\tau ,  \label{nablamu-L2}
\end{equation}%
where $C$ depends on $\Omega $, $F$ and $J$, but is independent
of $k$. Along the proof, we will adopt the same agreement for all the other constants $C$ appearing in the following subsections. 
Next, we compute the gradient of \eqref{K-nCH}$_{2}$. In light of
the regularity \eqref{Reg-Appr-nch}, we notice that $F^{\prime }(\phi
^{k}(t))\in H^{1}(\Omega )$ almost everywhere in $(0,T)$ and, in particular,
$\nabla F^{\prime }(\phi )=F^{\prime \prime }(\phi ^{k})\nabla \phi $ almost
everywhere in $\Omega \times (0,T)$ by the Stampacchia superposition
principle \cite{MM}. Then, by the convexity of $F$, we have
\begin{equation}
\int_{\Omega }\left\vert F^{\prime \prime }(\phi ^{k})\nabla \phi
^{k}\right\vert ^{2}\,\mathrm{d}x\leq \Vert \nabla \mu ^{k}\Vert
_{L^{2}(\Omega )}+\Vert \nabla J\ast \phi ^{k}\Vert _{L^{2}(\Omega )}.
\end{equation}%
Thus, we infer from \ref{J-ass} and the uniform $L^{\infty }$ bound of $%
\phi ^{k}$ that
\begin{equation}
\Vert \nabla \phi ^{k}\Vert _{L^{2}(\Omega )}\leq C\left( 1+\Vert \nabla \mu
^{k}\Vert _{L^{2}(\Omega )}\right) .
\end{equation}%
Thanks to \eqref{nablamu-L2}, the above inequality entails that
\begin{equation}
\int_{0}^{T}\Vert \nabla \phi ^{k}(\tau )\Vert _{L^{2}(\Omega )}^{2}\,%
\mathrm{d}\tau \leq C(1+T)+C\int_{0}^{T}\Vert \mathbf{u}^{k}(\tau )\Vert
_{L^{2}(\Omega )}\,\mathrm{d}\tau.  \label{nablaphi-L2}
\end{equation}%
In addition, by duality in %
\eqref{K-nCH}, we easily infer that
\begin{equation}
\Vert \partial _{t}\phi ^{k}\Vert _{H^{1}(\Omega )^{\prime }}\leq \Vert
\mathbf{u}^{k}\Vert _{L^{2}(\Omega )}+\Vert \nabla \mu ^{k}\Vert
_{L^{2}(\Omega )},  \label{phit-dual}
\end{equation}%
which implies that
\begin{equation}
\int_{0}^{T}\Vert \partial _{t}\phi ^{k}(\tau )\Vert _{H^{1}(\Omega
)^{\prime }}^{2}\,\mathrm{d}\tau \leq C+\int_{0}^{T}\Vert \mathbf{u}%
^{k}(\tau )\Vert _{L^{2}(\Omega )}^{2}\,\mathrm{d}\tau .  \label{phit-L2}
\end{equation}%
\medskip

\noindent \textbf{Existence of weak solutions.} Let us consider $k\geq
\overline{k}$ such that \eqref{kk} holds. As such, we have from %
\eqref{cons-mass} that $|\overline{\phi ^{k}}(t) |\leq 1-\xi $ for all $t\in
\lbrack 0,T]$ uniformly in $k$. Then, recalling that $\mathbf{u}%
^{k}\rightarrow \mathbf{u}$ in $L^{4}(0,T;L_{\sigma }^{4}(\Omega))\hookrightarrow
L^{2}(0,T;L_{\sigma }^{2}(\Omega ))$, we infer from \eqref{nablamu-L2}, %
\eqref{nablaphi-L2} and \eqref{phit-L2} that
\begin{equation}
\Vert \phi ^{k}\Vert _{L^{\infty }(\Omega \times (0,T))}\leq 1,\quad \Vert
\phi ^{k}\Vert _{L^{2}(0,T;H^{1}(\Omega ))}+\Vert \partial _{t}\phi
^{k}\Vert _{L^{2}(0,T;H^{1}(\Omega )^{\prime })}+\Vert \nabla \mu ^{k}\Vert
_{L^{2}(0,T;L^{2}(\Omega ))}\leq C.  \label{est1-w}
\end{equation}%
In order to recover a uniform estimate
of the full $H^{1}$ norm of $\mu ^{k}$, we multiply \eqref{K-nCH}$_{2}$ by $%
\phi ^{k}-\overline{\phi ^{k}}$ and integrate over $\Omega $. By the
generalized Poincar\'{e} inequality, the assumption \ref{J-ass} and the
uniform $L^{\infty }$ bound of $\phi ^{k}$, we find
\begin{equation*}
\int_{\Omega }F^{\prime }(\phi ^{k})\big(\phi ^{k}-\overline{\phi ^{k}}\big)%
\,\mathrm{d}x\leq C\left( 1+\Vert \nabla \mu ^{k}\Vert _{L^{2}(\Omega
)}\right) .
\end{equation*}%
We report that there exists two positive constants $C_{F}^{1}$ and $%
C_{F}^{2} $ such that (see, e.g. \cite{MZ})
\begin{equation}
\Vert F^{\prime }(\phi ^{k})\Vert _{L^{1}(\Omega )}\leq C_{F}^{1}\left\vert
\int_{\Omega }F^{\prime }(\phi ^{k})\big(\phi ^{k}-\overline{\phi ^{k}}\big)%
\,\mathrm{d}x\right\vert +C_{F}^{2},  \label{MZ}
\end{equation}%
where $C_{F}^{1}$ and $C_{F}^{2}$ only depends on $F$, $\Omega $ and $\varpi
$. Then, we conclude that
\begin{equation}
\Vert \mu ^{k}\Vert _{L^{1}(\Omega )}+\Vert F^{\prime }(\phi ^{k})\Vert
_{L^{1}(\Omega )}\leq C\left( 1+\Vert \nabla \mu ^{k}\Vert _{L^{2}(\Omega
)}\right) .  \label{L1-est}
\end{equation}%
Thus, we deduce from \eqref{L1-est}, the Poincar\'{e}-Wirtinger inequality
and the definition of $\mu ^{k}$ that
\begin{equation}
\Vert \mu ^{k}\Vert _{L^{2}(0,T;H^{1}(\Omega ))}+\Vert F^{\prime }(\phi
^{k})\Vert _{L^{2}(0,T;H^{1}(\Omega ))}\leq C.  \label{est2-w}
\end{equation}%
Therefore, we infer from the Banach-Alaoglu
theorem and the Aubin-Lions theorem that
\begin{equation}
\begin{split}
\phi ^{k}\rightharpoonup \phi \quad & \text{weakly$^\star$ in }L^{\infty }(\Omega \times (0,T)), \quad
\phi ^{k}\rightharpoonup \phi \quad  \text{weakly in }L^{2}(0,T;H^{1}(%
\Omega )), \\
\phi ^{k}\rightarrow \phi \quad & \text{strongly in }L^{2}(0,T;L^{2}(\Omega
)), \quad
\partial _{t}\phi ^{k}\rightharpoonup \partial _{t}\phi \quad  \text{weakly
in }L^{2}(0,T;H^{1}(\Omega )^{\prime }), \\
\mu ^{k}\rightharpoonup \mu \quad & \text{weakly in }L^{2}(0,T;H^{1}(\Omega
)), \quad
F^{\prime }(\phi ^{k})\rightharpoonup \xi \quad  \text{weakly in }%
L^{2}(0,T;H^{1}(\Omega )).
\end{split}
\label{conv-weak}
\end{equation}%
Clearly, the limit function $\phi $ satisfies $|\phi (x,t)|\leq 1$ almost
everywhere in $\Omega \times (0,T)$. Thanks to the strong convergence in %
\eqref{conv-weak}, we infer that $\phi ^{k}\rightarrow \phi $ almost
everywhere in $\Omega \times (0,T)$. Then, $F^{\prime }(\phi
^{k})\rightarrow \widetilde{F^{\prime }}(\phi)$ almost everywhere in $\Omega \times (0,T)$, where $\widetilde{F'}(s)=F'(s)$ if $s\in (-1,1)$ and 
$\widetilde{F'}(\pm 1)=\pm \infty$. An application of the Fatou lemma, together with \eqref{est2-w},
entails that $\widetilde{F^{\prime}}(\phi )\in L^{2}(0,T;L^{2}(\Omega ))$. In turn, it also implies that $|\phi (x,t)|<1 $ almost everywhere in $\Omega \times (0,T)$. In addition, this is sufficient to conclude that $\xi =F^{\prime }(\phi )\in
L^{2}(0,T;H^{1}(\Omega ))$. Finally, passing to
the limit as $k\rightarrow \infty $ in \eqref{K-nCH}, we obtain that $\phi $
is a weak solution to \eqref{nCH} fulfilling \eqref{0}, \eqref{weak-nCH-2},
while corresponding to $\phi _{0}\in H^{1}(\Omega )\cap L^{\infty }(\Omega )$
with $\Vert \phi _{0}\Vert _{L^{\infty }(\Omega )}\leq 1$ and $|\overline{%
\phi _{0}}|<1$.

In order to conclude this part, we are left to deal the general case where
the initial datum $\phi _{0}$ only belongs to $L^{\infty }(\Omega )$ with $%
\Vert \phi _{0}\Vert _{L^{\infty }(\Omega )}\leq 1$ and $|\overline{\phi _{0}%
}|<1$. To this aim, by classical mollification there exists a sequence $%
\{\phi _{0}^{m}\}_{m\in \mathbb{N}}$ such that $\phi _{0}^{m}\in C^{\infty }(%
\overline{\Omega })$: $-1\leq \phi _{0}^{m}(x)\leq 1$ for all $x\in
\overline{\Omega }$, for any $m\in \mathbb{N}$, and $\phi
_{0}^{m}\rightarrow \phi _{0}$ strongly in $L^{r}(\Omega )$, for any $r\in
\lbrack 1,\infty )$, $\phi _{0}^{m}\rightharpoonup \phi _{0}$ weakly$^\star$ in $L^{\infty }(\Omega )$, and $|\overline{\phi_0^m}|<1$. By the previous analysis, there exists a weak
solution $\phi ^{m}$ for any $m\in \mathbb{N}$. Then, since $\mathcal{E}%
(\phi _{0}^{m})$ is uniformly bounded in $m$ and the lower semicontinuity of
the norm with respect to the weak convergence, it is straightforward to
deduce  \eqref{est1-w} and \eqref{est2-w} by replacing $\phi ^{k}$ and $\mu
^{k}$ with $\phi ^{m}$ and $\mu ^{m}$. Thus, arguing as before, the sequence
$\phi ^{m}$ converges as in \eqref{conv-weak} to a limit function $\phi $
satisfying \eqref{0} and \eqref{weak-nCH-2}, as well as $\phi (0)=\phi _{0}$
in $\Omega $.

Finally, concerning the energy identity, the convexity of $F$, the
assumption \ref{J-ass}, the regularity \eqref{0} and \cite[Proposition 4.2]%
{CKRS2007} entail that%
\begin{equation*}
\mathcal{E}_{\mathrm{nloc}}(\phi (t))-\mathcal{E}_{\mathrm{nloc}}(\phi
_{0})=\int_{0}^{t}\langle \partial _{t}\phi (\tau ),\mu (\tau )\rangle \,%
\mathrm{d}\tau, \quad \text{ for all }t\in \lbrack 0,T].
\end{equation*}%
Owing to this, \eqref{EE-nCH} directly follows from choosing $v=\mu $ in %
\eqref{weak-nCH-2} and integrating the resulting equation in $[0,t]$ for any
$0\leq t\leq T$.

%
\medskip

\noindent \textbf{Sobolev estimates.}
We first observe that the regularity of the approximated solutions $\{\phi
^{k},\mu ^{k}\}_{k\in \mathbb{N}}$ in \eqref{Reg-Appr-nch} (in particular,
the strict separation property) allows us to compute the time and the
spatial derivatives of \eqref{K-nCH}$_{2}$, which gives
\begin{equation}
\partial _{t}\mu ^{k}=F^{\prime \prime }(\phi ^{k})\partial _{t}\phi
^{k}-J\ast \partial _{t}\phi ^{k},\quad \nabla \mu ^{k}=F^{\prime \prime
}(\phi ^{k})\nabla \phi ^{k}-\nabla J\ast \phi ^{k},\quad \text{a.e. in }%
\Omega \times (0,T).  \label{mut}
\end{equation}%
In addition, the map $t\rightarrow \Vert \nabla \mu (t)\Vert _{L^{2}(\Omega
)}^{2}$ belongs to $AC([0,T])$ and the chain rule $\frac{\mathrm{d}}{\mathrm{%
d}t} \frac12\Vert \nabla \mu ^{k}\Vert _{L^{2}(\Omega )}^{2}=(\partial _{t}\mu
,\Delta \mu )$ holds almost everywhere in $(0,T)$. Thus, multiplying %
\eqref{K-nCH}$_{1}$ by $\partial _{t}\mu ^{k}$, integrating over $\Omega $,
and exploiting \eqref{mut}, we obtain
\begin{equation}
\frac{1}{2}\frac{\mathrm{d}}{\mathrm{d}t}\Vert \nabla \mu ^{k}\Vert
_{L^{2}(\Omega )}^{2}+\int_{\Omega }F^{\prime \prime }(\phi ^{k})|\partial
_{t}\phi ^{k}|^{2}\,\mathrm{d}x-\int_{\Omega }J\ast \partial _{t}\phi
^{k}\,\partial _{t}\phi ^{k}\,\mathrm{d}x+\int_{\Omega }\mathbf{u}^{k}\cdot
\nabla \phi ^{k}\,\partial _{t}\mu ^{k}\,\mathrm{d}x=0.  \label{mu7}
\end{equation}%
We rewrite the key term $(\mathbf{u}^{k}\cdot \nabla \phi ^{k},\partial
_{t}\mu ^{k})$. By using \eqref{mut} and the properties of $\mathbf{u}^{k}$
in $C_0^\infty((0,T);C_{0,\sigma }^{\infty }(\Omega ;\mathbb{R}^{2}))$, we
observe that
\begin{equation}
\begin{split}
\int_{\Omega }\mathbf{u}^{k}\cdot \nabla \phi ^{k}\,\partial _{t}\mu ^{k}\,%
\mathrm{d}x& =\int_{\Omega }\left( \mathbf{u}^{k}\cdot \nabla \phi
^{k}\right) F^{\prime \prime }(\phi ^{k})\,\partial _{t}\phi ^{k}\,\mathrm{d}%
x-\int_{\Omega }\left( \mathbf{u}^{k}\cdot \nabla \phi ^{k}\right) J\ast
\partial _{t}\phi ^{k}\,\mathrm{d}x \\
& =\int_{\Omega }\mathbf{u}^{k}\cdot \nabla \left( F^{\prime }(\phi
^{k})\right) \partial _{t}\phi ^{k}\,\mathrm{d}x-\int_{\Omega }\left(
\mathbf{u}^{k}\cdot \nabla \phi ^{k}\right) J\ast \partial _{t}\phi ^{k}\,%
\mathrm{d}x \\
& =\int_{\Omega }\left( \mathbf{u}^{k}\cdot \nabla \mu ^{k}\right) \partial
_{t}\phi ^{k}\,\mathrm{d}x-\int_{\Omega }\left( \mathbf{u}^{k}\cdot \big(%
\nabla J\ast \phi ^{k}\big)\right) \partial _{t}\phi ^{k}\,\mathrm{d}x \\
& \quad -\int_{\Omega }\left( \mathbf{u}^{k}\cdot \nabla \phi ^{k}\right)
J\ast \partial _{t}\phi ^{k}\,\mathrm{d}x \\
& =\int_{\Omega }\left( \mathbf{u}^{k}\cdot \nabla \mu ^{k}\right) \partial
_{t}\phi ^{k}\,\mathrm{d}x-\int_{\Omega }\left( \mathbf{u}^{k}\cdot \big(%
\nabla J\ast \phi ^{k}\big)\right) \partial _{t}\phi ^{k}\,\mathrm{d}x \\
& \quad -\int_{\Omega }\left( \mathbf{u}^{k}\cdot \nabla (J\ast \partial
_{t}\phi ^{k})\right) \phi ^{k}\,\mathrm{d}x.
\end{split}
\label{drift}
\end{equation}%
By exploiting the assumption \ref{J-ass} and the uniform $L^{\infty }$
bound of $\phi ^{k}$, we have
\begin{equation}
\begin{split}
\left\vert \int_{\Omega }\left( \mathbf{u}^{k}\cdot \big(\nabla J\ast \phi
^{k}\big)\right) \partial _{t}\phi ^{k}\,\mathrm{d}x\right\vert & \leq \Vert
\mathbf{u}^{k}\Vert _{L^{2}(\Omega )}\Vert \nabla J\ast \phi ^{k}\Vert
_{L^{\infty }(\Omega )}\Vert \partial _{t}\phi ^{k}\Vert _{L^{2}(\Omega )} \\
& \leq \Vert \mathbf{u}^{k}\Vert _{L^{2}(\Omega )}\Vert J\Vert _{W^{1,1}(%
\mathbb{R}^{2})}\Vert \phi ^{k}\Vert _{L^{\infty }(\Omega )}\Vert \partial
_{t}\phi ^{k}\Vert _{L^{2}(\Omega )} \\
& \leq \frac{\alpha }{8}\Vert \partial _{t}\phi ^{k}\Vert _{L^{2}(\Omega
)}^{2}+C\Vert \mathbf{u}^{k}\Vert _{L^{2}(\Omega )}^{2}.
\end{split}
\label{I2}
\end{equation}%
Similarly, we also find
\begin{equation}
\begin{split}
\left\vert \int_{\Omega }\left( \mathbf{u}^{k}\cdot \nabla (J\ast \partial
_{t}\phi ^{k})\right) \phi ^{k}\,\mathrm{d}x\right\vert & \leq \Vert \mathbf{%
u}^{k}\Vert _{L^{2}(\Omega )}\Vert \nabla J\ast \partial _{t}\phi ^{k}\Vert
_{L^{2}(\Omega )}\Vert \phi ^{k}\Vert _{L^{\infty }(\Omega )} \\
& \leq \Vert \mathbf{u}^{k}\Vert _{L^{2}(\Omega )}\Vert J\Vert _{W^{1,1}(%
\mathbb{R}^{2})}\Vert \partial _{t}\phi ^{k}\Vert _{L^{2}(\Omega )}\Vert
\phi ^{k}\Vert _{L^{\infty }(\Omega )} \\
& \leq \frac{\alpha }{8}\Vert \partial _{t}\phi ^{k}\Vert _{L^{2}(\Omega
)}^{2}+C\Vert \mathbf{u}^{k}\Vert _{L^{2}(\Omega )}^{2}.
\end{split}
\label{I3}
\end{equation}%
In order to control the first term $(\mathbf{u}^{k}\cdot \nabla \mu
^{k},\partial _{t}\phi ^{k})$ on the right-hand side in \eqref{drift}, we
need a preliminary estimate of the $H^{1}$ norm of $\nabla \mu ^{k}$. To
this end, let us first observe from \eqref{K-nCH} that $\mu^k -\overline{\mu^k }=%
\mathcal{N} (\partial _{t}\phi^k +\mathbf{u}\cdot \nabla \phi^k )$. Then, in
light of \eqref{H_2}, we find
\begin{equation}
\Vert \nabla \mu ^{k}\Vert _{H^{1}(\Omega )}\leq C\left( \Vert \partial
_{t}\phi ^{k}\Vert _{L^{2}(\Omega )}+\Vert \mathbf{u}^{k}\cdot \nabla \phi
^{k}\Vert _{L^{2}(\Omega )}\right) .  \label{muH1}
\end{equation}%
In order to estimate the second term on the right-hand side in \eqref{muH1},
we deduce from \eqref{mut} that
\begin{equation}
\mathbf{u}^{k}\cdot \nabla \phi ^{k}=\frac{1}{F^{\prime \prime }(\phi ^{k})}%
\left( \mathbf{u}^{k}\cdot \nabla \mu ^{k}+\mathbf{u}^{k}\cdot (\nabla J\ast
\phi ^{k})\right), \quad \text{a.e. in }\Omega \times (0,T).  \label{convec}
\end{equation}%
By the strict convexity of $F$, we notice that $F^{\prime \prime
}(s)^{-1}\leq \alpha ^{-1}$ for any $s\in (-1,1)$. By using \ref{J-ass}, %
\eqref{LADY} and the uniform $L^{\infty }$ bound of $\phi ^{k}$, we obtain
\begin{equation}
\begin{split}
\Vert \mathbf{u}^{k}\cdot \nabla \phi ^{k}\Vert _{L^{2}(\Omega )}& \leq
\frac{2}{\alpha }\left( \Vert \mathbf{u}^{k}\cdot \nabla \mu ^{k}\Vert
_{L^{2}(\Omega )}+\Vert \mathbf{u}\cdot \big(\nabla J\ast \phi ^{k}\big)%
\Vert _{L^{2}(\Omega )}\right) \\
& \leq C\Vert \mathbf{u}^{k}\Vert _{L^{4}(\Omega )}\Vert \nabla \mu
^{k}\Vert _{L^{4}(\Omega )}+C\Vert \mathbf{u}^{k}\Vert _{L^{2}(\Omega
)}\Vert \nabla J\ast \phi ^{k}\Vert _{L^{\infty }(\Omega )} \\
& \leq C\Vert \mathbf{u}^{k}\Vert _{L^{4}(\Omega )}\Vert \nabla \mu
^{k}\Vert _{L^{2}(\Omega )}^{\frac{1}{2}}\Vert \nabla \mu ^{k}\Vert
_{H^{1}(\Omega )}^{\frac{1}{2}}+C\Vert \mathbf{u}^{k}\Vert _{L^{2}(\Omega
)}\Vert J\Vert _{W^{1,1}(\mathbb{R}^{2})}\Vert \phi ^{k}\Vert _{L^{\infty
}(\Omega )} \\
& \leq C\Vert \mathbf{u}^{k}\Vert _{L^{4}(\Omega )}\Vert \nabla \mu
^{k}\Vert _{L^{2}(\Omega )}^{\frac{1}{2}}\Vert \nabla \mu ^{k}\Vert
_{H^{1}(\Omega )}^{\frac{1}{2}}+C\Vert \mathbf{u}^{k}\Vert _{L^{2}(\Omega )}.
\end{split}
\label{convecL2}
\end{equation}%
Then, combining \eqref{muH1} and \eqref{convecL2}, we arrive at
\begin{equation}
\Vert \nabla \mu ^{k}\Vert _{H^{1}(\Omega )}\leq C\left( \Vert \partial
_{t}\phi ^{k}\Vert _{L^{2}(\Omega )}+\Vert \mathbf{u}^{k}\Vert
_{L^{4}(\Omega )}^{2}\Vert \nabla \mu ^{k}\Vert _{L^{2}(\Omega )}+\Vert
\mathbf{u}^{k}\Vert _{L^{2}(\Omega )}\right) .  \label{muH1-f}
\end{equation}%
Now, by using \eqref{LADY} and \eqref{muH1-f}, we find
\begin{equation}
\begin{split}
\left\vert \int_{\Omega }\left( \mathbf{u}^{k}\cdot \nabla \mu ^{k}\right)
\partial _{t}\phi ^{k}\,\mathrm{d}x\right\vert & \leq \Vert \mathbf{u}%
^{k}\Vert _{L^{4}(\Omega )}\Vert \nabla \mu ^{k}\Vert _{L^{4}(\Omega )}\Vert
\partial _{t}\phi ^{k}\Vert _{L^{2}(\Omega )} \\
& \leq \Vert \mathbf{u}^{k}\Vert _{L^{4}(\Omega )}\Vert \nabla \mu ^{k}\Vert
_{L^{2}(\Omega )}^{\frac{1}{2}}\Vert \nabla \mu ^{k}\Vert _{H^{1}(\Omega )}^{%
\frac{1}{2}}\Vert \partial _{t}\phi ^{k}\Vert _{L^{2}(\Omega )} \\
& \leq \Vert \mathbf{u}^{k}\Vert _{L^{4}(\Omega )}\Vert \nabla \mu ^{k}\Vert
_{L^{2}(\Omega )}^{\frac{1}{2}}\Vert \partial _{t}\phi ^{k}\Vert
_{L^{2}(\Omega )}^{\frac{3}{2}}+\Vert \mathbf{u}^{k}\Vert _{L^{4}(\Omega
)}^{2}\Vert \nabla \mu ^{k}\Vert _{L^{2}(\Omega )}\Vert \partial _{t}\phi
^{k}\Vert _{L^{2}(\Omega )} \\
& \quad +\Vert \mathbf{u}^{k}\Vert _{L^{2}(\Omega )}^{\frac{1}{2}}\Vert
\mathbf{u}^{k}\Vert _{L^{4}(\Omega )}\Vert \nabla \mu ^{k}\Vert
_{L^{2}(\Omega )}^{\frac{1}{2}}\Vert \partial _{t}\phi ^{k}\Vert
_{L^{2}(\Omega )} \\
& \leq \frac{\alpha }{8}\Vert \partial _{t}\phi ^{k}\Vert _{L^{2}(\Omega
)}^{2}+C\Vert \mathbf{u}^{k}\Vert _{L^{4}(\Omega )}^{4}\Vert \nabla \mu
^{k}\Vert _{L^{2}(\Omega )}^{2}+C\Vert \mathbf{u}^{k}\Vert _{L^{2}(\Omega
)}^{2}.
\end{split}
\label{convec-L2-2}
\end{equation}%
%
%
%
%
%
%
%
%
%
%
%
%
%
%
%
%
%
%
%
%
%
%
%
%
%
%
%
%
%
%
%
%
%
%
%
%
%
%
%
%
%
%
%
%
%
%
%
%
Concerning the last term $(J\ast \partial _{t}\phi ^{k},\partial _{t}\phi
^{k})$ in \eqref{mu7}, by exploiting \ref{J-ass}, the properties of the
Laplace operator $A_{0}$ and \eqref{phit-dual}, we are led to
\begin{equation}
\begin{split}
\int_{\Omega }J\ast \partial _{t}\phi ^{k}\,\partial _{t}\phi ^{k}\,\mathrm{d%
}x& \leq \Vert \nabla J\ast \partial _{t}\phi ^{k}\Vert _{L^{2}(\Omega
)}\Vert \nabla \mathcal{N}\partial _{t}\phi \Vert _{L^{2}(\Omega )} \\
& \leq \frac{\alpha }{8}\Vert \partial _{t}\phi \Vert _{L^{2}(\Omega
)}^{2}+C\left( \Vert \mathbf{u}^{k}\Vert _{L^{2}(\Omega )}^{2}+\Vert \nabla
\mu ^{k}\Vert _{L^{2}(\Omega )}^{2}\right) .
\end{split}
\label{Jterm}
\end{equation}%
Inserting the estimates \eqref{I2}, \eqref{I3}, \eqref{convec-L2-2} and %
\eqref{Jterm} in \eqref{mu7}, and recalling \ref{F-ass-1}, we eventually
deduce that
\begin{equation}
\frac{1}{2}\frac{\mathrm{d}}{\mathrm{d}t}\Vert \nabla \mu ^{k}\Vert
_{L^{2}(\Omega )}^{2}+\frac{\alpha }{2}\Vert \partial _{t}\phi ^{k}\Vert
_{L^{2}(\Omega )}^{2}\leq C\left( 1+\Vert \mathbf{u}^{k}\Vert _{L^{4}(\Omega
)}^{4}\right) \Vert \nabla \mu ^{k}\Vert _{L^{2}(\Omega )}^{2}+C\Vert
\mathbf{u}^{k}\Vert _{L^{2}(\Omega )}^{2}.  \label{mut2}
\end{equation}%
Therefore, the Gronwall lemma entails that
\begin{equation}
\Vert \nabla \mu ^{k}(t)\Vert _{L^{2}(\Omega )}^{2}\leq \left( \Vert \nabla
\mu ^{k}(0)\Vert _{L^{2}(\Omega )}^{2}+C\int_{0}^{t}\Vert \mathbf{u}%
^{k}(\tau )\Vert _{L^{2}(\Omega )}^{2}+\Vert \nabla \mu ^{k}(\tau )\Vert
_{L^{2}(\Omega )}^{2}\,\mathrm{d}\tau \right) 
\mathrm{e}^{C\int_{0}^{t}\Vert
\mathbf{u}(\tau )\Vert _{L^{4}(\Omega )}^{4}\,\mathrm{d}\tau},
\label{mu-LinfH1}
\end{equation}%
for all $t\in \lbrack 0,T]$. Furthermore, integrating in time \eqref{muH1-f}
and \eqref{mut2} on $[0,T]$, and using \eqref{mu-LinfH1}, we also obtain
\begin{equation}
\begin{split}
& \int_{0}^{T}\Vert \partial _{t}\phi ^{k}(\tau )\Vert _{L^{2}(\Omega
)}^{2}+\Vert \nabla \mu ^{k}(\tau )\Vert _{H^{1}(\Omega )}^{2}\,\mathrm{d}%
\tau \\
& \quad \leq C\left( \Vert \nabla \mu ^{k}(0)\Vert _{L^{2}(\Omega
)}^{2}+C\int_{0}^{T}\Vert \mathbf{u}^{k}(\tau )\Vert _{L^{2}(\Omega
)}^{2}+\Vert \nabla \mu ^{k}(\tau )\Vert _{L^{2}(\Omega )}^{2}\,\mathrm{d}%
\tau \right) \\
& \qquad \times \left( 1+\int_{0}^{T}\Vert \mathbf{u}^{k}(\tau )\Vert
_{L^{4}(\Omega )}^{4}\,\mathrm{d}\tau \right) 
\mathrm{exp}\left( C\int_{0}^{T}\Vert
\mathbf{u}^{k}(\tau )\Vert _{L^{4}(\Omega )}^{4}\,\mathrm{d}\tau \right).
\end{split}
\label{phit-mu-L2t}
\end{equation}%
Before concluding this step, we derive a further estimate for the $L^{p}$
norms of $\nabla \phi ^{k}$. Since $\mu ^{k}\in L^{2}(0,T;H^{2}(\Omega ))$,
it follows by comparison in \eqref{mut}$_{2}$ and by the Sobolev embedding $%
H^{1}(\Omega )\hookrightarrow L^{p}(\Omega )$ for any $p\in \lbrack 2,\infty
)$ that $F^{\prime \prime }(\phi ^{k})\nabla \phi ^{k}\in
L^{2}(0,T;L^{p}(\Omega ))$. This allows us to rigorously multiply \eqref{mut}%
$_{2}$ by $|\nabla \phi ^{k}|^{p-2}\nabla \phi ^{k}$ and integrate over $%
\Omega $. As a result, we get
\begin{equation*}
\int_{\Omega }F^{\prime \prime }(\phi ^{k})|\nabla \phi ^{k}|^{p}\,\mathrm{d}%
x=\int_{\Omega }|\nabla \phi ^{k}|^{p-2}\nabla \phi ^{k}\cdot \nabla \mu
^{k}\,\mathrm{d}x+\int_{\Omega }|\nabla \phi ^{k}|^{p-2}\nabla \phi
^{k}\cdot \nabla J\ast \phi ^{k}\,\mathrm{d}x.
\end{equation*}%
By using the assumption \ref{J-ass}, the H\"{o}lder inequality and the
Young inequalities, together with uniform $L^{\infty }$ bound of $\phi ^{k}$%
, it is easily seen that
\begin{equation}
\int_{\Omega }F^{\prime \prime }(\phi ^{k})|\nabla \phi ^{k}|^{p}\,\mathrm{d}%
x\leq C\left( 1+\Vert \nabla \mu ^{k}\Vert _{L^{p}(\Omega )}^{p}\right) ,
\label{Fsec-nablaphi}
\end{equation}%
for some $C$ depending on $p$. 
This entails, in particular, that
\begin{equation}
\Vert \nabla \phi ^{k}\Vert _{L^{p}(\Omega )}\leq C\left( 1+\Vert \nabla \mu
^{k}\Vert _{L^{p}(\Omega )}\right) .  \label{Lp}
\end{equation}%
%
%
%
%
%
%
%
%
%
%
%
%
%
%
%
%
%
%
%
%
%
%
%
%
%
%
%
%
%
%
%
%
%
%
%
%
%
%
%
%
%
%
%
%
%
%
%
%
Finally, we highlight that all the constants $C$ in \eqref{mu-LinfH1}, %
\eqref{phit-mu-L2t} and \eqref{Fsec-nablaphi} are also independent of the velocity field $\mathbf{u}^{k}$ 
and the initial condition $\phi_{0}^{k}$. In fact, they only depend on $\alpha $, $J$ and $\Omega $.
\medskip

\noindent \textbf{Regularity. The case (i).}
By definition of $\phi _{0}^{k}$, it is easily seen that $\phi
_{0}^{k}\rightarrow \phi _{0}$ and $\nabla \phi _{0}^{k}\rightarrow \nabla
\phi $ almost everywhere in $\Omega $. In light of \eqref{stamp}, we also
have that $F^{\prime \prime }(\phi _{0}^{k})|\nabla \phi _{0}^{k}|\leq
F^{\prime \prime }(\phi _{0})|\nabla \phi _{0}|$ almost everywhere in $%
\Omega $. Since $F^{\prime \prime }(\phi _{0})\nabla \phi _{0}\in
L^{2}(\Omega ;\mathbb{R}^{2})$ by assumption, we simply deduce that $%
\int_{\Omega }|F^{\prime \prime }(\phi _{0}^{k})\nabla \phi
_{0}^{k}-F^{\prime \prime }(\phi _{0})\nabla \phi _{0}|^{2}\,\mathrm{d}%
x\rightarrow 0$, namely $F^{\prime \prime }(\phi _{0}^{k})\nabla \phi
_{0}^{k}\rightarrow F^{\prime \prime }(\phi _{0})\nabla \phi _{0}$ in $%
L^{2}(\Omega ;\mathbb{R}^{2})$. On the other hand, it is straightforward to
prove that $J\ast \phi _{0}^{k}\rightarrow J\ast \phi _{0}$ in $H^{1}(\Omega
) $. In addition, it is possible to show from \eqref{Reg-Appr-nch} and %
\eqref{K-nCH} that $\nabla \mu ^{k}(0)=F^{\prime \prime }(\phi
_{0}^{k})\nabla \phi _{0}^{k}-\nabla J\ast \phi _{0}^{k}$ in $\Omega $.
Therefore, we deduce that $\nabla \mu ^{k}(0)\rightarrow F^{\prime \prime
}(\phi _{0})\nabla \phi _{0}-\nabla J\ast \phi _{0}$ in $L^{2}(\Omega ;%
\mathbb{R}^{2})$. As an immediate consequence, we get
\begin{equation}
\Vert \nabla \mu ^{k}(0)\Vert _{L^{2}(\Omega )}\rightarrow \Vert F^{\prime
\prime }(\phi _{0})\nabla \phi _{0}-\nabla J\ast \phi _{0}\Vert
_{L^{2}(\Omega )},\quad \text{as }k\rightarrow \infty .  \label{nablamu0}
\end{equation}%
Next, recalling that $\mathbf{u}^{k}\rightarrow \mathbf{u}$ in $%
L^{4}(0,T;L_{\sigma }^{4})$ and the uniform estimate \eqref{nablamu-L2}, we
infer from \eqref{mu-LinfH1} and \eqref{phit-mu-L2t} that
\begin{equation}
\Vert \nabla \mu ^{k}\Vert _{L^{\infty }(0,T;L^{2}(\Omega ))}+\Vert \partial
_{t}\phi ^{k}\Vert _{L^{2}(0,T;L^{2}(\Omega ))}+\Vert \nabla \mu ^{k}\Vert
_{L^{2}(0,T;H^{1}(\Omega ))}\leq C.  \label{a1}
\end{equation}%
Thanks to \eqref{L1-est}, we obtain
\begin{equation}
\Vert \mu ^{k}\Vert _{L^{\infty }(0,T;H^{1}(\Omega ))}+\Vert \mu ^{k}\Vert
_{L^{2}(0,T;H^{2}(\Omega ))}\leq C.  \label{a2}
\end{equation}%
Concerning the concentration $\phi ^{k}$, we deduce from \eqref{Lp}, %
\eqref{a2} and the interpolation inequality $\Vert u\Vert
_{L^{q}(0,T;L^{p}(\Omega ))}\leq C\Vert u\Vert _{L^{\infty
}(0,T;L^{2}(\Omega ))}\Vert u\Vert _{L^{2}(0,T;H^{1}(\Omega ))}$, where $q=%
\frac{2p}{p-2}$ and $p\in (2,\infty )$, that
\begin{equation}
\Vert \phi ^{k}\Vert _{L^{\infty }(\Omega \times (0,T))}\leq 1,\quad \Vert
\phi ^{k}\Vert _{L^{\infty }(0,T;H^{1}(\Omega ))}+\Vert \phi ^{k}\Vert
_{L^{q}(0,T;W^{1,p}(\Omega ))}\leq C.  \label{a3}
\end{equation}%
In a similar way, by comparison in \eqref{K-nCH}, we are led to
\begin{equation}
\Vert F^{\prime }(\phi ^{k})\Vert _{L^{\infty }(0,T;H^{1}(\Omega ))}+\Vert
F^{\prime }(\phi ^{k})\Vert _{L^{q}(0,T;W^{1,p}(\Omega ))}\leq C.
\label{a4}
\end{equation}%
Furthermore, recalling \eqref{phit-dual}, we obtain from $\mathbf{u}\in
L^{4}(0,T;L_{\sigma }^{4}(\Omega ))$ and \eqref{a2} that
\begin{equation}
\Vert \partial _{t}\phi ^{k}\Vert _{L^{4}(0,T;H^{1}(\Omega )^{\prime })}\leq C.  \label{a5}
\end{equation}%
Also, by the assumption \ref{h3}, an application of the Trudinger-Moser
inequality and the estimate \eqref{a4} (cf. \cite[Theorem 5.2]{GGG2017})
entails that
\begin{equation}
\Vert F^{\prime \prime }(\phi ^{k})\Vert _{L^{\infty }(0,T;L^{p}(\Omega
))}\leq C_p,\quad \forall \, p\in \lbrack 2,\infty ).  \label{a6}
\end{equation}%
Notice that $\partial _{t}F^{\prime }(\phi ^{k})=F^{\prime \prime }(\phi
^{k})\partial _{t}\phi ^{k}$. Owing to this, for any $v\in H^{1}(\Omega )$,
we have
\begin{equation}
\left\vert \langle \partial _{t}F^{\prime }(\phi ^{k}),v\rangle \right\vert
\leq \Vert F^{\prime \prime }(\phi ^{k})\Vert _{L^{4}(\Omega )}\Vert
\partial _{t}\phi ^{k}\Vert _{L^{2}(\Omega )}\Vert v\Vert _{L^{4}(\Omega )},
\label{a7}
\end{equation}%
which, in turn, implies that
\begin{equation}
\Vert \partial _{t}F^{\prime }(\phi ^{k})\Vert _{L^{2}(0,T;H^{1}(\Omega
)^{\prime })}\leq C.  \label{a8}
\end{equation}%
Thus, in light of \eqref{mut} and \eqref{a1}, we immediately deduce that
\begin{equation}
\Vert \partial _{t}\mu ^{k}\Vert _{L^{2}(0,T;H^{1}(\Omega )^{\prime })}\leq C.  \label{a9}
\end{equation}%
Exploiting the above uniform estimates \eqref{a1}-\eqref{a9}, by a
compactness argument (simpler than the one for the existence of weak
solutions), we pass to the limit in \eqref{K-nCH} as $k\rightarrow \infty $
obtaining that the limit function $\phi $ is a strong solution to \eqref{nCH}
satisfying \eqref{3}. In particular, \eqref{nCH}$_{1}$ holds almost
everywhere in $\Omega \times (0,T)$, \eqref{nCH}$_{2}$ holds almost
everywhere in $\partial \Omega \times (0,T)$. Since the map $t\rightarrow
\{x\in \Omega :|\phi (x,t)|=1\}$ is continuous in $[0,T]$ and $F^{\prime
}(\phi )\in L^{\infty }(0,T;H^{1}(\Omega ))$, we observe that $|\{x\in
\Omega :|\phi (x,t)|=1\}|=0$ for all $t\in \lbrack 0,T]$. Besides, the
estimates \eqref{ST1}, \eqref{ST2} and \eqref{ST3} follows from %
\eqref{L1-est}, \eqref{mu-LinfH1}, \eqref{phit-mu-L2t}, \eqref{Lp}, %
\eqref{nablamu0}, \eqref{a7} and the lower semicontinuity of the norm with
respect to the weak convergence.

Finally, if we also assume $\mathbf{u}\in L^{\infty }(0,T;L_{\sigma
}^{2}(\Omega ))$ then, by comparison in \eqref{nCH} (cf. \eqref{phit-dual}), it
is easily seen that $\partial _{t}\phi \in L^{\infty }(0,T;H^{1}(\Omega
)^{\prime })$, which concludes the proof related to the case (i). \medskip

\noindent \textbf{Separation property and further regularity. The case (ii).}
We may first argue as in the proof of \cite[Theorem 4.1]{GGG2022} to
conclude with (\ref{del-SP}) through a direct argument (see Remark \ref%
{Rem-SP}). Then we can also modify that proof slightly upon eliminating the
cut-off function $\eta _{n}$, instead by testing (\ref{weaky}) with $v:=\phi
_{n}\left( t\right) =\left( \phi \left( t\right) -k_{n}\right) _{+},$ for
the increasing sequence $k_{n}:=1-\delta -\delta 2^{-n}\overset{n\rightarrow
\infty }{\rightarrow }1-\delta $, $1-2\delta <k_{n}<1-\delta ,$ where $%
0<\delta <\delta _{0}/2$ (this implies that $\phi _{n}\left( 0\right) =0,$
for all $n\in \mathbb{N}_{0}$). Moreover, the drift term becomes%
\begin{equation}
\mathcal{Z}:=\int_{\Omega }\phi \,\mathbf{u}\cdot \nabla \phi _{n}\mathrm{d}%
x=\int_{\Omega }\mathbf{u}\cdot \nabla \left( \phi _{n}\right) ^{2}\mathrm{d}%
x=0,  \label{Z-term}
\end{equation}%
so that one has\footnote{%
Here $A_{n}\left( t\right) :=\{x\in \Omega :\phi \left( x,t\right) \geq
k_{n}\}, \quad t\in \left[ 0,1\right] .$}
\begin{equation*}
\frac{1}{2}\Vert \phi _{n}(t)\Vert _{L^{2}(\Omega )}^{2}+F^{\prime \prime
}\left( 1-2\delta \right) \int_{0}^{t}\Vert \nabla \phi _{n}\Vert
_{L^{2}(\Omega )}^{2}\,\mathrm{d}s\leq \int_{0}^{t}\int_{A_{n}(s)}(\nabla
J\ast \phi )\cdot \nabla \phi _{n}\mathrm{d}x\,\mathrm{d}s,
\end{equation*}%
for all $t\in \lbrack 0,1]$, assuming $T\geq 1$ without loss of generality. This inequality allows us to make minor changes
in the arguments employed in \cite[Theorem 4.1]{GGG2022} to deduce that $%
\left\Vert \phi \right\Vert _{L^{\infty }\left( \left[ 0,1\right] \times
\Omega \right) }\leq 1-\delta $, and to exploit (\ref{del-SP}) to obtain %
\eqref{SP}.

Next, setting the difference quotient $\partial _{t}^{h}f(t)=h^{-1}\left(
f(t+h)-f(t)\right) $ for $0<t<T-h$, we write
\begin{equation*}
\partial _{t}^{h}\mu (t)=\partial _{t}^{h}\phi (t)\left(
\int_{0}^{1}F^{\prime \prime }(s\phi (t+h)+(1-s)\phi (t))\,\mathrm{d}%
s\right) -J\ast \partial _{t}^{h}\phi (t),\quad \text{a.e. $t$ in }(0,T).
\end{equation*}
In light of \eqref{SP}, it easily follows that $\Vert s\phi (\cdot
+h)-(1-s)\phi \Vert _{L^{\infty }(\Omega \times (0,T-h))}\leq 1-\delta $ for
all $s\in (0,1)$. Owing to this, and using the properties of $J$ (see \ref%
{J-ass}) and the basic inequality $\Vert \partial _{t}^{h}\phi \Vert
_{L^{2}(0,T-h;L^{2}(\Omega ))}\leq \Vert \partial _{t}\phi \Vert
_{L^{2}(0,T;L^{2}(\Omega ))}$, we deduce that
\begin{equation*}
\Vert \partial _{t}^{h}\mu \Vert _{L^{2}(0,T-h;L^{2}(\Omega ))}\leq C,
\end{equation*}%
where $C$ is independent of $h$. This entails that $\partial _{t}\mu \in
L^{2}(0,T;L^{2}(\Omega ))$. In turn, thanks to $\mu \in
L^{2}(0,T;H^{2}(\Omega ))$, we also obtain $\mu \in C([0,T];H^{1}(\Omega ))$%
. The proof is thus concluded.
\end{proof}




\section{Proof of Theorem \protect\ref{MR:strong}: Existence of strong
solutions}

\label{proofmain} \setcounter{equation}{0}

First of all, we rewrite the nonlocal AGG model \eqref{syst2} in the
non-conservative form as
\begin{equation}
\begin{cases}
\rho (\phi )\partial _{t}\mathbf{u}+\rho (\phi )(\mathbf{u}\cdot \nabla )%
\mathbf{u}-\rho ^{\prime }(\phi )(\nabla \mu \cdot \nabla )\mathbf{u}-%
\mathrm{div}(\nu (\phi )D\mathbf{u})+\nabla \Pi =\mu \nabla \phi , \\
\mathrm{div}\,\mathbf{u}=0, \\
\partial _{t}\phi +\mathbf{u}\cdot \nabla \phi =\Delta \mu , \\
\mu =F^{\prime }(\phi )-J\ast \phi ,%
\end{cases}
\label{NL-AGG}
\end{equation}%
in $\Omega \times (0,T)$, which is endowed with the boundary and initial
conditions \eqref{bic}. \smallskip

\noindent \textbf{The approximate problem: the semi-Galerkin scheme.}
Let us consider the family of eigenvalues $\{\lambda _{j}\}_{j=1}^{\infty }$
and and corresponding eigenfunctions $\{\mathbf{w}_{j}\}_{j=1}^{\infty }$ of
the Stokes operator $\mathbf{A}$. For any integer $m\geq 1$, let ${\mathbf{V}%
}_{m}$ denote the finite-dimensional subspaces of $L_{\sigma }^{2}(\Omega )$
defined as ${\mathbf{V}}_{m}=\text{span}\{\mathbf{w}_{1},\cdots ,\mathbf{w}%
_{m}\}$. The orthogonal projection on ${\mathbf{V}}_{m}$ with respect to the
inner product in $L_{\sigma }^{2}(\Omega )$ is denoted by $\mathbb{P}_{m}$.
Recalling that $\Omega $ is a $C^{2}$-domain, we have that $\mathbf{w}%
_{j}\in H_{0,\sigma }^{2}(\Omega )$ for all $j\in \mathbb{N}$. In addition,
the following inequalities hold
\begin{equation}
\Vert \mathbf{v}\Vert _{H^{1}(\Omega )}\leq C_{m}\Vert \mathbf{v}\Vert
_{L^{2}(\Omega )},\quad \Vert \mathbf{v}\Vert _{H^{2}(\Omega )}\leq
C_{m}\Vert \mathbf{v}\Vert _{L^{2}(\Omega )}\quad \forall \,\mathbf{v}\in {%
\mathbf{V}}_{m}.  \label{sob}
\end{equation}%
Let us fix $T>0$. For any $m\in \mathbb{N}$, we claim that there exists an
approximate solution $(\mathbf{u}_{m},\phi _{m})$ to system \eqref{syst2}-%
\eqref{bic} in the following sense:
\begin{equation}
\begin{cases}
\mathbf{u}_{m}\in C([0,T];{\mathbf{V}}_{m})\cap H^{1}(0,T;{\mathbf{V}}%
_{m}), \\
\phi _{m}\in L^{\infty }(\Omega \times (0,T)):\quad |\phi (x,t)|<1 \ \text{ 
a.e. in }\Omega \times (0,T), \\
\phi _{m}\in L^{\infty }(0,T;H^{1}(\Omega ))\cap L^{q}(0,T;W^{1,p}(\Omega
)),\quad q=\frac{2p}{p-2},\quad \forall \,p\in (2,\infty ), \\
\partial _{t}\phi _{m}\in L^{\infty }(0,T;H^{1}(\Omega )^{\prime })\cap
L^{2}(0,T;L^{2}(\Omega )), \\
\mu _{m}\in L^{\infty }(0,T;H^{1}(\Omega ))\cap L^{2}(0,T;H^{2}(\Omega
))\cap H^{1}(0,T;H^{1}(\Omega )^{\prime }), \\
F^{\prime }(\phi _{m})\in L^{\infty }(0,T;H^{1}(\Omega )),\quad F^{\prime
\prime }(\phi _{m})\in L^{\infty }(0,T;L^{p}(\Omega )),\quad \forall \,p\in
\lbrack 2,\infty ),
\end{cases}
\label{rg}
\end{equation}%
such that
\begin{equation}
\begin{split}
& \left( \rho (\phi _{m})\partial _{t}\mathbf{u}_{m},\mathbf{w}\right)
+\left( \rho (\phi _{m})(\mathbf{u}_{m}\cdot \nabla )\mathbf{u}_{m},\mathbf{w%
}\right) +\left( \nu (\phi _{m})D\mathbf{u}_{m},\nabla \mathbf{w}\right) \\
& \quad -\frac{\rho _{1}-\rho _{2}}{2}\left( (\nabla \mu _{m}\cdot \nabla )%
\mathbf{u}_{m},\mathbf{w}\right) =-\left( \phi _{m}\nabla \mu _{m},\mathbf{w}%
\right) ,
\end{split}
\label{vel}
\end{equation}%
for all $\mathbf{w}\in {\mathbf{V}}_{m}$, in $[0,T]$, and
\begin{equation}
\partial _{t}\phi _{m}+\mathbf{u}_{m}\cdot \nabla \phi _{m}=\Delta \mu
_{m},\quad \mu _{m}=F^{\prime }(\phi _{m})-J\ast \phi _{m},\quad \text{a.e.
in }\Omega \times (0,T).  \label{eqapprox}
\end{equation}%
In addition, the approximate solution $(\mathbf{u}_{m},\phi _{m})$ satisfies
the boundary and initial conditions
\begin{equation}
\begin{cases}
\mathbf{u}_{m}=\mathbf{0},\quad \partial _{\mathbf{n}}\mu _{m}=0\quad &
\text{on }\partial \Omega \times (0,T), \\
\mathbf{u}_{m}(\cdot ,0)=\mathbb{P}_{m}\mathbf{u}_{0},\quad \phi (\cdot
,0)=\phi _{0}\quad & \text{in }\Omega .%
\end{cases}
\label{init}
\end{equation}%
\medskip

\noindent \textbf{Existence of the approximate solutions.} We perform a
fixed point argument to determine the existence of the approximate solutions
satisfying \eqref{vel}-\eqref{init} as in \cite{AGG2d, AGG3d}. To this aim,
we suppose that $\mathbf{v}\in C([0,T];{\mathbf{V}}_m)$ is given. Then the
corresponding convective nonlocal Cahn-Hilliard system reads as
\begin{equation}
\partial_t\phi_m+\mathbf{v}\cdot \nabla\phi_m=\Delta\mu_m, \quad
\mu_m=F^\prime(\phi_m)-J\ast \phi_m \quad \text{in } \Omega\times(0,T),
\label{eq2}
\end{equation}
with boundary and initial conditions
\begin{equation}
\partial_\textbf{n}\mu_m=0\quad \text{on }\partial\Omega\times(0,T),\quad
\phi_m(\cdot, 0)=\phi_0\quad \text{in }\Omega\times(0,T).  \label{bound2}
\end{equation}
Thanks to the case (i) of Theorem \ref{ExistCahn}, there exists a unique
solution $\phi_m$ to \eqref{eq2}-\eqref{bound2} such that
\begin{equation}  \label{reg}
\begin{cases}
\phi_m\in L^\infty(\Omega\times(0,T)):\quad \vert\phi\vert<1 \ \text{ 
a.e. in } \Omega\times(0,T), \\
 \phi_m \in L^\infty(0,T;H^1(\Omega)) \cap L^q(0,T;W^{1,p}(\Omega)),\quad q=%
\frac{2p}{p-2},\quad \forall \, p\in(2,\infty), \\
 \partial_t\phi_m\in L^\infty(0,T;H^1(\Omega)^\prime)\cap
L^2(0,T;L^2(\Omega)), \\
 \mu_m\in L^\infty(0,T;H^1(\Omega))\cap L^2(0,T; H^2(\Omega))\cap H^1(0,T;
H^1(\Omega)^\prime), \\
 F^\prime(\phi_m)\in L^\infty(0,T; H^1(\Omega)),\quad
F^{\prime\prime}(\phi_m)\in L^\infty(0,T; L^p(\Omega)),\quad \forall \,
p\in[2,\infty).
\end{cases}%
\end{equation}
Moreover, on account of \eqref{EE-nCH}, by repeating line by line the proof
of Theorem \ref{ExistCahn} (cf. energy estimates), we find
\begin{equation}  \label{nablamu-L2-m}
\int_0^T \| \nabla \mu_m(\tau)\|_{L^2(\Omega)}^2 \, \mathrm{d} \tau \leq C
+ \int_0^T \| \mathbf{v} (\tau)\|_{L^2(\Omega)}^2 \, \mathrm{d} \tau.
\end{equation}
where $C$ depends only on $\Omega$, $F$ and $J$, but is independent of $m$ as all the other constants $C$ in the sequel of this proof.

%
We now make the ansatz
\begin{equation*}
\mathbf{u}_{m}(x,t)=\sum_{j=1}^{m}a_{j}^{m}(t)\mathbf{w}_{j}(x),\quad
\forall \,(x,t)\in \Omega \times \lbrack 0,T],
\end{equation*}%
as the solution to the Galerkin approximation of \eqref{vel}, that is,
\begin{equation}
\begin{split}
& \left( \rho (\phi _{m})\partial _{t}\mathbf{u}_{m},\mathbf{w}_{l}\right)
+\left( \rho (\phi _{m})(\mathbf{u}_{m}\cdot \nabla )\mathbf{u}_{m},\mathbf{w%
}_{l}\right) +\left( \nu (\phi _{m})D\mathbf{u}_{m},\nabla \mathbf{w}%
_{l}\right) \\
& \quad -\frac{\rho _{1}-\rho _{2}}{2}\left( (\nabla \mu _{m}\cdot \nabla )%
\mathbf{u}_{m},\mathbf{w}_{l}\right) =-\left( \phi _{m}\nabla \mu _{m},%
\mathbf{w}_{l}\right) ,\quad \forall \,l=1,\dots ,m.
\end{split}
\label{velgal}
\end{equation}%
satisfying the initial condition $\mathbf{u}_{m}(\cdot ,0)=\mathbb{P}_{m}\mathbf{u}%
_{0} $.

Arguing as in \cite[Section 4]{AGG2d}, we introduce $\mathbf{A}%
^{m}(t)=(a_{1}^{m}(t),\ldots ,a_{m}^{m}(t))^{T}$ and we observe that %
\eqref{velgal} is equivalent to the system of differential equations
\begin{equation*}
\mathbf{M}^{m}(t)\frac{\mathrm{d}}{\mathrm{d}t}\mathbf{A}^{m}+\mathbf{L}%
^{m}(t)\mathbf{A}^{m}=\mathbf{G}^{m}(t),
\end{equation*}%
where the matrices $\mathbf{M}^{m}(t),\mathbf{L}^{m}(t)$ and the vector $%
\mathbf{G}^{m}(t)$ are defined as follows:%
\begin{equation*}
\left( \mathbf{M}^{m}(t)\right) _{lj}=\int_{\Omega }\rho (\phi _{m}(t))%
\mathbf{w}_{j}\cdot \mathbf{w}_{l}\,\mathrm{d}x,
\end{equation*}%
\begin{align*}
\left( \mathbf{L}^{m}(t)\right) _{lj}& =\int_{\Omega }\rho (\phi _{m}(t))(%
\mathbf{v}(t)\cdot \nabla )\mathbf{w}_{j}\cdot \mathbf{w}_{l}+\nu (\phi
_{m}(t))D\mathbf{w}_{j}:\nabla \mathbf{w}_{l}\mathrm{d}x \\
& -\int_{\Omega }\left( \frac{\rho _{1}-\rho _{2}}{2}\right) (\nabla \mu
_{m}(t)\cdot \nabla )\mathbf{w}_{j}\cdot \mathbf{w}_{l},
\end{align*}%
\begin{equation*}
\left( \mathbf{G}(t)\right) _{l}=-\int_{\Omega }\phi _{m}(t)\nabla \mu
_{m}(t)\cdot \mathbf{w}_{l}\,\mathrm{d}x,
\end{equation*}%
as well as $\mathbf{A}^{m}(0)=\left( (\mathbf{u}_{0},\mathbf{w}_{1}),\ldots
,(\mathbf{u}_{0},\mathbf{w}_{m})\right) ^{T}$. The regularity properties in %
\eqref{reg} imply that both $\phi _{m}$ and $\mu _{m}$ belongs to $C_{%
\mathrm{w}}([0,T];H^{1}(\Omega ))\cap C([0,T];L^{p}(\Omega ))$ for any $p\in
\lbrack 1,\infty )$ (cf. \cite{STRAUSS}).
In turn, since $\rho (\cdot )$ is a linear function and $\nu $ is globally
Lipschitz, $\rho (\phi _{m})$ and $\nu (\phi _{m})$ also belong to $%
C([0,T];L^{p}(\Omega ))$ for any $p\in \lbrack 1,\infty )$. As such, we
immediately observe that, for any $s,t\in \lbrack 0,T]$,
\begin{equation*}
\left\vert (\mathbf{M}^{m}(t))_{lj}-(\mathbf{M}^{m}(s))_{lj}\right\vert \leq
\Vert \mathbf{w}_{j}\Vert _{L^{\infty }(\Omega )}\Vert \mathbf{w}_{l}\Vert
_{L^{\infty }(\Omega )}\int_{\Omega }|\rho (\phi _{m}(t))-\rho (\phi
_{m}(s))|\,\mathrm{d}x\underset{s\rightarrow t}{\longrightarrow }{0}.
\end{equation*}%
Since ${\mathbf{V}}_{m}\subset H_{0,\sigma }^{2}(\Omega )$, we observe that
\begin{equation*}
(\mathbf{G}^{m}(t))_{l}-(\mathbf{G}^{m}(s))_{l}=-\int_{\Omega }\left( \phi
_{m}(t)-\phi _{m}(s)\right) \nabla \mu _{m}(t)\cdot \mathbf{w}_{l}\,\mathrm{d%
}x-\int_{\Omega }\left( \mu _{m}(t)-\mu _{m}(s)\right) \nabla \phi
_{m}(s)\cdot \mathbf{w}_{l}\,\mathrm{d}x.
\end{equation*}%
Then, recalling that $\phi _{m}$ and $\mu _{m}$ belongs to $C_{\mathrm{w}%
}([0,T];H^{1}(\Omega ))$, we infer that
\begin{align*}
|(\mathbf{G}^{m}(t))_{l}-(\mathbf{G}^{m}(s))_{l}|& \leq \Vert \mathbf{w}%
_{l}\Vert _{L^{\infty }(\Omega )}\Vert \nabla \mu _{m}(t)\Vert
_{L^{2}(\Omega )}\Vert \phi _{m}(t)-\phi _{m}(s)\Vert _{L^{2}(\Omega )} \\
& \quad +\Vert \mathbf{w}_{l}\Vert _{L^{\infty }(\Omega )}\Vert \nabla \phi
_{m}(s)\Vert _{L^{2}(\Omega )}\Vert \mu _{m}(t)-\mu _{m}(s)\Vert
_{L^{2}(\Omega )}\underset{s\rightarrow t}{\longrightarrow }{0}.
\end{align*}%
Furthermore, exploiting once again that ${\mathbf{V}}_{m}\subset H_{0,\sigma
}^{2}(\Omega )$, we notice that
\begin{align*}
& \left\vert \int_{\Omega }\rho (\phi _{m}(t))(\mathbf{v}(t)\cdot \nabla )%
\mathbf{w}_{j}\cdot \mathbf{w}_{l}\,\mathrm{d}x-\int_{\Omega }\rho (\phi
_{m}(s))(\mathbf{v}(s)\cdot \nabla )\mathbf{w}_{j}\cdot \mathbf{w}_{l}\,%
\mathrm{d}x\right\vert \\
& =\left\vert \int_{\Omega }\left( \rho (\phi _{m}(t))-\rho (\phi
_{m}(s))\right) (\mathbf{v}(t)\cdot \nabla )\mathbf{w}_{j}\cdot \mathbf{w}%
_{l}\,\mathrm{d}x+\int_{\Omega }\rho (\phi _{m}(s))((\mathbf{v}(t)-\mathbf{v}%
(s))\cdot \nabla )\mathbf{w}_{j}\cdot \mathbf{w}_{l}\,\mathrm{d}x\right\vert
\\
& \leq \Vert \rho (\phi _{m}(t))-\rho (\phi _{m}(s))\Vert _{L^{3}(\Omega
)}\Vert \mathbf{v}(t)\Vert _{L^{2}(\Omega )}\Vert \nabla \mathbf{w}_{j}\Vert
_{L^{6}(\Omega )}\Vert \mathbf{w}_{l}\Vert _{L^{\infty }(\Omega )} \\
& \quad +\Vert \rho (\phi _{m}(s))\Vert _{L^{6}(\Omega )}\Vert \mathbf{v}(t)-%
\mathbf{v}(s)\Vert _{L^{2}(\Omega )}\Vert \nabla \mathbf{w}_{j}\Vert
_{L^{3}(\Omega )}\Vert \mathbf{w}_{l}\Vert _{L^{\infty }(\Omega )}\underset{%
s\rightarrow t}{\longrightarrow }{0}
\end{align*}%
and
\begin{align*}
&\left\vert \int_{\Omega }\left( \nu (\phi _{m}(t))-\nu (\phi _{m}(s))\right)
D\mathbf{w}_{j}:\nabla \mathbf{w}_{l}\,\mathrm{d}x\right\vert \leq \Vert \nu
(\phi _{m}(t))-\nu (\phi _{m}(s))\Vert _{L^{2}(\Omega )}\Vert D\mathbf{w}%
_{j}\Vert _{L^{4}(\Omega )}\Vert \nabla \mathbf{w}_{l}\Vert _{L^{4}(\Omega )}%
\underset{s\rightarrow t}{\longrightarrow }{0}.
\end{align*}%
On the other hand, integrating by parts and exploiting the boundary
conditions, we have
\begin{equation*}
\int_{\Omega }(\nabla \mu _{m}(t)\cdot \nabla )\mathbf{w}_{j}\cdot \mathbf{w}%
_{l}\,\mathrm{d}x=-\int_{\Omega }\mu (t)\Delta \mathbf{w}_{j}\cdot \mathbf{w}%
_{l}\,\mathrm{d}x-\int_{\Omega }\mu (t)\nabla \mathbf{w}_{j}:\nabla \mathbf{w%
}_{l}\,\mathrm{d}x.
\end{equation*}%
Thus, we find
\begin{equation*}
\left\vert \int_{\Omega }(\nabla (\mu _{m}(t)-\mu _{m}(s))\cdot \nabla )%
\mathbf{w}_{j}\cdot \mathbf{w}_{l}\,\mathrm{d}x\right\vert \leq C\Vert
\mathbf{w}_{j}\Vert _{H^{2}(\Omega )}\Vert \mathbf{w}_{l}\Vert
_{H^{2}(\Omega )}\Vert \mu _{m}(s)-\mu _{m}(t)\Vert _{L^{2}(\Omega )}%
\underset{s\rightarrow t}{\longrightarrow }{0}.
\end{equation*}%
Thus, we derive that $\mathbf{M}^{m}$ and $\mathbf{L}^{m}$ belong to $%
C([0,T];\mathbb{R}^{m\times m})$ and $\mathbf{G}^{m}\in C([0,T];\mathbb{R}%
^{m})$. Furthermore, being $\rho $ strictly positive, we also have that $%
\mathbf{M}^{m}$ is positive definite and thus the inverse $(\mathbf{M}%
^{m})^{-1}\in C([0,T];\mathbb{R}^{m\times m})$. Thus, the existence and
uniqueness theorem for systems of linear ODEs guarantees that there exists a
unique solution $\mathbf{u}_{m}\in C^{1}([0,T];{\mathbf{V}}_{m})$.

Next, multiplying \eqref{velgal} by $a_{l}^{m}$ and summing over $l$, we
obtain
\begin{align*}
& \int_{\Omega }\rho (\phi _{m})\partial _{t}\left( \frac{|\mathbf{u}%
_{m}|^{2}}{2}\right) \,\mathrm{d}x+\int_{\Omega }\rho (\phi _{m})\mathbf{v}%
\cdot \nabla \left( \frac{|\mathbf{u}_{m}|^{2}}{2}\right) \,\mathrm{d}%
x+\int_{\Omega }\nu (\phi _{m})|D\mathbf{u}_{m}|^{2}\,\mathrm{d}x \\
& \quad -\int_{\Omega }\left( \frac{\rho _{1}-\rho _{2}}{2}\right) \nabla
\mu _{m}\cdot \nabla \left( \frac{|\mathbf{u}_{m}|^{2}}{2}\right) \,\mathrm{d%
}x=-\int_{\Omega }\phi _{m}\nabla \mu _{m}\cdot \mathbf{u}_{m}\,\mathrm{d}x.
\end{align*}%
Arguing exactly as in \cite[Section 4.2]{AGG2d} and exploiting \eqref{eqapprox}, we
deduce that
\begin{equation}
\frac{\mathrm{d}}{\mathrm{d}t}\int_{\Omega }\rho (\phi _{m})\frac{|\mathbf{u}%
_{m}|^{2}}{2}\,\mathrm{d}x+\int_{\Omega }\nu (\phi _{m})|D\mathbf{u}%
_{m}|^{2}\,\mathrm{d}x=-\int_{\Omega }\phi _{m}\nabla \mu _{m}\cdot \mathbf{u%
}_{m}\,\mathrm{d}x.
\end{equation}%
By the Poincar\'{e}-Korn (see \eqref{korn}) and the Young inequalities, as well as $|\phi _{m}|<1\text{ almost everywhere in}\ \Omega \times (0,T)$, we
infer that
\begin{equation}
\begin{split}
-\int_{\Omega }\phi _{m}\nabla \mu _{m}\cdot \mathbf{u}_{m}\,\mathrm{d}x&
=\int_{\Omega }\mu _{m}\nabla \phi _{m}\cdot \mathbf{u}_{m}\,\mathrm{d}x \\
& =\int_{\Omega }F^{\prime }(\phi _{m})\nabla \phi _{m}\cdot \mathbf{u}_{m}\,%
\mathrm{d}x-\int_{\Omega }J\ast \phi _{m}\nabla \phi _{m}\cdot \mathbf{u}%
_{m}\,\mathrm{d}x \\
& =\underbrace{\int_{\Omega }\nabla (F(\phi _{m}))\cdot \mathbf{u}_{m}\,%
\mathrm{d}x}_{=0}-\int_{\Omega }\phi _{m}\nabla J\ast \phi _{m}\cdot \mathbf{%
u}_{m}\,\mathrm{d}x \\
& \leq \Vert \phi _{m}\Vert _{L^{\infty }(\Omega )}\Vert \nabla J\Vert
_{L^{1}(\mathbb{R}^{2})}\Vert \phi _{m}\Vert _{L^{2}(\Omega )}\Vert \mathbf{u%
}_{m}\Vert _{L^{2}(\Omega )} \\
& \leq \sqrt{\frac{2}{\lambda _{1}}}\Vert \nabla J\Vert _{L^{1}(\mathbb{R}%
^{2})}|\Omega |^{\frac{1}{2}}\Vert D\mathbf{u}_{m}\Vert _{L^{2}(\Omega )} \\
& \leq \frac{\nu _{\ast }}{2}\Vert D\mathbf{u}_{m}\Vert ^{2}+\frac{\Vert
\nabla J\Vert _{L^{1}(\mathbb{R}^{2})}^{2}|\Omega |}{\nu _{\ast }\lambda _{1}%
}.
\end{split}
\label{RHS-NS}
\end{equation}%
Here we have used that $\nabla F(\phi _{m})=F^{\prime }(\phi _{m})\nabla
\phi _{m}$ almost everywhere in $\Omega \times (0,T)$ (see Remark \ref{R3}).
We are thus led to the differential inequality
\begin{equation*}
\frac{\mathrm{d}}{\mathrm{d}t}\int_{\Omega }\rho (\phi _{m})\frac{|\mathbf{u}%
_{m}|^{2}}{2} \, \d x+\frac{\nu _{\ast }}{2}\Vert D\mathbf{u}_{m}\Vert ^{2}\leq
\frac{\Vert \nabla J\Vert _{L^{1}(\mathbb{R}^{2})}^{2}|\Omega |}{\nu _{\ast
}\lambda _{1}}.
\end{equation*}%
Integrating the above inequality in time, with $s\in \lbrack 0,T]$ and
exploiting the upper and lower bounds of $\rho $, we get
\begin{equation}
\max_{t\in \lbrack 0,T]}\Vert \mathbf{u}_{m}(t)\Vert _{L^{2}(\Omega
)}^{2}\leq \frac{\rho ^{\ast }}{\rho _{\ast }}\Vert \mathbf{u}_{0}\Vert
_{L^{2}(\Omega )}^{2}+\frac{2\Vert \nabla J\Vert _{L^{1}(\mathbb{R}%
^{2})}^{2}|\Omega |T}{\nu _{\ast }\lambda _{1}}=:M^{2}.  \label{u_m-1est}
\end{equation}%
Let us now introduce the closed ball
\begin{equation*}
X:=\left\{ \mathbf{u}\in C([0,T];{\mathbf{V}}_{m}):\Vert \mathbf{u}\Vert
_{C([0,T];{\mathbf{V}}_{m})}\leq M\right\} ,
\end{equation*}%
and define the map
\begin{equation*}
\mathcal{S}:X\rightarrow X,\quad \mathcal{S}(\mathbf{v}):=\mathbf{u}_{m}.
\end{equation*}%
We need now to show that $\mathcal{S}$ is compact. To this aim, we control
the time derivative of $\mathbf{u}_{m}$. In particular, multiplying %
\eqref{velgal} by $\mathrm{d} \left( a_{l}^{m}\right) /\mathrm{d}t$ and summing over $l$, we
have
\begin{align*}
\rho _{\ast }\Vert \partial _{t}\mathbf{u}_{m}\Vert _{L^{2}(\Omega )}^{2}&
\leq -(\rho (\phi _{m})(\mathbf{v}\cdot \nabla )\mathbf{u}_{m},\partial _{t}%
\mathbf{u}_{m})-(\nu (\phi _{m})D\mathbf{u}_{m},\nabla \partial _{t}\mathbf{u%
}_{m}) \\
& \quad +\frac{\rho _{1}-\rho _{2}}{2}((\nabla \mu _{m}\cdot \nabla )\mathbf{%
u}_{m},\partial _{t}\mathbf{u}_{m})-(\phi _{m}\nabla J\ast \phi
_{m},\partial _{t}\mathbf{u}_{m}).
\end{align*}%
Here, we have used that $-(\phi \nabla \mu _{m},\partial _{t}\mathbf{u}%
_{m})=-(\phi _{m}\nabla J\ast \phi _{m},\partial _{t}\mathbf{u}_{m})$ (cf. %
\eqref{RHS-NS}). By exploiting \eqref{sob} and the global bound $|\phi
_{m}|\leq 1\text{ almost everywhere in}\ \Omega \times (0,T)$, we
obtain
\begin{align*}
\rho _{\ast }\Vert \partial _{t}\mathbf{u}_{m}\Vert _{L^{2}(\Omega )}^{2}&
\leq \rho ^{\ast }\Vert \mathbf{v}\Vert _{L^{4}(\Omega )}\Vert \nabla
\mathbf{u}_{m}\Vert _{L^{4}(\Omega )}\Vert \partial _{t}\mathbf{u}_{m}\Vert
_{L^{2}(\Omega )}+\nu ^{\ast }\Vert D\mathbf{u}_{m}\Vert _{L^{2}(\Omega
)}\Vert \nabla \partial _{t}\mathbf{u}_{m}\Vert _{L^{2}(\Omega )} \\
& \quad + \left| \frac{\rho _{1}-\rho _{2}}{2}\right|\Vert \nabla \mu _{m}\Vert
_{L^{2}(\Omega )}\Vert \nabla \mathbf{u}_{m}\Vert _{L^{4}(\Omega )}\Vert
\partial _{t}\mathbf{u}_{m}\Vert _{L^{4}(\Omega )} \\
& \quad +\Vert \phi _{m}\Vert _{L^{\infty }(\Omega )}\Vert \nabla J\ast \phi
_{m}\Vert _{L^{2}(\Omega )}\Vert \partial _{t}\mathbf{u}_{m}\Vert
_{L^{2}(\Omega )} \\
& \leq C\Vert \nabla \mathbf{v}\Vert _{L^{2}(\Omega )}\Vert \mathbf{u}%
_{m}\Vert _{H^{2}(\Omega )}\Vert \partial _{t}\mathbf{u}_{m}\Vert
_{L^{2}(\Omega )}+C\Vert D\mathbf{u}_{m}\Vert _{L^{2}(\Omega )}\Vert \nabla
\partial _{t}\mathbf{u}_{m}\Vert _{L^{2}(\Omega )} \\
& \quad +C\Vert \nabla \mu _{m}\Vert _{L^{2}(\Omega )}\Vert \mathbf{u}%
_{m}\Vert _{H^{2}(\Omega )}\Vert \nabla \partial _{t}\mathbf{u}_{m}\Vert
_{L^{2}(\Omega )}+\Vert \nabla J\Vert _{L^{1}(\mathbb{R}^{2})}\Vert \phi
_{m}\Vert _{L^{2}(\Omega )}\Vert \partial _{t}\mathbf{u}_{m}\Vert
_{L^{2}(\Omega )} \\
& \leq C_{m}\Vert \mathbf{v}\Vert _{L^{2}(\Omega )}\Vert \mathbf{u}_{m}\Vert
_{L^{2}(\Omega )}\Vert \partial _{t}\mathbf{u}_{m}\Vert _{L^{2}(\Omega
)}+C_{m}\Vert \mathbf{u}_{m}\Vert _{L^{2}(\Omega )}\Vert \partial _{t}%
\mathbf{u}_{m}\Vert _{L^{2}(\Omega )} \\
& \quad +C_{m}\Vert \nabla \mu _{m}\Vert _{L^{2}(\Omega )}\Vert \mathbf{u}%
_{m}\Vert _{L^{2}(\Omega )}\Vert \partial _{t}\mathbf{u}_{m}\Vert
_{L^{2}(\Omega )}+\Vert \nabla J\Vert _{L^{1}(\mathbb{R}^{2})}|\Omega |^{%
\frac{1}{2}}\Vert \partial _{t}\mathbf{u}_{m}\Vert _{L^{2}(\Omega )} \\
& \leq C_{m}\Vert \partial _{t}\mathbf{u}_{m}\Vert _{L^{2}(\Omega )}\left(
M^{2}+M\left( 1+\Vert \nabla \mu _{m}\Vert _{L^{2}(\Omega )}\right) +\Vert
\nabla J\Vert _{L^{1}(\mathbb{R}^{2})}|\Omega |^{\frac{1}{2}}\right) \\
& \leq \frac{\rho _{\ast }}{2}\Vert \partial _{t}\mathbf{u}_{m}\Vert
_{L^{2}(\Omega )}^{2}+C_{m}\left( M^{4}+M^{2}\left( 1+\Vert \nabla \mu
_{m}\Vert _{L^{2}(\Omega )}^{2}\right) +\Vert \nabla J\Vert _{L^{1}(\mathbb{R%
}^{2})}^{2}|\Omega |\right) .
\end{align*}%
Then, integrating over $[0,T]$ and using \eqref{nablamu-L2-m}, we deduce
that
\begin{align*}
\int_{0}^{T}\Vert \partial _{t}\mathbf{u}_{m}(\tau )\Vert ^{2}\,\mathrm{d}%
\tau & \leq \frac{2}{\rho _{\ast }}\left[ C_{m}(M^{2}+M^{4})T+C_{m}M^{2}%
\left( C+\Vert \mathbf{v}\Vert _{L^{2}(0,T;L^{2}(\Omega ))}^{2}\right)
+C_{m}\Vert \nabla J\Vert _{L^{1}(\mathbb{R}^{2})}^{2}|\Omega |T\right] \\
& \leq \frac{2}{\rho _{\ast }}\left[ C_{m}(M^{2}+M^{4})T+C_{m}M^{2}\left(
C+M^{2}T\right) +C_{m}\Vert \nabla J\Vert _{L^{1}(\mathbb{R}%
^{2})}^{2}|\Omega |T\right] =:\widetilde{M}^{2},
\end{align*}%
namely
\begin{equation}
\Vert \partial _{t}\mathbf{u}_{m}\Vert _{L^{2}(0,T;{\mathbf{V}}%
_{m})}^{2}\leq \widetilde{M},  \label{tildeu}
\end{equation}%
%
%
%
%
%
%
%
%
%
%
%
%
%
%
%
%
%
%
%
%
%
%
%
%
%
%
%
%
%
%
%
%
%
%
%
%
%
%
%
%
%
%
%
%
%
Recalling that ${\mathbf{V}}_{m}$ is finite dimensional, the Aubin-Lions
Lemma entails
$
C([0,T];{\mathbf{V}}_{m})\cap H^{1}(0,T;{\mathbf{V}}_{m})\overset{c}{%
\hookrightarrow }C([0,T];{\mathbf{V}}_{m}).
$
Therefore, since $\mathcal{S}:X\rightarrow Y$, where $Y=\{\mathbf{u}\in
X:\Vert \partial _{t}\mathbf{u}_{m}\Vert _{L^{2}(0,T;{\mathbf{V}}%
_{m})}^{2}\leq \widetilde{M}\}$, it follows that the map $\mathcal{S}$ is
compact $($more precisely, $\overline{\mathcal{S}(X)}$ is compact
in $C([0,T];{\mathbf{V}}_{m}))$.

In order to complete our fixed point argument, we are left to show that $\mathcal{S}%
:X\rightarrow X$ is continuous. To this aim, we consider a sequence $\{%
\mathbf{v}_{n}\}_{n=1}^{\infty }\subset X$ such that $\mathbf{v}%
_{n}\rightarrow \mathbf{v}^{\star }$ in $C([0,T];{\mathbf{V}}_{m})$;
consequently, there exists a sequence $\{(\phi _{n},\mu
_{n})\}_{n=1}^{\infty }$ and $(\phi ^{\star },\mu ^{\star })$ that solve the
convective nonlocal Cahn-Hilliard equation \eqref{eq2}-\eqref{bound2}, where
$\mathbf{v}$ is replaced by $\mathbf{v}_{n}$ and $\mathbf{\mathit{v}}^{\star
}$, respectively. Following the uniqueness argument performed in the proof
of Theorem \ref{ExistCahn}, we obtain
\begin{align*}
& \frac{1}{2}\frac{\mathrm{d}}{\mathrm{d}t}
\left\Vert \phi _{n}-\phi ^{\star
}\right\Vert _{\ast }^{2}+\frac{3\alpha }{4}
\left\Vert \phi _{n}-\phi ^{\star }\right\Vert
_{L^{2}(\Omega )}^{2} \\
& \quad \leq C \left\Vert \phi _{n}-\phi ^{\star }\right\Vert _{\ast }^{2}+
\left( \phi_{n}(\mathbf{v}_{n}-\mathbf{v}^{\star }),\nabla \mathcal{N}(\phi _{n}-\phi
^{\star })\right) +(\mathbf{v}^{\star }(\phi _{n}-\phi ^{\star }),\nabla
\mathcal{N}(\phi _{n}-\phi ^{\star })),
\end{align*}%
Here we have used that $\overline{\phi _{n}}=\overline{\phi ^{\star }}=%
\overline{\phi _{0}}$. In light of $\mathbf{v}^{\star }\in C([0,T];{\mathbf{V%
}}_{m})$ and \eqref{sob}, we notice that
\begin{align*}
|(\mathbf{v}^{\star }(\phi _{n}-\phi ^{\star }),\nabla \mathcal{N}(\phi
_{n}-\phi ^{\star }))|& \leq C\Vert \mathbf{v}^{\star }\Vert _{L^{\infty
}(\Omega )}\Vert \phi \Vert _{L^{2}(\Omega )}
\left\Vert \nabla \mathcal{N}(\phi
_{n}-\phi ^{\star })\right\Vert _{L^{2}(\Omega )} \\
& \leq \frac{\alpha }{4}\Vert \phi \Vert _{L^{2}(\Omega
)}^{2}+C_{m}M^{2} \left\Vert \phi _{n}-\phi ^{\star }\right\Vert _{\ast }^{2}.
\end{align*}%
Since $|\phi _{n} |<1\text{ almost everywhere in}\ \Omega \times (0,T)$,
we also find
\begin{align*}
|\left( \phi _{n}(\mathbf{v}_{n}-\mathbf{v}^{\star }),\nabla \mathcal{N}%
(\phi _{n}-\phi ^{\star })\right) |
& \leq \Vert \phi _{n}\Vert _{L^{\infty
}(\Omega )}\Vert {\mathbf{v}_{n}-\mathbf{v}^{\star }}\Vert
_{L^{2}(\Omega )}
\left \Vert \nabla \mathcal{N}(\phi _{n}-\phi ^{\star }) \right\Vert
_{L^{2}(\Omega )} \\
& \leq C\Vert \phi _{n}-\phi ^{\star }\left\Vert _{\ast }^{2}+C\Vert {\mathbf{v}_{n}-\mathbf{v}^{\star }}\right\Vert _{L^{2}(\Omega )}^{2}.
\end{align*}%
Thus, we obtain
\begin{equation*}
\frac{1}{2}\frac{\mathrm{d}}{\mathrm{d}t}
\left\Vert \phi _{n}-\phi ^{\star }\right\Vert
_{\ast }^{2}+\frac{\alpha }{2}
\left\Vert \phi _{n}-\phi ^{\star }\right\Vert
_{L^{2}(\Omega )}^{2}
\leq C\left\Vert \phi _{n}-\phi ^{\star }\right\Vert _{\ast
}^{2}+C\left\Vert {\mathbf{v}_{n}-\mathbf{v}^{\star }}\right\Vert _{L^{2}(\Omega )}^{2}
\end{equation*}%
and the Gronwall Lemma yields
\begin{equation}
\left\Vert \phi _{n}-\phi ^{\star }\right\Vert _{L^{\infty }(0,T;H^{1}(\Omega )^{\prime})}^{2}
+\left\Vert \phi _{n}-\phi ^{\star }\right\Vert _{L^{2}(0,T;L^{2}(\Omega
))}^{2}\leq C\mathrm{e}^{CT}\left( T+T^{2}\right) \left\Vert {\mathbf{v}_{n}-%
\mathbf{v}^{\star }}\right\Vert _{C([0,T];{\mathbf{V}}_{m})}^{2}\underset{%
n\rightarrow \infty }{\longrightarrow }{0}.  \label{conv-cont}
\end{equation}%
On the other hand, recalling that $\{\mathbf{v}_{n}\}_{n}$ and $\mathbf{v}%
^{\star }$ belong to $X$, by the regularity (i) in Theorem \ref{ExistCahn}
(more precisely, \eqref{ST1}-\eqref{ST2}) we infer that%
\begin{equation}
\Vert \partial _{t}\phi _{n}\Vert _{L^{\infty }(0,T;H^{1}(\Omega )^{\prime
})}+\Vert \partial _{t}\phi _{n}\Vert _{L^{2}(0,T;L^{2}(\Omega ))}+\Vert \mu
_{n}\Vert _{L^{\infty }(0,T;H^{1}(\Omega ))}+\Vert \nabla \mu _{n}\Vert
_{L^{2}(0,T;H^{1}(\Omega ))}\leq C,  \label{e1}
\end{equation}%
\begin{equation}
\Vert \partial _{t}\phi ^{\star }\Vert _{L^{\infty }(0,T;H^{1}(\Omega
)^{\prime })}+\Vert \partial _{t}\phi ^{\star }\Vert _{L^{2}(0,T;L^{2}(\Omega
))}+\Vert \mu ^{\star }\Vert _{L^{\infty }(0,T;H^{1}(\Omega ))}+\Vert \nabla
\mu ^{\star }\Vert _{L^{2}(0,T;H^{1}(\Omega ))}\leq C.  \label{e2}
\end{equation}%
On account of the estimates \eqref{L1-est}, %
\eqref{Lp}, by repeating the argument used to obtain \eqref{a4}, %
\eqref{a6} and \eqref{a9}, we find that
\begin{equation}
\Vert \phi _{n}\Vert _{L^{\infty }(0,T;H^{1}(\Omega ))}+\Vert \partial
_{t}\mu _{n}\Vert _{L^{2}(0,T;H^{1}(\Omega )^{\prime })}+\Vert F^{\prime
}(\phi _{n})\Vert _{L^{\infty }(0,T;H^{1}(\Omega ))}+||F^{\prime \prime }(\phi _{n})\Vert _{L^{\infty }(0,T;L^{p}(\Omega ))}\leq C,  \label{e3}
\end{equation}%
\begin{equation}
\Vert \phi ^{\star }\Vert _{L^{\infty }(0,T;H^{1}(\Omega ))}+\Vert \partial
_{t}\mu ^{\star }\Vert _{L^{2}(0,T;H^{1}(\Omega )^{\prime })}+\Vert
F^{\prime }(\phi ^{\star })\Vert _{L^{\infty }(0,T;H^{1}(\Omega
))}+||F^{\prime \prime }(\phi ^{\star })\Vert _{L^{\infty }(0,T;L^{p}(\Omega
))}\leq C,  \label{e4}
\end{equation}%
for any $p\in \lbrack 2,\infty )$. Here, the $C$ depends
on $p$, but it is independent of $n$. Then, we first observe from Lebesgue's
interpolation, the global bound in $L^{\infty }(\Omega \times (0,T))$ of $%
\phi _{n}$ and $\phi ^{\star }$, and \eqref{conv-cont} that
\begin{equation}
\Vert \phi _{n}-\phi ^{\ast }\Vert _{L^{4}(0,T;L^{4}(\Omega ))}\underset{%
n\rightarrow \infty }{\longrightarrow }{0}.  \label{L4-n-star}
\end{equation}%
Furthermore, in light of the above estimates, the Aubin-Lions lemma
ensures that (up to subsequences) $\mu _{n}-\mu ^{\star }\rightarrow \mu
^{\infty }$ as $n\rightarrow \infty $ in $L^{2}(0,T;L^{2}(\Omega ))$. We
claim that $\mu ^{\infty }\equiv 0$. In fact, since $\phi _{n}\rightarrow
\phi ^{\star }$ (up to a subsequence) almost everywhere in $\Omega \times
(0,T)$, we deuce from \eqref{e3} and \eqref{e4} that $F^{\prime }(\phi
_{n})\rightharpoonup F^{\prime }(\phi ^{\star })$ weakly in $L^{2}(\Omega
\times (0,T))$. Also, it is easily seen that $J\ast \phi _{n}\rightarrow
J\ast \phi ^{\star }$ in $L^{2}(\Omega \times (0,T))$. Thus, we immediately
infer that $\mu ^{\infty }\equiv 0$. 
More precisely, we have
\begin{equation}
\Vert \mu _{n}-\mu ^{\star }\Vert _{L^{2}(0,T;L^{2}(\Omega ))}\underset{%
n\rightarrow \infty }{\longrightarrow }{0}.  \label{conv5}
\end{equation}

We now define $\mathbf{u}_{n}=\mathcal{S}(\mathbf{v}_{n})\in Y$, for any $%
n\in \mathbb{N}$, and $\mathbf{u}^{\star }=\mathcal{S}(\mathbf{v}^{\star
})\in Y$. We set $\mathbf{u}=\mathbf{u}_{n}-\mathbf{u}^{\star }$, $%
\Phi =\phi _{n}-\phi ^{\star }$, $\mathbf{V}=\mathbf{v}_{n}-\mathbf{v}%
^{\star }$, $\Theta =\mu _{n}-\mu ^{\star }$, and we observe that
\begin{align*}
& \left( \rho (\phi _{n})\partial _{t}\mathbf{u},\mathbf{w}\right) +\left(
(\rho (\phi _{n})-\rho (\phi ^{\star }))\partial _{t}\mathbf{u}^{\star },%
\mathbf{w}\right) +(\rho (\phi _{n})(\mathbf{v}_{n}\cdot \nabla )\mathbf{u}%
_{n}-\rho (\phi ^{\star })(\mathbf{v}^{\star }\cdot \nabla )\mathbf{u}%
^{\star },\mathbf{w})+\left( \nu (\phi _{n})D\mathbf{u},\nabla \mathbf{w}%
\right) \\
& +\left( (\nu (\phi _{n})-\nu (\phi ^{\star }))D\mathbf{u}^{\star },\nabla
\mathbf{w}\right) -\frac{\rho _{1}-\rho _{2}}{2}((\nabla \mu _{n}\cdot
\nabla )\mathbf{u}_{n}-(\nabla \mu ^{\star }\cdot \nabla )\mathbf{u}^{\star
},\mathbf{w})  =(\mu _{n}\nabla \phi _{n}-\mu ^{\star }\nabla \phi ^{\star },\mathbf{w}),
\end{align*}%
for all $\mathbf{w}\in {\mathbf{V}}_{m}$ and in $[0,T]$.
Choosing then $\mathbf{w}=\mathbf{u}$, we obtain
\begin{align*}
& \frac{1}{2}\frac{\mathrm{d}}{\mathrm{d}t}\int_{\Omega }\rho (\phi _{n})|%
\mathbf{u}|^{2}\,\mathrm{d}x+\int_{\Omega }\nu (\phi _{n})|D\mathbf{u}|^{2}\,%
\mathrm{d}x \\
& =\frac{\rho _{1}-\rho _{2}}{4}\int_{\Omega }\partial _{t}\phi _{n}|\mathbf{%
u}|^{2}\,\mathrm{d}x-\frac{\rho _{1}-\rho _{2}}{2}\int_{\Omega }\Phi
(\partial _{t}\mathbf{u}^{\star }\cdot \mathbf{u})\,\mathrm{d}x \\
& \quad -\int_{\Omega }(\rho (\phi _{n})(\mathbf{v}_{n}\cdot \nabla )\mathbf{%
u}_{n}-\rho (\phi ^{\star })(\mathbf{v}^{\star }\cdot \nabla )\mathbf{u}%
^{\star })\cdot \mathbf{u}\,\mathrm{d}x-\int_{\Omega }(\nu (\phi _{n})-\nu
(\phi ^{\star }))D\mathbf{u}^{\star }:\nabla \mathbf{u}\,\mathrm{d}x \\
& \quad +\frac{\rho _{1}-\rho _{2}}{2}\int_{\Omega }((\nabla \mu _{n}\cdot
\nabla )\mathbf{u}_{n}-(\nabla \mu ^{\star }\cdot \nabla )\mathbf{u}^{\star
})\cdot \mathbf{u}\,\mathrm{d}x+\int_{\Omega }(\mu _{n}\nabla \phi _{n}-\mu
^{\star }\nabla \phi ^{\star })\cdot \mathbf{u}\,\mathrm{d}x.
\end{align*}%
Thanks to the embedding $H_{0,\sigma }^{1}(\Omega )\hookrightarrow
L^{4}(\Omega ;\mathbb{R}^{2})$, by exploiting \eqref{sob}, we have
\begin{equation*}
\left\vert \frac{\rho _{1}-\rho _{2}}{4}\int_{\Omega }\partial _{t}\phi _{n}|%
\mathbf{u}|^{2}\,\mathrm{d}x\right\vert \leq C\Vert \partial _{t}\phi
_{n}\Vert _{L^{2}(\Omega )}\Vert \mathbf{u}\Vert _{L^{4}(\Omega )}^{2}\leq
C_{m}\Vert \partial _{t}\phi _{n}\Vert _{L^{2}(\Omega )}\Vert \mathbf{u}%
\Vert _{L^{2}(\Omega )}^{2},
\end{equation*}%
and
\begin{align*}
\left\vert \frac{\rho _{1}-\rho _{2}}{2}\int_{\Omega }\Phi (\partial _{t}%
\mathbf{u}^{\star }\cdot \mathbf{u})\,\mathrm{d}x\right\vert & \leq C\Vert
\Phi \Vert _{L^{4}(\Omega )}\Vert \partial _{t}\mathbf{u}^{\star }\Vert
_{L^{2}(\Omega )}\Vert \mathbf{u}\Vert _{L^{4}(\Omega )} \\
& \leq C_{m}\Vert \partial _{t}\mathbf{u}^{\star }\Vert _{L^{2}(\Omega
)}^{2}\Vert \mathbf{u}\Vert _{L^{2}(\Omega )}^{2}+C_{m}\Vert \Phi \Vert
_{L^{4}(\Omega )}^{2}.
\end{align*}%
Similarly, recalling that $\mathbf{v}_{n},\mathbf{v}^{\star },\mathbf{u}_{n},%
\mathbf{u}^{\star }\in X$,
\begin{align*}
& -\int_{\Omega }(\rho (\phi _{n})(\mathbf{v}_{n}\cdot \nabla )\mathbf{u}%
_{n}-\rho (\phi ^{\star })(\mathbf{v}^{\star }\cdot \nabla )\mathbf{u}%
^{\star })\cdot \mathbf{u}\,\mathrm{d}x \\
& =-\frac{\rho _{1}-\rho _{2}}{2}\int_{\Omega }\Phi ((\mathbf{v}_{n}\cdot
\nabla )\mathbf{u}_{n})\cdot \mathbf{u}\,\mathrm{d}x-\int_{\Omega }\rho
(\phi ^{\star })(\mathbf{v}\cdot \nabla )\mathbf{u}_{n})\cdot \mathbf{u}\,%
\mathrm{d}x  -\int_{\Omega }\rho (\phi ^{\star })(\mathbf{v}^{\star }\cdot \nabla )%
\mathbf{u})\cdot \mathbf{u}\,\mathrm{d}x \\
& \leq C\Vert \Phi \Vert _{L^{4}(\Omega )}\Vert \mathbf{v}_{n}\Vert
_{L^{\infty }(\Omega )}\Vert \nabla \mathbf{u}_{n}\Vert _{L^{2}(\Omega
)}\Vert \mathbf{u}\Vert _{L^{4}(\Omega )}+C\Vert \mathbf{v}\Vert
_{L^{4}(\Omega )}\Vert \nabla \mathbf{u}_{n}\Vert _{L^{2}(\Omega )}\Vert
\mathbf{u}\Vert _{L^{4}(\Omega )} \\
& \quad +C\Vert \mathbf{v}^{\star }\Vert _{L^{4}(\Omega )}\Vert \nabla
\mathbf{u}\Vert _{L^{2}(\Omega )}\Vert \mathbf{u}\Vert _{L^{4}(\Omega )} \\
& \leq C_{m}M^{2}\Vert \Phi \Vert _{L^{4}(\Omega )}\ \Vert \mathbf{u}\Vert
_{L^{2}(\Omega )}+C_{m}M\Vert \mathbf{v}\Vert _{L^{2}(\Omega )}\Vert \mathbf{%
u}\Vert _{L^{2}(\Omega )}+C_{m}M\Vert \mathbf{u}\Vert _{L^{2}(\Omega )}^{2}
\\
& \leq C_{m}\left( 1+M^{2}\right) \Vert \mathbf{u}\Vert _{L^{2}(\Omega
)}^{2}+C_{m}M^{2}\left( \Vert \Phi \Vert _{L^{4}(\Omega )}^{2}+\Vert \mathbf{%
V}\Vert _{L^{2}(\Omega )}^{2}\right) .
\end{align*}%
Since $\nu \in W^{1,\infty }(\mathbb{R})$, we get%
\begin{align*}
-\int_{\Omega }(\nu (\phi _{n})-\nu (\phi ^{\star }))D\mathbf{u}^{\star }
:\nabla \mathbf{u}\,\mathrm{d}x &\leq C\Vert \Phi \Vert _{L^{4}(\Omega )}\Vert
D\mathbf{u}^{\star }\Vert _{L^{2}(\Omega )}\Vert \nabla \mathbf{u}\Vert
_{L^{4}(\Omega )} \\
& \leq C_{m}M^{2}\Vert \mathbf{u}\Vert _{L^{2}(\Omega )}^{2}+C_{m}\Vert \Phi
\Vert _{L^{4}(\Omega )}^{2}.
\end{align*}%
On the other hand, observing that
\begin{align*}
& \frac{\rho _{1}-\rho _{2}}{2}\int_{\Omega }((\nabla \mu _{n}\cdot \nabla )%
\mathbf{u}_{n}-(\nabla \mu ^{\star }\cdot \nabla )\mathbf{u}^{\star })\cdot
\mathbf{u}\,\mathrm{d}x \\
& =\frac{\rho _{1}-\rho _{2}}{2}\int_{\Omega }((\nabla \mu _{n}\cdot \nabla )%
\mathbf{u}+(\nabla \Theta \cdot \nabla )\mathbf{u}^{\star })\cdot \mathbf{u}%
\,\mathrm{d}x \\
& =-\frac{\rho _{1}-\rho _{2}}{2}\left( (\mu _{n}\nabla \mathbf{u},\nabla
\mathbf{u})+(\mu _{n}\mathbf{u},\Delta \mathbf{u})\right) -\frac{\rho
_{1}-\rho _{2}}{2}\left( (\Theta \nabla \mathbf{u}^{\star },\nabla \mathbf{u}%
)+(\Theta \mathbf{u},\Delta \mathbf{u}^{\star })\right) ,
\end{align*}%
%
%
%
%
%
%
%
%
%
%
%
%
%
%
%
%
%
%
%
%
%
%
%
%
%
%
%
%
%
%
%
%
%
%
%
%
%
%
%
%
%
%
%
%
%
%
%
%
%
%
%
%
%
%
%
%
%
%
%
%
%
%
%
%
%
%
%
%
%
%
%
we infer from \eqref{sob} and \eqref{e3} that
\begin{align*}
& \left\vert \frac{\rho _{1}-\rho _{2}}{2}\int_{\Omega }((\nabla \mu
_{n}\cdot \nabla )\mathbf{u}_{n}-(\nabla \mu ^{\star }\cdot \nabla )\mathbf{u%
}^{\star })\cdot \mathbf{u}\,\mathrm{d}x\right\vert \\
& \leq C\Vert \mu _{n}\Vert _{L^{2}(\Omega )}\Vert \nabla \mathbf{u}\Vert
_{L^{4}(\Omega )}^{2}+C\Vert \mu _{n}\Vert _{L^{2}(\Omega )}\Vert \mathbf{u}%
\Vert _{L^{\infty }}\Vert \Delta \mathbf{u}\Vert _{L^{2}(\Omega )} \\
& \quad +C\Vert \Theta \Vert _{L^{2}(\Omega )}\Vert \nabla \mathbf{u}^{\star
}\Vert _{L^{4}(\Omega )}\Vert \nabla \mathbf{u}\Vert _{L^{4}(\Omega
)}+C\Vert \Theta \Vert _{L^{2}(\Omega )}\Vert \mathbf{u}\Vert _{L^{\infty
}(\Omega )}\Vert \Delta \mathbf{u}^{\star }\Vert _{L^{2}(\Omega )} \\
& \leq C_{m}\left( 1+M^{2}\right) \Vert \mathbf{u}\Vert _{L^{2}(\Omega
)}^{2}+C_{m}\Vert \Theta \Vert _{L^{2}(\Omega )}^{2}.
\end{align*}%
Finally, using again \eqref{sob} and \eqref{e3}-\eqref{e4}, we get
\begin{align*}
\int_{\Omega }(\mu _{n}\nabla \phi _{n}-\mu ^{\star }\nabla \phi ^{\star
})\cdot \mathbf{u}\,\mathrm{d}x& =\int_{\Omega }(-\Phi \nabla \mu
_{n}+\Theta \nabla \phi ^{\star })\cdot \mathbf{u}\,\mathrm{d}x \\
& \leq \Vert \nabla \mu _{n}\Vert _{L^{2}(\Omega )}\Vert \Phi \Vert
_{L^{4}(\Omega )}\Vert \mathbf{u}\Vert _{L^{4}(\Omega )}+\Vert \Theta \Vert
_{L^{2}(\Omega )}\Vert \nabla \phi ^{\star }\Vert _{L^{2}(\Omega )}\Vert
\mathbf{u}\Vert _{L^{\infty }(\Omega )} \\
& \leq C_{m}\Vert \mathbf{u}\Vert _{L^{2}(\Omega )}^{2}+C_{m}
\left(\Vert \Theta
\Vert _{L^{2}(\Omega )}^{2}+\Vert \Phi \Vert _{L^{4}(\Omega )}^{2}\right).
\end{align*}%
Combining the above inequalities and recalling that $\rho \geq \rho _{\ast }$%
, we are thus led to the differential inequality
\begin{equation*}
\frac{\mathrm{d}}{\mathrm{d}t}\int_{\Omega }\rho (\phi _{n})|\mathbf{U}%
|^{2}\,\mathrm{d}x\leq H_{1}(t)\int_{\Omega }\rho (\phi _{n})|\mathbf{%
\mathit{U}}|^{2}\,\mathrm{d}x+H_{2}(t),
\end{equation*}%
where
\begin{equation*}
H_{1} :=C_{m}\left( 1+\Vert \partial _{t}\phi _{n}\Vert _{L^{2}(\Omega
)}^{2}+\Vert \partial _{t}\mathbf{u}^{\star }\Vert _{L^{2}(\Omega
)}^{2}\right) ,\quad
H_{2} :=C_{m}\left( 1+\Vert \Phi \Vert _{L^{4}(\Omega )}^{2}+\Vert \Theta
\Vert _{L^{2}(\Omega )}^{2}+\Vert \mathbf{v}\Vert _{L^{2}(\Omega
)}^{2}\right) .
\end{equation*}%
Hence, the Gronwall lemma entails
\begin{equation}
\max_{t\in \lbrack 0,T]}\Vert \mathbf{u}(t)\Vert _{L^{2}(\Omega )}^{2}\leq
\frac{1}{\rho _{\ast }}\mathrm{e}^{\int_{0}^{T}H_{1}(\tau )\,\mathrm{d}\tau
}\int_{0}^{T}H_{2}(\tau )\,\mathrm{d}\tau .  \label{Gron-cont}
\end{equation}%
Note that $H_{1}\in L^{1}(0,T)$ and $H_{2}\in L^{1}(0,T)$, thanks to \eqref{tildeu} and \eqref{e1}-\eqref{e4}.
In addition, in light of $%
\mathbf{v}_{n}\rightarrow \widetilde{\mathbf{v}}$ in $C([0,T];{\mathbf{V}}%
_{m})$ and \eqref{L4-n-star}-\eqref{conv5}, we deduce that from %
\eqref{Gron-cont} that $\mathbf{u}_{n}\rightarrow \mathbf{u}^{\star }$ in $%
C([0,T];{\mathbf{V}}_{m})$, implying that the map $\mathcal{S}$ is
continuous.

In conclusion, we can apply the Schauder fixed point
theorem to $\mathcal{S}$. This gives the
existence of an approximate solution $(\mathbf{u}_{m},\phi _{m})$ in $[0,T]$
satisfying \eqref{rg}-\eqref{init}. \medskip

\noindent \textbf{Uniform estimates independent of the approximation
parameter.} Integrating \eqref{eqapprox}$_{1}$ over $\Omega $, we find
\begin{equation*}
\overline{\phi _{m}}(t)=\frac{1}{|\Omega |}\int_{\Omega }\phi _{m}(t)\,%
\mathrm{d}x=\frac{1}{|\Omega |}\int_{\Omega }\phi _{0}\,\mathrm{d}x,\quad
\forall \,t\in \lbrack 0,T].
\end{equation*}%
Taking $\mathbf{w}=\mathbf{u}_{m}$ in \eqref{vel} and arguing as above, we
get
\begin{equation*}
\frac{\mathrm{d}}{\mathrm{d}t}\int_{\Omega }\frac{1}{2}\rho (\phi _{m})|%
\mathbf{u}_{m}|^{2}\,\mathrm{d}x+\int_{\Omega }\nu (\phi _{m})|D\mathbf{u}%
_{m}|^{2}\,\mathrm{d}x=\int_{\Omega }\mu _{m}\nabla \phi _{m}\cdot \mathbf{u}%
_{m}\,\mathrm{d}x.
\end{equation*}%
Let us recall that $\phi _{m}$ satisfies the energy identity \eqref{EE-nCH},
i.e.,
\begin{equation*}
\mathcal{E}_{\mathrm{nloc}}(\phi _{m}(t))+\int_{0}^{t}\Vert \nabla \mu
_{m}(\tau )\Vert _{L^{2}(\Omega )}^{2}\,\mathrm{d}\tau
+\int_{0}^{t}\int_{\Omega }\phi _{m}\,\mathbf{u}_{m}\cdot \nabla \mu _{m}\,%
\mathrm{d}x\,\mathrm{d}\tau =\mathcal{E}_{\mathrm{nloc}}((\phi _{0}),\quad
\forall \,t\in \lbrack 0,T].
\end{equation*}%
Therefore, we have
\begin{equation}
\frac{\mathrm{d}}{\mathrm{d}t}{E}(\mathbf{u}_{m},\phi _{m})+\int_{\Omega
}\nu (\phi _{m})|D\mathbf{u}_{m}|^{2}\,\mathrm{d}x+\int_{\Omega }|\nabla \mu
_{m}|^{2}\,\mathrm{d}x=0,  \label{ene}
\end{equation}%
where
\begin{equation*}
{E}(\mathbf{u}_{m},\phi _{m})=\int_{\Omega }\frac{1}{2}\rho (\phi _{m})|%
\mathbf{u}_{m}|^{2}\,\mathrm{d}x+\mathcal{E}_{\mathrm{nloc}}(\phi _{m}).
\end{equation*}%
Notice that, being $|\phi _{m}|<1\text{ almost everywhere in}\ \Omega
\times (0,T)$, $\mathcal{E}_{\mathrm{nloc}}(\phi _{m})\geq -C_{e}$ almost
everywhere in $(0,T)$, where $C_{e}$ is independent of $m$. Then, we can
define $\widehat{E}(\mathbf{u}_{m},\phi _{m})=E(\mathbf{u}_{m},\phi
_{m})+C_{e}\geq 0$. We now integrate \eqref{ene} with respect to time in $[0,T]$ and we obtain
\begin{equation}
\widehat{E}(\mathbf{u}_{m}(t),\phi _{m}(t))+\int_{0}^{t}\int_{\Omega }\nu
(\phi _{m})|D\mathbf{u}_{m}(\tau )|^{2}\,\mathrm{d}x\,\mathrm{d}\tau
+\int_{0}^{t}\int_{\Omega }|\nabla \mu _{m}(\tau )|^{2}\,\mathrm{d}x\,%
\mathrm{d}\tau =\widehat{E}(\mathbb{P}_{m}\mathbf{u}_{0},\phi _{0}).
\label{eneg1}
\end{equation}%
By the properties of $\mathbb{P}_{m}$, we immediately deduce that
\begin{equation*}
\widehat{E}(\mathbb{P}_{m}\mathbf{u}_{0},\phi _{0})\leq C_{e}+\frac{\rho
_{\ast }}{2}\Vert \mathbf{u}_{0}\Vert _{L^{2}(\Omega )}^{2}+\mathcal{E}_{%
\mathrm{nloc}}(\phi _{0}).
\end{equation*}%
Therefore, we conclude that
\begin{align}
\Vert \mathbf{u}_{m}\Vert _{L^{\infty }(0,T;L^{2}(\Omega ))}+\Vert \mathbf{%
u}_{m}\Vert _{L^{2}(0,T;H^{1}(\Omega ))}\leq C_{E},  \quad 
 \Vert \nabla \mu _{m}\Vert _{L^{2}(0,T;H)}\leq C_{E},  \label{mu1}
\end{align}%
where $C_{E}$ is independent of $m$.
Owing to \eqref{mu1}, the two-dimensional continuous embedding of $L^{\infty
}(0,T;L_{\sigma }^{2}(\Omega ))\cap L^{2}(0,T;H_{0,\sigma }^{1}(\Omega ))$
into $L^{4}(0,T;L_{\sigma }^{4}(\Omega ))$ and the assumptions on $\phi _{0}$%
, we can apply (i) of Theorem \ref{ExistCahn}. In
particular, \eqref{ST1}-\eqref{ST3} entails that
\begin{equation}
\begin{cases}
 \Vert \phi _{m}\Vert _{L^{\infty }(0,T;H^{1}(\Omega ))}+\Vert F^{\prime
}(\phi _{m})\Vert _{L^{\infty }(0,T;H^{1}(\Omega ))}\leq C, \\
 \Vert \phi _{m}\Vert _{L^{q}(0,T;W^{1,p}(\Omega ))}\leq C_{p},\quad q=%
\frac{2p}{p-2},\quad \forall\,p\in (2,\infty ), \\
 \Vert \partial _{t}\phi _{m}\Vert _{L^{\infty }(0,T;H^{1}(\Omega )^{\prime
})}+\Vert \partial _{t}\phi _{m}\Vert _{L^{2}(0,T;L^{2}(\Omega ))}\leq C, \\
 \Vert \mu _{m}\Vert _{L^{\infty }(0,T;H^{1}(\Omega ))}+\Vert \mu _{m}\Vert
_{L^{2}(0,T;H^{2}(\Omega ))}+\Vert \mu _{m}\Vert _{H^{1}(0,T;H^{1}(\Omega
)^{\prime })}\leq C, \\
 ||F^{\prime \prime }(\phi _{m})\Vert _{L^{\infty }(0,T;L^{p}(\Omega
))}\leq C_{p},\quad \forall\, p\in \lbrack 2,\infty ).
\end{cases}
\label{5}
\end{equation}%
Next, taking $\mathbf{w}=\partial _{t}\mathbf{u}_{m}$ in \eqref{vel}, we find
\begin{equation}
\begin{split}
& \frac{1}{2}\frac{\mathrm{d}}{\mathrm{d}t}\int_{\Omega }\nu (\phi _{m})|D%
\mathbf{u}_{m}|^{2}\,\mathrm{d}x+\int_{\Omega }\rho (\phi _{m})|\partial _{t}%
\mathbf{u}_{m}|^{2}\,\mathrm{d}x \\
& =-\int_{\Omega }\rho (\phi _{m})((\mathbf{u}_{m}\cdot \nabla )\mathbf{u}%
_{m})\cdot \partial _{t}\mathbf{u}_{m}\,\mathrm{d}x+\int_{\Omega }\nu
^{\prime }(\phi _{m})\partial _{t}\phi _{m}|D\mathbf{u}_{m}|^{2}\,\mathrm{d}x
\\
& \quad +\frac{\rho _{1}-\rho _{2}}{2}\int_{\Omega }((\nabla \mu _{m}\cdot
\nabla )\mathbf{u}_{m})\cdot \partial _{t}\mathbf{u}_{m}\,\mathrm{d}%
x+\int_{\Omega }\mu _{m}\nabla \phi _{m}\cdot \partial _{t}\mathbf{u}_{m}\,%
\mathrm{d}x.
\end{split}
\label{dtu}
\end{equation}%
In addition, following \cite{GMT2019, AGG2d}, we can choose $\mathbf{w}=\mathbf{A}\mathbf{u}_{m}$ in \eqref{vel}, obtaining
\begin{align*}
-\frac{1}{2}(\nu (\phi _{m})\Delta \mathbf{u}_{m},\mathbf{A}\mathbf{u}_{m})&
=-(\rho (\phi _{m})\partial _{t}\mathbf{u}_{m},\mathbf{A}\mathbf{u}%
_{m})-(\rho (\phi _{m})(\mathbf{u}_{m}\cdot \nabla )\mathbf{A}\mathbf{u}%
_{m})+\frac{\rho _{1}-\rho _{2}}{2}((\nabla \mu _{m}\cdot \nabla )\mathbf{u}%
_{m},\mathbf{A}\mathbf{u}_{m}) \\
& \quad +(\mu _{m}\nabla \phi _{m},\mathbf{A}\mathbf{u}_{m})+(\nu ^{\prime
}(\phi _{m})D\mathbf{u}_{m}\nabla \phi _{m},\mathbf{A}\mathbf{u}_{m}).
\end{align*}%
By the regularity theory of the Stokes operator, there exists $\pi _{m}\in
C([0,T];H^{1})$ such that $-\Delta \mathbf{u}_{m}+\nabla \pi _{m}=\mathbf{A}%
\mathbf{u}_{m}$ almost everywhere in $\Omega \times (0,T)$. Furthermore,
Lemma \ref{press} implies that
\begin{equation*}
\Vert \pi _{m}\Vert _{L^{4}(\Omega )}\leq C\Vert \nabla \mathbf{u}_{m}\Vert
_{L^{2}(\Omega )}^{\frac{1}{2}}\Vert \mathbf{A}\mathbf{u}_{m}\Vert
_{L^{2}(\Omega )}^{\frac{1}{2}}.
\end{equation*}%
Since $(\nu (\phi _{m})\nabla \pi _{m},%
\mathbf{A}\mathbf{u}_{m})=-(\nu ^{\prime }(\phi _{m})\pi _{m}\nabla \phi
_{m},\mathbf{A}\mathbf{u}_{m}),$ we arrive at
\begin{equation}
\begin{split}
\frac{1}{2}(\nu (\phi _{m})\mathbf{A}\mathbf{u}_{m},\mathbf{A}\mathbf{u}%
_{m})& =-(\rho (\phi _{m})\partial _{t}\mathbf{u}_{m},\mathbf{A}\mathbf{u}%
_{m})-(\rho (\phi _{m})(\mathbf{u}_{m}\cdot \nabla )\mathbf{u}_{m},\mathbf{A}%
\mathbf{u}_{m}) \\
& \quad +\frac{\rho _{1}-\rho _{2}}{2}((\nabla \mu _{m}\cdot \nabla )\mathbf{%
u}_{m},\mathbf{A}\mathbf{u}_{m})+(\mu _{m}\nabla \phi _{m},\mathbf{A}\mathbf{%
u}_{m}) \\
& \quad +(\nu ^{\prime }(\phi _{m})D\mathbf{u}_{m}\nabla \phi _{m},\mathbf{A}%
\mathbf{u}_{m})-\frac{1}{2}(\nu ^{\prime }(\phi _{m})\pi _{m}\nabla \phi
_{m},\mathbf{A}\mathbf{u}_{m}).
\end{split}
\label{Au}
\end{equation}%
Let us now estimate the terms on the right-hand side in \eqref{dtu}
and \eqref{Au}. Set $\omega _{1}$ a positive constant whose value will be
determined later on. By using \eqref{LADY} and \eqref{mu1}, we have
\begin{align*}
\left\vert \int_{\Omega }\rho (\phi _{m})((\mathbf{u}_{m}\cdot \nabla )%
\mathbf{u}_{m})\cdot \partial _{t}\mathbf{u}_{m}\,\mathrm{d}x\right\vert &
\leq \rho ^{\ast }\Vert \mathbf{u}_{m}\Vert _{L^{4}(\Omega )}\Vert \nabla
\mathbf{u}_{m}\Vert _{L^{4}(\Omega )}\Vert \partial _{t}\mathbf{u}_{m}\Vert
_{L^{2}(\Omega )} \\
& \leq \frac{\rho _{\ast }}{8}\Vert \partial _{t}\mathbf{u}_{m}\Vert
_{L^{2}(\Omega )}^{2}+\frac{\nu _{\ast }\omega _{1}}{32}\Vert \mathbf{A}%
\mathbf{u}_{m}\Vert _{L^{2}(\Omega )}^{2}+C\Vert D\mathbf{u}_{m}\Vert
_{L^{2}(\Omega )}^{4}
\end{align*}%
and
\begin{align*}
\left\vert \int_{\Omega }\nu ^{\prime }(\phi _{m})\partial _{t}\phi _{m}|D%
\mathbf{u}_{m}|^{2}\,\mathrm{d}x\right\vert & \leq C\Vert \partial _{t}\phi
_{m}\Vert _{L^{2}(\Omega )}\Vert D\mathbf{u}_{m}\Vert _{L^{4}(\Omega )}^{2}
\\
& \leq \frac{\nu _{\ast }\omega _{1}}{32}\Vert \mathbf{A}\mathbf{u}_{m}\Vert
_{L^{2}(\Omega )}^{2}+C\Vert \partial _{t}\phi _{m}\Vert _{L^{2}(\Omega
)}^{2}\Vert D\mathbf{u}_{m}\Vert _{L^{2}(\Omega )}^{2}.
\end{align*}%
Exploiting \eqref{LADY} once again, together with \eqref{5}, we obtain
\begin{align*}
&\left\vert \frac{\rho _{1}-\rho _{2}}{2}\int_{\Omega }((\nabla \mu
_{m}\cdot \nabla )\mathbf{u}_{m})\cdot \partial _{t}\mathbf{u}_{m}\,\mathrm{d%
}x\right\vert\\
& \leq C\Vert \nabla \mu _{m}\Vert _{L^{4}(\Omega )}\Vert \nabla \mathbf{u}%
_{m}\Vert _{L^{4}(\Omega )}\Vert \partial _{t}\mathbf{u}_{m}\Vert
_{L^{2}(\Omega )} \\
& \leq C\Vert \nabla \mu _{m}\Vert _{L^{2}(\Omega )}^{\frac{1}{2}}\Vert
\nabla \mu _{m}\Vert _{H^{1}(\Omega )}^{\frac{1}{2}}\Vert D\mathbf{u}%
_{m}\Vert _{L^{2}(\Omega )}^{\frac{1}{2}}\Vert \mathbf{A}\mathbf{u}_{m}\Vert
_{L^{2}(\Omega )}^{\frac{1}{2}}\Vert \partial _{t}\mathbf{u}_{m}\Vert
_{L^{2}(\Omega )} \\
& \leq \frac{\rho _{\ast }}{8}\Vert \partial _{t}\mathbf{u}_{m}\Vert
_{L^{2}(\Omega )}^{2}+\frac{\nu _{\ast }\omega _{1}}{32}\Vert \mathbf{A}%
\mathbf{u}_{m}\Vert _{L^{2}(\Omega )}^{2}+C\Vert \nabla \mu _{m}\Vert
_{H^{1}(\Omega )}^{2}\Vert D\mathbf{u}_{m}\Vert _{L^{2}(\Omega )}^{2}
\end{align*}%
and
\begin{equation*}
\left\vert \int_{\Omega }\mu _{m}\nabla \phi _{m}\cdot \partial _{t}\mathbf{u%
}_{m}\,\mathrm{d}x\right\vert =\left\vert \int_{\Omega }\phi _{m}\nabla \mu
_{m}\cdot \partial _{t}\mathbf{u}_{m}\,\mathrm{d}x\right\vert \leq \frac{%
\rho _{\ast }}{8}\Vert \partial _{t}\mathbf{u}_{m}\Vert ^{2}+C\Vert \nabla
\mu _{m}\Vert _{L^{2}(\Omega )}^{2}.
\end{equation*}%
Arguing as in the proof of \cite[Section 4]{AGG2d} related to the terms $%
I_{8}$ and $I_{9}$, we find
\begin{equation*}
\left\vert (\rho (\phi _{m})\partial _{t}\mathbf{u}_{m},\mathbf{A}\mathbf{u}%
_{m})\right\vert \leq \frac{2\omega _{1}(\rho ^{\ast })^{2}}{\rho _{\ast }}%
\Vert \mathbf{A}\mathbf{u}_{m}\Vert _{L^{2}(\Omega )}^{2}+\frac{\rho _{\ast }%
}{8\omega _{1}}\Vert \partial _{t}\mathbf{u}_{m}\Vert _{L^{2}(\Omega )}^{2}
\end{equation*}%
and
\begin{equation*}
\left\vert (\rho (\phi _{m})(\mathbf{u}_{m}\cdot \nabla )\mathbf{u}_{m},%
\mathbf{A}\mathbf{u}_{m})\right\vert \leq \frac{\nu _{\ast }}{32}\Vert
\mathbf{A}\mathbf{u}_{m}\Vert _{L^{2}(\Omega )}^{2}+C\Vert D\mathbf{u}%
_{m}\Vert _{L^{2}(\Omega )}^{4}.
\end{equation*}%
Proceeding as above, we get
\begin{align*}
\left\vert \frac{\rho _{1}-\rho _{2}}{2}((\nabla \mu _{m}\cdot \nabla )%
\mathbf{u}_{m},\mathbf{A}\mathbf{u}_{m})\right\vert & \leq C\Vert \nabla \mu
_{m}\Vert _{L^{4}(\Omega )}\Vert \nabla \mathbf{u}_{m}\Vert _{L^{4}(\Omega
)}\Vert \mathbf{A}\mathbf{u}_{m}\Vert _{L^{2}(\Omega )} \\
& \leq C\Vert \nabla \mu _{m}\Vert _{L^{2}(\Omega )}^{\frac{1}{2}}\Vert
\nabla \mu _{m}\Vert _{H^{1}(\Omega )}^{\frac{1}{2}}\Vert D\mathbf{\mathit{u}%
}_{m}\Vert _{L^{2}(\Omega )}^{\frac{1}{2}}\Vert \mathbf{A}\mathbf{u}%
_{m}\Vert _{L^{2}(\Omega )}^{\frac{3}{2}} \\
& \leq \frac{\nu _{\ast }}{32}\Vert \mathbf{A}\mathbf{u}_{m}\Vert
_{L^{2}(\Omega )}^{2}+C\Vert \nabla \mu _{m}\Vert _{H^{1}(\Omega )}^{2}\Vert
D\mathbf{u}_{m}\Vert _{L^{2}(\Omega )}^{2}
\end{align*}%
and
\begin{equation*}
\left\vert (\mu _{m}\nabla \phi _{m},\mathbf{A}\mathbf{u}_{m})\right\vert
=\left\vert (\phi _{m}\nabla \mu _{m},\mathbf{A}\mathbf{u}_{m})\right\vert
\leq \frac{\nu _{\ast }}{32}\Vert \mathbf{A}\mathbf{u}_{m}\Vert
_{L^{2}(\Omega )}^{2}+C\Vert \nabla \mu _{m}\Vert _{L^{2}(\Omega )}^{2}.
\end{equation*}%
By $\nu ^{\prime }\in W^{1,\infty }(\mathbb{R})$ and \eqref{5}, it follows
that
\begin{align*}
\left\vert (\nu ^{\prime }(\phi _{m})D\mathbf{u}_{m}\nabla \phi _{m},\mathbf{%
A}\mathbf{u}_{m})\right\vert & \leq C\Vert D\mathbf{u}_{m}\Vert
_{L^{4}(\Omega )}\Vert \nabla \phi _{m}\Vert _{L^{4}(\Omega )}\Vert \mathbf{A%
}\mathbf{u}_{m}\Vert _{L^{2}(\Omega )} \\
& \leq C\Vert D\mathbf{u}_{m}\Vert _{L^{2}(\Omega )}^{\frac{1}{2}}\Vert
\mathbf{A}\mathbf{u}_{m}\Vert _{L^{2}(\Omega )}^{\frac{3}{2}}\Vert \nabla
\phi _{m}\Vert _{L^{4}(\Omega )} \\
& \leq \frac{\nu _{\ast }}{32}\Vert \mathbf{A}\mathbf{u}_{m}\Vert
_{L^{2}(\Omega )}^{2}+C\Vert \nabla \phi _{m}\Vert _{L^{4}(\Omega
)}^{4}\Vert D\mathbf{u}_{m}\Vert _{L^{2}(\Omega )}^{2}.
\end{align*}%
Lastly, by \eqref{LADY} and \eqref{5}, we infer that
\begin{align*}
\left\vert (\nu ^{\prime }(\phi _{m})\pi _{m}\nabla \phi _{m},\mathbf{A}%
\mathbf{u}_{m})\right\vert & \leq C\Vert \pi _{m}\Vert _{L^{4}(\Omega
)}\Vert \nabla \phi _{m}\Vert _{L^{4}(\Omega )}\Vert \mathbf{A}\mathbf{u}%
_{m}\Vert _{L^{2}(\Omega )} \\
& \leq C\Vert \nabla \mathbf{u}_{m}\Vert _{L^{2}(\Omega )}^{\frac{1}{2}%
}\Vert \mathbf{A}\mathbf{u}_{m}\Vert _{L^{2}(\Omega )}^{\frac{3}{2}}\Vert
\nabla \phi _{m}\Vert _{L^{4}(\Omega )} \\
& \leq \frac{\nu _{\ast }}{32}\Vert \mathbf{A}\mathbf{u}_{m}\Vert
_{L^{2}(\Omega )}^{2}+C\Vert \nabla \phi _{m}\Vert _{L^{4}(\Omega
)}^{4}\Vert D\mathbf{u}_{m}\Vert _{L^{2}(\Omega )}^{2}.
\end{align*}%
Adding up \eqref{dtu} with \eqref{Au} multiplied by $\omega _{1}$, and taking into
account the previous estimates, we end up with
\begin{equation*}
\frac{\mathrm{d}}{\mathrm{d}t}H_{m}+\frac{\rho _{\ast }}{2}\Vert \partial
_{t}\mathbf{u}_{m}\Vert _{L^{2}(\Omega )}^{2}+\left( \frac{\nu _{\ast
}\omega _{1}}{4}-\frac{2\omega _{1}^{2}(\rho ^{\ast })^{2}}{\rho _{\ast }}%
\right) \Vert \mathbf{A}\mathbf{u}_{m}\Vert _{L^{2}(\Omega )}^{2}\leq
D_{m}H_{m}+Q_{m},
\end{equation*}
where
\begin{align*}
& H_{m}(t):=\frac{1}{2}\int_{\Omega }\nu (\phi _{m}(t))|D\mathbf{u}%
_{m}(t)|^{2}\,\mathrm{d}x, \\
& D_{m}(t):=C\left( 1+\Vert D\mathbf{u}_{m}(t)\Vert _{L^{2}(\Omega
)}^{2}+\Vert \partial _{t}\phi _{m}(t)\Vert _{L^{2}(\Omega )}^{2}+\Vert
\nabla \mu _{m}(t)\Vert _{H^{1}(\Omega )}^{2}+\Vert \nabla \phi _{m}(t)\Vert
_{L^{4}(\Omega )}^{4}\right) , \\
& Q_{m}(t):=C\Vert \nabla \mu _{m}(t)\Vert _{L^{2}(\Omega )}^{2}.
\end{align*}%
In turn, setting $\omega _{1}=\frac{\nu _{\ast }\rho _{\ast }}{16(\rho
^{\ast })^{2}}>0$,
\begin{equation}
\frac{\mathrm{d}}{\mathrm{d}t}H_{m}+\frac{\rho _{\ast }}{2}\Vert \partial
_{t}\mathbf{u}_{m}(t)\Vert _{L^{2}(\Omega )}^{2}+\frac{\nu _{\ast }\omega
_{1}}{8}\Vert \mathbf{A}\mathbf{u}_{m}(t)\Vert _{L^{2}(\Omega )}^{2}\leq
D_{m}H_{m}+Q_{m}.  \label{dHm}
\end{equation}%
Observe now that, by \eqref{mu1} and \eqref{5}, $D_{m}\in L^{1}(0,T)$ and $Q_{m}\in
L^{1}(0,T)$. Thus, an application of the Gronwall lemma gives
\begin{equation}
H_{m}(t)\leq \left( H_{m}(0)+\int_{0}^{T}Q_{m}(\tau )\,\mathrm{d}\tau
\right) \mathrm{exp}\left( \int_{0}^{T}D_{m}(\tau )\,\mathrm{d}\tau \right),\quad \forall
\,t\in \lbrack 0,T].  \label{Hm1}
\end{equation}%
By the properties of the projector $\mathbb{P}_{m}$ and \eqref{mu1} and %
\eqref{5}, we observe that
\begin{equation*}
H_{m}(0)\leq C\Vert \mathbf{u}_{0}\Vert _{H_{0,\sigma }^{1}(\Omega
)}^{2},\quad \int_{0}^{T}Q_{m}(\tau )\,\mathrm{d}\tau \leq C,\quad
\int_{0}^{T}D_{m}(\tau )\,\mathrm{d}\tau \leq C.
\end{equation*}%
Thus, we conclude from \eqref{dHm} and \eqref{Hm1} that
\begin{equation}
\label{Um}
\Vert \mathbf{u}_{m}\Vert _{L^{\infty }(0,T;H_{0,\sigma }^{1}(\Omega
))}+\Vert \partial _{t}\mathbf{u}_{m}\Vert _{L^{2}(0,T;L_{\sigma
}^{2}(\Omega ))}+\Vert \mathbf{u}_{m}\Vert _{L^{2}(0,T;H_{0,\sigma
}^{2}(\Omega ))}\leq C.
\end{equation}%

\noindent \textbf{Passage to the limit and existence of global strong
solutions.} Thanks to the estimates \eqref{mu1}, \eqref{5} and \eqref{Um}
(which are uniform with respect to the parameter $m$), we deduce the
following convergences (up to subsequences)
\begin{equation}
\begin{split}
& \mathbf{u}_{m}\rightharpoonup \mathbf{u}\quad \text{weakly$^\star$ in }%
L^{\infty }(0,T;H_{0,\sigma }^{1}(\Omega )), \\
& \mathbf{u}_{m}\rightharpoonup \mathbf{u}\quad \text{weakly in }%
L^{2}(0,T;H_{0,\sigma }^{2}(\Omega ))\cap H^{1}(0,T;L_{\sigma }^{2}(\Omega
)), \\
& \phi _{m}\rightharpoonup \phi \quad \text{weakly$^\star$ in }L^{\infty
}(0,T;H^{1}(\Omega ))\cap L^{\infty }(\Omega \times (0,T)), \\
& \phi _{m}\rightharpoonup \phi \quad \text{weakly in }L^{q}(0,T;W^{1,p}(%
\Omega )),\quad q=\frac{2p}{p-2},\quad \forall\,p\in (2,\infty ), \\
& \phi _{m}\rightharpoonup \phi \quad \text{weakly in }
H^{1}(0,T;L^{2}(\Omega )\cap W^{1,\infty }(0,T;H^{1}(\Omega )^{\prime }), \\
& \mu _{m}\rightharpoonup \mu \quad \text{weakly$^\star$ in }L^{\infty
}(0,T;H^{1}(\Omega )), \\
& \mu _{m}\rightharpoonup \mu \quad \text{weakly in }L^{2}(0,T;H^{2}(\Omega
))\cap H^{1}(0,T;H^{1}(\Omega )^{\prime }).
\end{split}
\label{conv1}
\end{equation}%
By means of Aubin-Lions Lemma, we have the following strong convergences,
\begin{equation}
\begin{split}
& \mathbf{u}_{m}\rightarrow \mathbf{u}\quad \text{strongly in }%
L^{2}(0,T;H_{0,\sigma }^{1}(\Omega )), \\
& \phi _{m}\rightarrow \phi \quad \text{strongly in }C([0,T];L^{p}(\Omega
)),\quad \forall \,p\in \lbrack 2,\infty ), \\
& \mu _{m}\rightarrow \mu \quad \text{strongly in }L^{2}(0,T;H^{1}(\Omega )).
\end{split}
\label{convtris}
\end{equation}%
As an immediate consequence, we infer that
\begin{equation}
\begin{split}
& \rho (\phi _{m})\rightarrow \rho (\phi )\quad \text{strongly in }%
C([0,T];L^{p}(\Omega )),\quad \forall \,p\in \lbrack 2,\infty ), \\
& \nu (\phi _{m})\rightarrow \nu (\phi )\quad \text{strongly in }%
C([0,T];L^{p}(\Omega )),\quad \forall \,p\in \lbrack 2,\infty ).
\end{split}
\label{convbis}
\end{equation}%
On the other hand, we only know so far that $\phi \in L^{\infty }(\Omega
\times (0,T))$ is such that $\Vert \phi \Vert _{L^{\infty }(\Omega \times
(0,T))}\leq 1$. But, due to the convergence (up to a subsequence) $\phi
_{m}\rightarrow \phi $ almost everywhere in $\Omega \times (0,T)$ and %
\eqref{5}, the Fatou lemma entails that%
$F^{\prime }(\phi )^{2} \in L^2(0,T;L^2(\Omega))$.
In turn, this gives that $|\phi|<1$ almost everywhere in $\Omega
\times (0,T)$. Owing to this, it is possible to show that
\begin{equation}
F^{\prime }(\phi _{m})\rightharpoonup F^{\prime }(\phi )\quad \text{%
weakly$^\star$ in }L^{\infty }(0,T;H^{1}(\Omega )).
\end{equation}%
The above properties are sufficient to show the convergence of the nonlinear
terms in \eqref{vel}-\eqref{eqapprox}. Then, in a standard way, we pass to
the limit as $m\rightarrow \infty $ in \eqref{vel}-\eqref{eqapprox}.
Reasoning now as in \cite{AGG2d}, we infer the existence of a pressure $\Pi
\in L^{2}(0,T;H_{(0)}^{1}(\Omega ))$, such that
\begin{equation*}
\nabla \Pi =-\rho (\phi )\partial _{t}\mathbf{u}-\rho (\phi )(\mathbf{u}%
\cdot \nabla )\mathbf{u}+\rho ^{\prime }(\phi )(\nabla \mu \cdot \nabla )%
\mathbf{u}+\text{div}(\nu (\phi )D\mathbf{u})+\mu \nabla \phi ,
\end{equation*}%
almost everywhere in $\Omega \times (0,T)$.

Concerning the separation property, thanks to the regularity \eqref{regg} on
$[0,T]$, we infer from Remark \ref{Rem-SP} (cf. also Theorem \ref{ExistCahn}%
) that, for any $0<\tau \leq T$, there exists $\delta =\delta (\tau )\in (0,1)$
such that it holds
\begin{equation}
\sup_{t\in \lbrack \tau ,T]}\Vert \phi (t)\Vert _{L^{\infty }(\Omega )}\leq
1-\delta .  \label{del-SP-2}
\end{equation}%
Instead, if we additionally assume $\Vert \phi _{0}\Vert \leq 1-\delta _{0},$
for some $\delta _{0}\in (0,1)$, then an application of Theorem \ref%
{ExistCahn}, case (ii) implies that there exists $\delta ^{\star }>0$
(depending also on $\delta _{0}$) such that the solution $\phi$ satisfies \eqref{SP-Star}.

In order to complete the proof of the existence, we are left to to discuss
the \textit{globality} of the solution $(\mathbf{u},\Pi ,\phi )$. In fact,
we have only shown so far the existence of a solution $(\mathbf{u},\Pi ,\phi
)$ to \eqref{syst2}-\eqref{bic} defined on a given time interval $[0,T]$ for
any fixed $T>0$. Nevertheless, we can easily construct a global
solution $(\mathbf{u},\Pi ,\phi )$ to \eqref{syst2}-\eqref{bic} defined on
the time interval $[0,\infty )$ and satisfying (i)-(iv). Indeed, we first
consider the solution $(\mathbf{u}_{1},\Pi _{1},\phi _{1})$ defined on $%
[0,1] $ originating from $(\mathbf{u}_{0},\phi _{0})$. Next, we notice that $%
\mathbf{u}_{1}(1)\in H_{0,\sigma }^{1}(\Omega )$ and $\phi _{1}(1)\in
H^{1}(\Omega )\cap L^{\infty }(\Omega )$ with $|\overline{\phi _{1}(1)}|<1$.
In addition, in light of \eqref{del-SP-2}, $\Vert \phi _{1}(1)\Vert
_{L^{\infty }(\Omega )}\leq 1-\delta ,$ for some $\delta >0$. Thanks to the
proof above, for any $n\in \mathbb{N}$ with $n\geq 2$, there exists a
solution $(\mathbf{u}_{n},\Pi _{n},\phi _{n})$ to \eqref{syst2}-\eqref{bic}
defined in the time interval $[1,n]$. In particular, by the uniqueness
property proven in the (next) Section \ref{Uniqueness}, we have that
$\mathbf{u}_{n}=\mathbf{u}_{k},$ $\Pi _{n}=\Pi _{k},$ $\phi _{n}=\phi
_{k}$ in $[1,n],$
provided that $n<k$. Therefore, we obtain the solution $(\mathbf{u},\Pi
,\phi )$ defined in $[0,\infty )$ by setting
\begin{equation}
(\mathbf{u}(t),\Pi (t),\phi (t))=%
\begin{cases}
(\mathbf{u}_{1}(t),\Pi _{1}(t),\phi _{1}(t)),\quad & t\in \lbrack 0,1], \\
(\mathbf{u}_{n}(t),\Pi _{n}(t),\phi _{n}(t)),\quad & t\in \lbrack 1,n],\quad
\forall \,n\geq 2.%
\end{cases}%
\end{equation}%
Finally, we observe that $(\mathbf{u},\Pi ,\phi )$ satisfies the energy
equality
\begin{equation*}
E(\mathbf{u}(t),\phi (t))+\int_{\tau }^{t}\left\Vert \sqrt{\nu (\phi (s))}|D%
\mathbf{u}(s)\right\Vert _{L^{2}(\Omega )}^{2}+\Vert \nabla \mu (s)\Vert
_{L^{2}(\Omega ))}^{2}\,\mathrm{d}s=E(\mathbf{u}(\tau ),\phi (\tau ))
\end{equation*}%
for every $0<\tau \leq t<\infty $, which clearly follows from the regularity
in each interval $[0,T]$ (cf. \eqref{conv1}-\eqref{convbis}). This implies
that $\mathbf{u}\in L^{\infty }(0,\infty ;L_{\sigma }^{2}(\Omega ))\cap
L^{2}(0,\infty ;H_{0,\sigma }^{1}(\Omega ))$ and $\nabla \mu \in
L^{2}(0,\infty ;L^{2}(\Omega ;\mathbb{R}^{2}))$. By interpolation, it
follows that $\mathbf{u}\in L^{4}(0,\infty ;L_{\sigma }^{4}(\Omega ))$.
Then, in light of Theorem \ref{ExistCahn}, we deduce from the estimates %
\eqref{ST1}-\eqref{ST2} as $T\rightarrow \infty $ that $\nabla \mu \in
L^{\infty }(0,\infty ;L^{2}(\Omega ;\R^{2}))$ and $\partial _{t}\phi \in
L^{2}(0,\infty ;L^{2}(\Omega ))$. By \eqref{ST3}, we also infer that
\begin{align*}
& \phi \in L^{\infty }(0,\infty ;H^{1}(\Omega ))\cap L_{\mathrm{uloc}%
}^{q}([0,\infty );W^{1,p}(\Omega )),\quad q=\frac{2p}{p-2},\quad p\in
(2,\infty ), \\
& \partial _{t}\phi \in L^{\infty }(0,\infty ;H^{1}(\Omega )^{\prime
}),\quad F^{\prime }(\phi )\in L^{\infty }(0,\infty ;H^{1}(\Omega )),\quad
F^{\prime \prime }(\phi )\in L^{\infty }(0,\infty ;L^{p}(\Omega )),\quad
p\in \lbrack 2,\infty ), \\
& \mu \in BC_{\mathrm{w}}([0,\infty );H^{1}(\Omega ))\cap L_{\mathrm{uloc}%
}^{2}([0,\infty );H^{2}(\Omega ))\cap H_{\mathrm{uloc}}^{1}([0,\infty
);H^{1}(\Omega )^{\prime }).
\end{align*}%
Moreover, by Remark \eqref{Rem-SP}, there exists $\delta >0$ such that
$
\sup_{t\in \lbrack \tau ,\infty )}\Vert \phi (t)\Vert _{L^{\infty }(\Omega
)}\leq 1-\delta $. 
On the other hand, recalling that $\mathbf{u}$ is the limit of the
approximate solutions $\mathbf{u}_{m}$ in $[0,n]$ for each $n\in
\mathbb{N}$ and that each $\mathbf{u}_{m}$ satisfies \eqref{dHm}, it is
easily seen from the uniform Gronwall lemma that
\begin{equation*}
\mathbf{u}\in L^{\infty }(0,\infty ;H_{0,\sigma }^{1}(\Omega ))\cap H_{%
\mathrm{uloc}}^{1}([0,\infty );L_{\sigma }^{2}(\Omega ))\cap L_{\mathrm{uloc}%
}^{2}([0,\infty );H_{0,\sigma }^{2}(\Omega )),
\end{equation*}%
and, in turn, $\Pi \in L_{\mathrm{uloc}}^{2}([0,\infty );H_{(0)}^{1}(\Omega
))$. The proof of the existence of global strong solutions is thus concluded.


\section{Proof of Theorem \protect\ref{MR:strong}: Continuous dependence
estimate for ``separated" strong solutions}

\label{Uniqueness} \setcounter{equation}{0}

Consider two sets of initial data $(\mathbf{u}_{0}^{1},\phi _{0}^{1}$) and $(%
\mathbf{u}_{0}^{2},\phi _{0}^{2})$ satisfying the assumptions of Theorem \ref%
{MR:strong}. In particular, we consider \textquotedblleft separated" initial
data, i.e. $\Vert \phi _{0}^{i}\Vert _{L^{\infty }(\Omega )}<1$ for $i=1,2$.
We denote by $(\mathbf{u}_{j},\Pi _{j},\phi _{j})$, $j=1,2$,
the strong solutions to \eqref{syst2}-\eqref{bic}
originating from ($\mathbf{u}_{0}^{j},\phi _{0}^{j}$). Clearly both the solutions satisfy %
\eqref{regg} and the statement (iv) of Theorem \ref{MR:strong}. Let us set $%
\mathbf{u}=\mathbf{u}_{1}-\mathbf{u}_{2}$, $P=\Pi _{1}-\Pi _{2}$, $\Phi
=\phi _{1}-\phi _{2}$ and $\Theta =F^{\prime }(\phi
_{1})-F^{\prime }(\phi _{2})-J\ast \Phi $. These functions satisfy the system
\begin{equation}
\begin{split}
&\rho (\phi _{1})\partial _{t}\mathbf{u}+(\rho (\phi _{1})-\rho (\phi
_{2})\partial _{t}\mathbf{u}_{2}+\rho (\phi _{1})(\mathbf{u}_{1}\cdot \nabla
)\mathbf{u}+\rho (\phi _{1})(\mathbf{u}\cdot \nabla )\mathbf{u}_{2}+(\rho
(\phi _{1})-\rho (\phi _{2}))(\mathbf{u}_{2}\cdot \nabla )\mathbf{u}_{2} \\
&\quad -\frac{\rho _{1}-\rho _{2}}{2}(\nabla \mu _{1}\cdot \nabla )\mathbf{u}-\frac{%
\rho _{1}-\rho _{2}}{2}(\nabla \Theta \cdot \nabla )\mathbf{u}_{2}-\mathrm{%
div}\,(\nu (\phi _{1})D\mathbf{u})-\mathrm{div}\,(\nu ((\phi _{1})-\nu (\phi
_{2}))D\mathbf{u}_{2})+\nabla P \\
&=\mu _{1}\nabla \Phi +\Theta \nabla \phi _{2}, \\[5pt]
&\partial _{t}\Phi +\mathbf{u}_{1}\cdot \nabla \Phi +\mathbf{u}\cdot \nabla
\phi _{2}=\Delta \Theta ,%
\end{split}
\label{phi_s}
\end{equation}
almost everywhere in $\Omega \times (0,T)$. We observe that
\begin{equation}
-\int_{\Omega }\partial _{t}\rho (\phi _{1})\frac{|\mathbf{u}|^{2}}{2}\,%
\mathrm{d}x+\int_{\Omega }\rho (\phi _{1})(\mathbf{u}_{1}\cdot \nabla
\mathbf{u})\cdot \mathbf{u}\,\mathrm{d}x-\frac{\rho _{1}-\rho _{2}}{2}%
\int_{\Omega }(\nabla \mu _{1}\cdot \nabla )\mathbf{u}\cdot \mathbf{u}\,%
\mathrm{d}x=0,  \label{U:rel1}
\end{equation}%
\begin{equation}
\int_{\Omega }(\nabla \Theta \cdot \nabla )\mathbf{u}_{2}\cdot \mathbf{u}\,%
\mathrm{d}x=-\int_{\Omega }\Theta \Delta \mathbf{u}_{2}\cdot \mathbf{u}\,%
\mathrm{d}x-\int_{\Omega }\Theta \nabla \mathbf{u}_{2}:\nabla \mathbf{u}\,%
\mathrm{d}x,  \label{U:rel2}
\end{equation}%
and
\begin{equation}
\int_{\Omega }\left( \mu _{1}\nabla \Phi +\Theta \nabla \phi _{2}\right)
\cdot \mathbf{u}\,\mathrm{d}x=-\int_{\Omega }\Phi (\nabla J\ast \phi
_{1})\cdot \mathbf{u}\,\mathrm{d}x-\int_{\Omega }\phi _{2}(\nabla J\ast \Phi
)\cdot \mathbf{u}\,\mathrm{d}x.  \label{U:rel3}
\end{equation}%
Multiplying \eqref{phi_s}$_{1}$ by $\mathbf{u}$ and integrating over $\Omega
$, we find (cf. \cite[Equation 6.3]{AGG2d})
\begin{equation}
\begin{split}
& \frac{\mathrm{d}}{\mathrm{d}t}\int_{\Omega }\frac{\rho (\phi _{1})}{2}|%
\mathbf{u}|^{2}\,\mathrm{d}x+\int_{\Omega }\nu (\phi _{1})|D\mathbf{U}|^{2}\,%
\mathrm{d}x \\
& =-\int_{\Omega }(\rho (\phi _{1})-\rho (\phi _{2})\partial _{t}\mathbf{u}%
_{2}\cdot \mathbf{u}\,\mathrm{d}x-\int_{\Omega }\rho (\phi _{1})(\mathbf{u}%
\cdot \nabla )\mathbf{u}_{2}\cdot \mathbf{u}\,\mathrm{d}x-\int_{\Omega
}(\rho (\phi _{1})-\rho (\phi _{2}))(\mathbf{u}_{2}\cdot \nabla )\mathbf{u}%
_{2}\cdot \mathbf{u}\,\mathrm{d}x \\
& \quad -\int_{\Omega }(\nu (\phi _{1})-\nu (\phi _{2}))D\mathbf{u}%
_{2}:\nabla \mathbf{u}\,\mathrm{d}x-\frac{\rho _{1}-\rho _{2}}{2}%
\int_{\Omega }\Theta \Delta \mathbf{u}_{2}\cdot \mathbf{u}\,\mathrm{d}x-%
\frac{\rho _{1}-\rho _{2}}{2}\int_{\Omega }\Theta \nabla \mathbf{u}%
_{2}:\nabla \mathbf{u}\,\mathrm{d}x \\
& \quad -\int_{\Omega }\Phi (\nabla J\ast \phi _{1})\cdot \mathbf{u}\,%
\mathrm{d}x-\int_{\Omega }\phi _{2}(\nabla J\ast \Phi )\cdot \mathbf{U}\,%
\mathrm{d}x.
\end{split}
\label{U-uni}
\end{equation}%
By the strict convexity of $F$, we notice that
\begin{align*}
\int_{\Omega }\nabla \Theta \cdot \nabla \Phi \,\mathrm{d}x& =\int_{\Omega
}F^{\prime \prime }(\phi _{1})|\nabla \Phi |^{2}\,\mathrm{d}x+\int_{\Omega
}(F^{\prime \prime }(\phi _{1})-F^{\prime \prime }(\phi _{2}))\nabla \phi
_{2}\cdot \nabla \phi \,\mathrm{d}x-\int_{\Omega }\nabla J\ast \Phi \cdot
\nabla \Phi \,\mathrm{d}x \\
& \geq \alpha \Vert \nabla \Phi \Vert _{L^{2}(\Omega )}^{2}+\int_{\Omega
}(F^{\prime \prime }(\phi _{1})-F^{\prime \prime }(\phi _{2}))\nabla \phi
_{2}\cdot \nabla \phi \,\mathrm{d}x-\int_{\Omega }\nabla J\ast \Phi \cdot
\nabla \Phi \,\mathrm{d}x.
\end{align*}%
Then, multiplying \eqref{phi_s}$_{2}$ by $\Phi $ and integrating over $%
\Omega $, we obtain
\begin{equation}
\begin{split}
& \frac{1}{2}\frac{\mathrm{d}}{\mathrm{d}t}\Vert \Phi \Vert _{L^{2}(\Omega
)}^{2}+\alpha \Vert \nabla \Phi \Vert _{L^{2}(\Omega )}^{2} \\
& =\int_{\Omega }\phi _{2}(\mathbf{u}\cdot \nabla \Phi )\,\mathrm{d}%
x+\int_{\Omega }\nabla J\ast \Phi \cdot \nabla \Phi \,\mathrm{d}%
x-\int_{\Omega }(F^{\prime \prime }(\phi _{1})-F^{\prime \prime }(\phi
_{2}))\nabla \phi _{2}\cdot \nabla \phi \,\mathrm{d}x.
\end{split}
\label{Phi-uni}
\end{equation}%
Adding together \eqref{U-uni} and \eqref{Phi-uni} and exploiting the
hypothesis \ref{nu}, we arrive at
\begin{equation}
\begin{split}
& \frac{\mathrm{d}}{\mathrm{d}t}\left( \int_{\Omega }\frac{\rho (\phi _{1})}{%
2}|\mathbf{u}|^{2}\,\mathrm{d}x+\frac{1}{2}\Vert \Phi \Vert _{L^{2}(\Omega
)}^{2}\right) +\nu _{\ast }\Vert D\mathbf{u}\Vert _{L^{2}(\Omega
)}^{2}+\alpha \Vert \nabla \Phi \Vert _{L^{2}(\Omega )}^{2} \\
& \leq -\int_{\Omega }(\rho (\phi _{1})-\rho (\phi _{2})\partial _{t}\mathbf{%
u}_{2}\cdot \mathbf{u}\,\mathrm{d}x-\int_{\Omega }\rho (\phi _{1})(\mathbf{u}%
\cdot \nabla )\mathbf{u}_{2}\cdot \mathbf{u}\,\mathrm{d}x-\int_{\Omega
}(\rho (\phi _{1})-\rho (\phi _{2}))(\mathbf{u}_{2}\cdot \nabla )\mathbf{u}%
_{2}\cdot \mathbf{u}\,\mathrm{d}x \\
& \quad -\int_{\Omega }(\nu (\phi _{1})-\nu (\phi _{2}))D\mathbf{u}%
_{2}:\nabla \mathbf{u}\,\mathrm{d}x-\frac{\rho _{1}-\rho _{2}}{2}%
\int_{\Omega }\Theta \Delta \mathbf{u}_{2}\cdot \mathbf{u}\,\mathrm{d}x-%
\frac{\rho _{1}-\rho _{2}}{2}\int_{\Omega }\Theta \nabla \mathbf{u}%
_{2}:\nabla \mathbf{u}\,\mathrm{d}x \\
& \quad -\int_{\Omega }\Phi (\nabla J\ast \phi _{1})\cdot \mathbf{u}\,%
\mathrm{d}x-\int_{\Omega }\phi _{2}(\nabla J\ast \Phi )\cdot \mathbf{U}\,%
\mathrm{d}x+\int_{\Omega }\phi _{2}(\mathbf{u}\cdot \nabla \Phi )\,\mathrm{d}%
x \\
& \quad +\int_{\Omega }\nabla J\ast \Phi \cdot \nabla \Phi \,\mathrm{d}%
x-\int_{\Omega }(F^{\prime \prime }(\phi _{1})-F^{\prime \prime }(\phi
_{2}))\nabla \phi _{2}\cdot \nabla \phi \,\mathrm{d}x.
\end{split}
\label{UNI-1}
\end{equation}%
By using \eqref{RHO-Jtilde}, \eqref{LADY} and \eqref{korn}, we obtain
\begin{align*}
\left\vert \int_{\Omega }(\rho (\phi _{1})-\rho (\phi _{2})\partial _{t}%
\mathbf{u}_{2}\cdot \mathbf{u}\,\mathrm{d}x\right\vert & \leq C\Vert \Phi
\Vert _{L^{4}(\Omega )}\Vert \partial _{t}\mathbf{u}_{2}\Vert _{L^{2}(\Omega
)}\Vert \mathbf{u}\Vert _{L^{4}(\Omega )} \\
& \leq C\Vert \partial _{t}\mathbf{u}_{2}\Vert _{L^{2}(\Omega )}\Vert \Phi
\Vert _{L^{2}(\Omega )}^{\frac{1}{2}}\left( \Vert \Phi \Vert _{L^{2}(\Omega
)}^{\frac{1}{2}}+\Vert \nabla \Phi \Vert _{L^{2}(\Omega )}^{\frac{1}{2}%
}\right) \Vert \mathbf{u}\Vert _{L^{2}(\Omega )}^{\frac{1}{2}}\Vert D\mathbf{%
u}\Vert _{L^{2}(\Omega )}^{\frac{1}{2}} \\
& \leq \frac{\nu _{\ast }}{12}\Vert D\mathbf{u}\Vert _{L^{2}(\Omega )}^{2}+%
\frac{\alpha }{10}\Vert \nabla \Phi \Vert _{L^{2}(\Omega )}^{2} \\
& \quad +C\left( 1+\Vert \partial _{t}\mathbf{u}_{2}\Vert _{L^{2}(\Omega
)}^{2}\right) \left( \Vert \Phi \Vert _{L^{2}(\Omega )}^{2}+\Vert \mathbf{u}%
\Vert _{L^{2}(\Omega )}^{2}\right) .
\end{align*}%
In a similar way, by \eqref{regg} and \eqref{LADY}, we have
\begin{align*}
\left\vert \int_{\Omega }\rho (\phi _{1})(\mathbf{u}\cdot \nabla )\mathbf{u}%
_{2}\cdot \mathbf{u}\,\mathrm{d}x\right\vert & \leq C\Vert \mathbf{u}\Vert
_{L^{4}(\Omega )}\Vert \nabla \mathbf{u}_{2}\Vert _{L^{2}(\Omega )}\Vert
\mathbf{u}\Vert _{L^{4}(\Omega )} \\
& \leq \frac{\nu _{\ast }}{12}\Vert D\mathbf{u}\Vert _{L^{2}(\Omega
)}^{2}+C\Vert \mathbf{u}\Vert _{L^{2}(\Omega )}^{2}.
\end{align*}%
Thanks to \eqref{RHO-Jtilde}, \eqref{regg} and \eqref{W14bis}, it follows
that
\begin{align*}
\left\vert \int_{\Omega }(\rho (\phi _{1})-\rho (\phi _{2}))(\mathbf{u}%
_{2}\cdot \nabla )\mathbf{u}_{2}\cdot \mathbf{u}\,\mathrm{d}x\right\vert &
\leq C\Vert \Phi \Vert _{L^{4}(\Omega )}\Vert \mathbf{u}_{2}\Vert
_{L^{\infty }(\Omega )}\Vert \nabla \mathbf{u}_{2}\Vert _{L^{2}(\Omega
)}\Vert \mathbf{u}\Vert _{L^{4}(\Omega )} \\
& \leq C(\Vert \Phi \Vert _{L^{2}(\Omega )}+\Vert \nabla \Phi \Vert
_{L^{2}(\Omega )})\Vert \mathbf{u}_{2}\Vert _{H^{2}(\Omega )}^{\frac{1}{2}%
}\Vert \mathbf{u}\Vert _{L^{2}(\Omega )}^{\frac{1}{2}}\Vert D\mathbf{u}\Vert
_{L^{2}(\Omega )}^{\frac{1}{2}} \\
& \leq \frac{\nu _{\ast }}{12}\Vert D\mathbf{u}\Vert _{L^{2}(\Omega )}^{2}+%
\frac{\alpha }{10}\Vert \nabla \Phi \Vert _{L^{2}(\Omega )}^{2} \\
& \quad +C\Vert \mathbf{u}_{2}\Vert _{H^{2}(\Omega )}^{2}\Vert \mathbf{u}%
\Vert _{L^{2}(\Omega )}^{2}+C\Vert \Phi \Vert _{L^{2}(\Omega )}^{2}.
\end{align*}%
Next, recalling that the separation property of $\phi _{1}$ and $\phi _{2}$
implies that $|F^{\prime }(\phi _{1})-F^{\prime }(\phi _{2})|\leq C|\phi |$
almost everywhere in $\Omega \times (0,T)$ for some universal constant $C$
(depending only on the norms of the initial data), and by using the
assumption on $J$, we infer that
\begin{align*}
\left\vert \frac{\rho _{1}-\rho _{2}}{2}\int_{\Omega }\Theta \Delta \mathbf{u%
}_{2}\cdot \mathbf{u}\,\mathrm{d}x\right\vert & =\frac{\rho _{1}-\rho _{2}}{2%
}\left\vert \int_{\Omega }(F^{\prime }(\phi _{1})-F^{\prime }(\phi
_{2}))\Delta \mathbf{u}_{2}\cdot \mathbf{u}\,\mathrm{d}x-\int_{\Omega
}(J\ast \Phi )\Delta \mathbf{u}_{2}\cdot \mathbf{u}\,\mathrm{d}x\right\vert
\\
& \leq C\Vert \Phi \Vert _{L^{4}(\Omega )}\Vert \Delta \mathbf{u}_{2}\Vert
_{L^{2}(\Omega )}\Vert \mathbf{u}\Vert _{L^{4}(\Omega )} \\
& \leq C\Vert \Phi \Vert _{L^{2}(\Omega )}^{\frac{1}{2}}\left( \Vert \Phi
\Vert _{L^{2}(\Omega )}^{\frac{1}{2}}+\Vert \nabla \Phi \Vert _{L^{2}(\Omega
)}^{\frac{1}{2}}\right) \Vert \Delta \mathbf{u}_{2}\Vert _{L^{2}(\Omega
)}\Vert \mathbf{u}\Vert _{L^{2}(\Omega )}^{\frac{1}{2}}\Vert D\mathbf{u}%
\Vert _{L^{2}(\Omega )}^{\frac{1}{2}} \\
& \leq \frac{\nu _{\ast }}{12}\Vert D\mathbf{u}\Vert _{L^{2}(\Omega )}^{2}+%
\frac{\alpha }{10}\Vert \nabla \Phi \Vert _{L^{2}(\Omega )}^{2} \\
& \quad +C\left( 1+\Vert \Delta \mathbf{u}_{2}\Vert _{L^{2}(\Omega
)}^{2}\right) \left( \Vert \Phi \Vert _{L^{2}(\Omega )}^{2}+\Vert \mathbf{u}%
\Vert _{L^{2}(\Omega )}^{2}\right) .
\end{align*}%
Similarly, we find
\begin{align*}
\left\vert \frac{\rho _{1}-\rho _{2}}{2}\int_{\Omega }\Theta \nabla \mathbf{u%
}_{2}:\nabla \mathbf{u}\,\mathrm{d}x\right\vert & =\frac{\rho _{1}-\rho _{2}%
}{2}\left\vert \int_{\Omega }(F^{\prime }(\phi _{1})-F^{\prime }(\phi
_{2}))\nabla \mathbf{u}_{2}:\nabla \mathbf{u}\,\mathrm{d}x-\int_{\Omega
}(J\ast \Phi )\nabla \mathbf{u}_{2}:\nabla \mathbf{u}\,\mathrm{d}x\right\vert
\\
& \leq C\Vert \Phi \Vert _{L^{4}(\Omega )}\Vert \nabla \mathbf{u}_{2}\Vert
_{L^{4}(\Omega )}\Vert \nabla \mathbf{u}\Vert _{L^{2}(\Omega )} \\
& \leq C\Vert \Phi \Vert _{L^{2}(\Omega )}^{\frac{1}{2}}\left( \Vert \Phi
\Vert _{L^{2}(\Omega )}^{\frac{1}{2}}+\Vert \nabla \Phi \Vert _{L^{2}(\Omega
)}^{\frac{1}{2}}\right) \Vert \nabla \mathbf{u}_{2}\Vert _{L^{4}(\Omega
)}\Vert \nabla \mathbf{u}\Vert _{L^{2}(\Omega )} \\
& \leq \frac{\nu _{\ast }}{12}\Vert D\mathbf{u}\Vert _{L^{2}(\Omega )}^{2}+%
\frac{\alpha }{10}\Vert \nabla \Phi \Vert _{L^{2}(\Omega )}^{2}+C\left(
1+\Vert \nabla \mathbf{u}_{2}\Vert _{L^{4}(\Omega )}^{4}\right) \Vert \Phi
\Vert _{L^{2}(\Omega )}^{2}
\end{align*}%
and
\begin{align*}
\left\vert \int_{\Omega }(\nu (\phi _{1})-\nu (\phi _{2}))D\mathbf{u}%
_{2}:\nabla \mathbf{u}\,\mathrm{d}x\right\vert & \leq C\Vert \Phi \Vert
_{L^{4}(\Omega )}\Vert D\mathbf{u}_{2}\Vert _{L^{4}(\Omega )}\Vert \nabla
\mathbf{u}\Vert _{L^{2}(\Omega )} \\
& \leq \frac{\nu _{\ast }}{12}\Vert D\mathbf{u}\Vert _{L^{2}(\Omega )}^{2}+%
\frac{\alpha }{10}\Vert \nabla \Phi \Vert _{L^{2}(\Omega )}^{2}+C\left(
1+\Vert \nabla \mathbf{u}_{2}\Vert _{L^{4}(\Omega )}^{4}\right) \Vert \Phi
\Vert _{L^{2}(\Omega )}^{2}.
\end{align*}%
Due to the separation property and \ref{J-ass}, we also have
\begin{align*}
\left\vert \int_{\Omega }\Phi (\nabla J\ast \phi _{1})\cdot \mathbf{u}\,%
\mathrm{d}x+\int_{\Omega }\phi _{1}(\nabla J\ast \Phi )\cdot \mathbf{u}%
\mathrm{d}x\right\vert & \leq C\left( \Vert \phi _{1}\Vert _{L^{\infty
}(\Omega )}+\Vert \phi _{2}\Vert _{L^{\infty }(\Omega )}\right) \Vert \Phi
\Vert _{L^{2}(\Omega )}\Vert \mathbf{u}\Vert _{L^{2}(\Omega )} \\
& \leq C\left( \Vert \Phi \Vert _{L^{2}(\Omega )}^{2}+\Vert \mathbf{u}\Vert
_{L^{2}(\Omega )}^{2}\right),
\end{align*}%
\begin{equation*}
\left\vert \int_{\Omega }\phi _{2}(\mathbf{u}\cdot \nabla \Phi )\,\mathrm{d}%
x\right\vert \leq \Vert \phi _{2}\Vert _{L^{\infty }(\Omega )}\Vert \mathbf{u%
}\Vert _{L^{2}(\Omega )}\Vert \nabla \Phi \Vert _{L^{2}(\Omega )}\leq \frac{%
\alpha }{10}\Vert \nabla \Phi \Vert _{L^{2}(\Omega )}^{2}+C\Vert \mathbf{u}%
\Vert _{L^{2}(\Omega )}^{2},
\end{equation*}%
and
\begin{equation*}
\left\vert \int_{\Omega }\nabla J\ast \Phi \cdot \nabla \Phi \,\mathrm{d}%
x\right\vert \leq \frac{\alpha }{10}\Vert \nabla \Phi \Vert _{L^{2}(\Omega
)}^{2}+C\Vert \Phi \Vert _{L^{2}(\Omega )}^{2}.
\end{equation*}%
Lastly, using again the separation property, we observe that $|F^{\prime
\prime }(\phi _{1})-F^{\prime \prime }(\phi _{2})|\leq C|\phi |$ almost
everywhere in $\Omega \times (0,T)$ for some universal constant $C$.
Combining this fact and \eqref{LADY}, we get
\begin{align*}
\left\vert \int_{\Omega }(F^{\prime \prime }(\phi _{1})-F^{\prime \prime
}(\phi _{2}))\nabla \phi _{2}\cdot \nabla \phi \,\mathrm{d}x\right\vert &
\leq C\Vert \Phi \Vert _{L^{4}(\Omega )}\Vert \nabla \phi _{2}\Vert
_{L^{4}(\Omega )}\Vert \nabla \Phi \Vert _{L^{2}(\Omega )} \\
& \leq C\Vert \Phi \Vert _{L^{2}(\Omega )}^{\frac{1}{2}}\left( \Vert \Phi
\Vert _{L^{2}(\Omega )}^{\frac{1}{2}}+\Vert \nabla \Phi \Vert _{L^{2}(\Omega
)}^{\frac{1}{2}}\right) \Vert \nabla \phi _{2}\Vert _{L^{4}(\Omega )}\Vert
\nabla \Phi \Vert _{L^{2}(\Omega )} \\
& \leq \frac{\alpha }{10}\Vert \nabla \Phi \Vert _{L^{2}(\Omega
)}^{2}+C\left( 1+\Vert \nabla \phi _{2}\Vert _{L^{4}(\Omega )}^{4}\right)
\Vert \Phi \Vert _{L^{2}(\Omega )}^{2}.
\end{align*}%
Therefore, adding up the above estimates, we deduce that
\begin{equation*}
\frac{\mathrm{d}}{\mathrm{d}t}\left( \int_{\Omega }\frac{\rho (\phi _{1})}{2}%
|\mathbf{u}|^{2}\,\mathrm{d}x+\frac{1}{2}\Vert \Phi \Vert _{L^{2}(\Omega
)}^{2}\right) \leq K\left( \int_{\Omega }\frac{\rho (\phi _{1})}{2}|%
\mathbf{u}|^{2}\,\mathrm{d}x+\frac{1}{2}\Vert \Phi \Vert _{L^{2}(\Omega
)}^{2}\right),
\end{equation*}%
where $K$ is defined by \eqref{contdepconst}. Note that $K\in L^{1}(0,T)$
owing to \eqref{regg}. In conclusion, an application of the Gronwall lemma gives uniqueness of strong
solutions as well as \eqref{contdep}.

\section{Proof of Theorem \protect\ref{MR:weak}: Propagation of regularity
for weak solutions}

Let $(\mathbf{u},\phi )$ be a weak solution on $[0,T]$ satisfying (i)-(iv)
as ensured by Theorem \ref{R:weak-sol} and let $\tau \in (0,T)$ be fixed.
Since $F^{\prime }(\phi )\in L^{2}(0,T;H^{1}(\Omega ))$, exploiting the
conservation of mass and \eqref{regg-weak}, there exists $\tau _{1}\in
(0,\tau )$ such that
\begin{equation*}
\phi (\tau _{1})\in H^{1}(\Omega ),\quad \left\vert \overline{\phi (\tau
_{1})}\right\vert <1,\quad \text{and}\quad F^{\prime }(\phi (\tau _{1}))\in
H^{1}(\Omega ).
\end{equation*}%
Recalling now that $C_{\mathrm{w}}([0,T];L_{\sigma }^{2})\cap
L^{2}(0,T;H_{0,\sigma }^{1}(\Omega ))\hookrightarrow L^{4}(0,T;L_{\sigma
}^{4}(\Omega ))$, an application of Theorem \ref{ExistCahn} (see also Remark %
\ref{R3}) entails that
\begin{equation}
\begin{cases}
 \phi \in L^{\infty }(\tau _{1},T;L^{\infty }(\Omega )):\quad |\phi
(x,t)|<1\ \text{ for a.a. } x\in\Omega ,\,\forall \,t\in \lbrack \tau _{1},T], \\
 \phi \in L^{\infty }(\tau _{1},T;H^{1}(\Omega ))\cap L^{q}(\tau
_{1},T;W^{1,p}(\Omega )),\quad q=\frac{2p}{p-2},\quad \forall \,p\in
(2,\infty ), \\
 \partial _{t}\phi \in L^{\infty }(\tau _{1},T;H^{1}(\Omega )^{\prime
})\cap L^{2}(\tau _{1},T;L^{2}(\Omega )), \\
 \mu \in L^{\infty }(\tau _{1},T;H^{1}(\Omega ))\cap L^{2}(\tau
_{1},T;H^{2}(\Omega ))\cap H^{1}(\tau _{1},T;H^{1}(\Omega )^{\prime }), \\
 F^{\prime }(\phi )\in L^{\infty }(\tau _{1},T;H^{1}(\Omega )),\quad
F^{\prime \prime }(\phi )\in L^{\infty }(\tau _{1},T;L^{p}(\Omega )),\quad
\forall \,p\in \lbrack 2,\infty ),
\end{cases}
\label{W:reg1}
\end{equation}%
which satisfies
\begin{equation}
\partial _{t}\phi +\mathbf{u}\cdot \nabla \phi =\Delta \mu ,\quad \mu
=F^{\prime }(\phi )-J\ast \phi ,\quad \text{ a.e. in } \Omega \times (\tau
_{1},T),  \label{W:problem}
\end{equation}%
as well as
\begin{equation}
\mathcal{E}(\phi (t))+\int_{\tau_1}^{t}\Vert \nabla \mu (\tau )\Vert
_{L^{2}(\Omega )}^{2}\,\mathrm{d}\tau +\int_{\tau_1}^{t}\int_{\Omega }\phi \,%
\mathbf{u}\cdot \nabla \mu \,\mathrm{d}x\,\mathrm{d}\tau =\mathcal{E}(\phi
_{\tau_1}),\quad \forall \,t\in \lbrack \tau _{1},T].  \label{W:en-id}
\end{equation}%
In addition, there exists $\tau _{2}\in (\tau _{1},\tau )$ such that
\begin{equation}
\sup_{t\in \lbrack \tau _{2},T]}\Vert \phi (t)\Vert _{L^{\infty }(\Omega
)}\leq 1-\delta .  \label{W:sp}
\end{equation}%
Furthermore, we also have
\begin{equation}
\partial _{t}\mu \in L^{2}(\tau _{2},T;L^{2}(\Omega ))\quad \text{and}\quad
\mu \in C([\tau _{2},T];H^{1}(\Omega )).  \label{W:reg2}
\end{equation}%
It is worth pointing out that the uniqueness of weak solutions in Theorem %
\ref{ExistCahn} guarantees that $\phi $ and the solution originating from $%
\phi (\tau _{1})$ coincide.

Next, in light of the above propagation of regularity of the concentration,
we improve the regularity of $\partial _{t}(\rho (\phi )\mathbf{u})$. To
this end, we first recall from \eqref{weak-NS} that
\begin{equation}
\langle \partial _{t}(\rho (\phi )\mathbf{u}),\mathbf{v}\rangle
_{H_{0,\sigma }^{2}(\Omega )}=(\mathbf{f},\mathbf{v}),\quad \forall \,%
\mathbf{v}\in H_{0,\sigma }^{2}(\Omega ),  \label{W:mom1}
\end{equation}%
for almost any $t\in (\tau _{1},T)$, where
\begin{equation*}
(\mathbf{f},\mathbf{v})=(\rho (\phi )\mathbf{u}\otimes \mathbf{u},\nabla
\mathbf{v})-(\nu (\phi )D\mathbf{u},\nabla \mathbf{v})-\frac{\rho _{1}-\rho
_{2}}{2}(\mathbf{u}\otimes \nabla \mu ,\nabla \mathbf{v})+(\mu \nabla \phi ,%
\mathbf{v}).
\end{equation*}%
Thanks to \eqref{W:reg1}, we find that
\begin{align*}
|(\mathbf{f},\mathbf{v})|& \leq C\left( \Vert \mathbf{u}\Vert _{L^{4}(\Omega
)}^{2}+\Vert D\mathbf{u}\Vert _{L^{2}(\Omega )}+\Vert \nabla \mu \Vert
_{L^{4}(\Omega )}\Vert \mathbf{u}\Vert _{L^{4}(\Omega )}+\Vert \mu \Vert
_{L^{4}(\Omega )}\Vert \nabla \phi \Vert _{L^{2}(\Omega )}\right) \Vert
\mathbf{v}\Vert _{H_{0,\sigma }^{1}(\Omega )} \\
& \leq C\left( 1+\Vert \nabla \mathbf{u}\Vert _{L^{2}(\Omega )}+\Vert \nabla
\mu \Vert _{H^{1}(\Omega )}\right) \Vert \mathbf{v}\Vert _{H_{0,\sigma
}^{1}(\Omega )},
\end{align*}%
for any $\mathbf{v}\in H_{0,\sigma }^{2}(\Omega )$ and almost any $t\in
(\tau _{1},T)$. Since $H_{0,\sigma }^{2}(\Omega )$ is dense in $H_{0,\sigma
}^{1}(\Omega )$, the functional $\mathbf{f}$ has a unique extension to $%
H_{0,\sigma }^{1}(\Omega )$. As a result, we deduce that $\mathbf{f}\in
L^{2}(\tau _{1},T;H_{0,\sigma }^{1}(\Omega )^{\prime })$. By definition of
the weak time derivative, this clearly entails that $\partial _{t}(\rho
(\phi )\mathbf{u})\in L^{2}(\tau _{1},T;H_{0,\sigma }^{1}(\Omega )^{\prime
}) $ and
\begin{equation}
\langle \partial _{t}(\rho (\phi )\mathbf{u}),\mathbf{v}\rangle
_{H_{0,\sigma }^{1}(\Omega )}=(\mathbf{f},\mathbf{v}),\quad \forall \,%
\mathbf{v}\in H_{0,\sigma }^{1}(\Omega ),  \label{W:mom2}
\end{equation}%
almost everywhere in $(\tau _{1},T)$. As a consequence, we can apply \cite[%
Lemma 5.3]{Frigeri} which gives that the chain rule
\begin{equation*}
\langle \partial _{t}(\rho (\phi )\mathbf{u}),\mathbf{u}\rangle
_{H_{0,\sigma }^{1}(\Omega )^{\prime }}=\frac{1}{2}\frac{\mathrm{d}}{\mathrm{%
d}t}\int_{\Omega }\rho (\phi )|\mathbf{u}|^{2}\,\mathrm{d}x+\frac{1}{2}%
\int_{\Omega }\partial _{t}\rho (\phi )|\mathbf{u}|^{2}\,\mathrm{d}x
\end{equation*}%
holds almost everywhere in $(\tau _{1},T)$. Then, since $\mathbf{u}\in
L^{2}(0,T;H_{0,\sigma }^{1}(\Omega ))$ is now allowed as a test function in %
\eqref{W:mom2}, we obtain%
\begin{align*}
\frac{1}{2}\frac{\mathrm{d}}{\mathrm{d}t}\int_{\Omega }\rho (\phi )|\mathbf{u%
}|^{2}\,\mathrm{d}x+\frac{1}{2}\int_{\Omega }\partial _{t}\rho (\phi )|%
\mathbf{u}|^{2}\,\mathrm{d}x-\int_{\Omega }\rho (\phi )\mathbf{u}\otimes
\mathbf{u}& :\nabla \mathbf{u}\,\mathrm{d}x \\
+\int_{\Omega }\nu (\phi )|D\mathbf{u}|^{2}\,\mathrm{d}x+\frac{\rho
_{1}-\rho _{2}}{2}\int_{\Omega }\mathbf{u}\cdot \left( \nabla \mu \cdot
\nabla \right) \mathbf{u}\,\mathrm{d}x& =\int_{\Omega }\mu \nabla \phi \cdot
\mathbf{u}\,\mathrm{d}x,
\end{align*}%
almost everywhere in $(\tau _{1},T)$. Thanks to \eqref{W:problem}, we
observe that
\begin{equation*}
\frac{\rho _{1}-\rho _{2}}{2}\int_{\Omega }\partial _{t}\phi \,\frac{|%
\mathbf{u}|^{2}}{2}\,\mathrm{d}x-\int_{\Omega }\rho (\phi )\mathbf{u}\cdot
\nabla \left( \frac{|\mathbf{u}|^{2}}{2}\right) \,\mathrm{d}x+\frac{\rho
_{1}-\rho _{2}}{2}\int_{\Omega }\nabla \mu \cdot \nabla \left( \frac{|%
\mathbf{u}|^{2}}{2}\right) \,\mathrm{d}x=0.
\end{equation*}%
Thus, after an integration in time, we reach
\begin{align*}
& \frac{1}{2}\int_{\Omega }\rho (\phi (t))|\mathbf{u}(t)|^{2}\,\mathrm{d}%
x+\int_{0}^{t}\left\Vert \sqrt{\nu (\phi (s))}D\mathbf{u}(s)\right\Vert
_{L^{2}(\Omega )}^{2}\,\mathrm{d}s-\int_{0}^{t}\int_{\Omega }\mu \nabla \phi
\cdot \mathbf{u}\,\mathrm{d}x\,\mathrm{d}s \\
& =\frac{1}{2}\int_{\Omega }\rho (\phi (\tau _{1}))|\mathbf{u}(\tau
_{1})|^{2}\,\mathrm{d}x,
\end{align*}%
for all $t\in \lbrack \tau _{1},T]$. In light of \eqref{W:en-id}, we find
the energy identity
\begin{equation*}
E(\mathbf{u}(t),\phi (t))+\int_{\tau_1}^{t}\left\Vert \sqrt{\nu (\phi (s))}D%
\mathbf{u}(s)\right\Vert _{L^{2}(\Omega )}^{2}+\Vert \nabla \mu (s)\Vert
_{L^{2}(\Omega )}^{2}\,\mathrm{d}s=E(\mathbf{u}(\tau _{1}),\phi (\tau
_{1})),\quad \forall \,\tau _{1}\leq t\leq T.
\end{equation*}

Next, owing to the energy identity, and exploiting \eqref{W:reg1} and %
\eqref{W:sp}, there exists $\tau_3 \in (\tau_2, \tau)$, such that
\begin{align}  \label{W:tau3}
\mathbf{u}(\tau_3)\in H^1_{0,\sigma}(\Omega),\quad \phi(\tau_3)\in
H^1(\Omega),\quad F^\prime(\phi(\tau_3))\in H^1(\Omega),
\quad \Vert\phi(\tau_3)\Vert_{L^\infty(\Omega)}\leq 1-\delta.
\end{align}
An application of Theorem \ref{MR:strong} yields the existence of a unique
global strong solution $\left( \widetilde{\mathbf{u}}, \widetilde{\Pi},
\widetilde{\phi}\right)$ to \eqref{syst2}-\eqref{bic} on $[\tau_3, \infty)$
departing from the initial datum $(\mathbf{u}(\tau_3),\phi(\tau_3))$. Our
aim is to show that actually $\left( \widetilde{\mathbf{u}}(t),\widetilde{%
\phi}(t)\right)$ coincides with $(\mathbf{u}(t),\phi(t))$ on $[\tau_3,T]$.
To achieve it, we argue similarly to the proof of the uniqueness for strong
``separated" solutions given in Section \ref{Uniqueness}. In particular, we
will only show the main differences. For the clarity of presentation, we set
$(\mathbf{u}_1(t),\phi_1(t))=(\mathbf{u}(t+\tau_3),\phi(t+\tau_3))$ for $%
t\in [0,T-\tau_3]$ and $(\mathbf{u}_2(t),\phi_2(t))=(\widetilde{\mathbf{u}}%
(t),\widetilde{\phi}(t))$ for $t\in [0,\infty)$. The initial data become $(%
\mathbf{u}_1(0),\phi_1(0))=(\mathbf{u}_2(0),\phi_2(0))=(\mathbf{u}%
(\tau_3),\phi(\tau_3))$. Furthermore, we recall that
\begin{equation}  \label{W:reg:phi_1}
\begin{cases}
\mathbf{u}_1 \in C_{\mathrm{w}}([0,T-\tau_3]; L^2_{\sigma})\cap
L^2(0,T-\tau_3; H^1_{0,\sigma}(\Omega)), \quad \partial_t(\rho(\phi_1)%
\mathbf{u}_1 ) \in L^{2}(0,T-\tau_3;H^1_{0,\sigma}(\Omega)^{\prime}), \\
\phi_1\in L^\infty(0,T-\tau_3;H^1(\Omega)) \cap L^q(
0,T-\tau_3;W^{1,p}(\Omega)),\quad q=\frac{2p}{p-2},\quad \forall \,
p\in(2,\infty), \\
\partial_t\phi_1\in L^\infty( 0,T-\tau_3;H^1(\Omega)^\prime)\cap
L^2(0,T-\tau_3; L^2(\Omega)), \\
\mu_1\in L^\infty(0,T-\tau_3; H^1(\Omega)) \cap L^2(0,T-\tau_3;H^2(\Omega))
\cap H^1(0,T-\tau_3; H^1(\Omega)^\prime), \\
 F^\prime(\phi_1)\in L^\infty(0,T-\tau_3; H^1(\Omega)).
\end{cases}%
\end{equation}
Thanks to \eqref{W:problem} and \eqref{W:reg:phi_1}, \eqref{W:mom2} can be
rewritten as follows
\begin{equation}  \label{W:mom3}
\begin{split}
&\langle \partial_t(\rho(\phi)\mathbf{u}),\mathbf{v} \rangle_{H^1_{0,%
\sigma}(\Omega)} -\int_\Omega \partial_t \rho(\phi_1) \, \mathbf{u}_1\cdot
\mathbf{w} \, \mathrm{d} x + \int_\Omega \left( \rho(\phi_1) \mathbf{u}_1
\cdot \nabla\right) \mathbf{u}_1 \cdot \mathbf{w} \, \mathrm{d} x \\
&\quad - \int_\Omega \left( \rho^{\prime}(\phi_1)\nabla \mu_1 \cdot \nabla
\right) \mathbf{u}_1 \cdot \mathbf{w} \, \mathrm{d} x +\int_\Omega
\nu(\phi_1) D \mathbf{u}_1 : D \mathbf{w} \, \mathrm{d} x =- \int_\Omega
\phi_1 (\nabla J\ast \phi_1) \cdot \mathbf{w} \, \mathrm{d} x,
\end{split}%
\end{equation}
for any $\mathbf{w} \in H^1_{0,\sigma}(\Omega)$ and almost any $t\in
(\tau_1,T)$. At the same time, the pair $(\mathbf{u}_2,\phi_2)$ satisfies %
\eqref{regg} as well as
\begin{equation}  \label{W:mom4}
\begin{split}
&\rho(\phi_2 ) \partial_t \mathbf{u}_2+\rho(\phi_2 )(\mathbf{u}_2
\cdot\nabla)\mathbf{u}_2 -\rho^\prime(\phi_2 )(\nabla\mu_2 \cdot\nabla)%
\mathbf{u}_2 -\mathrm{div}\, (\nu(\phi_2 )D\mathbf{u}_2 ) +\nabla
\Pi_2=\mu_2 \nabla\phi_2, \\[5pt]
&\partial_t \phi_2 +\mathbf{u}_2 \cdot \nabla\phi_2 =\Delta\mu_2,\quad
\mu_2= F^{\prime}(\phi_2)-J\ast \phi_2,
\end{split}
\end{equation}
almost everywhere in $\Omega \times (0,T-\tau_3)$. Moreover, it is essential
to notice that both $\phi_1$ and $\phi_2$ are strictly separated since the
initial concentration $\phi(\tau_2)$ is strictly separated, i.e.
$\| \phi_i(t)\|_{L^\infty(\Omega)}\leq 1-\delta$, for all $t\in[0,T-\tau_3]$, 
$i=1,2$, for some $\delta \in (0,1)$.

We now set $(\mathbf{u},\Phi )=(\mathbf{u}_{1}-\mathbf{u}_{2},\phi _{1}-\phi
_{2})$ in $[0,T-\tau _{3}]$. It follows from \eqref{W:mom3} and %
\eqref{W:mom4} that this pair satisfies the weak formulation:
\begin{equation}
\begin{split}
& \langle \partial _{t}(\rho (\phi _{1})\mathbf{u}),\mathbf{w}\rangle
_{H_{0,\sigma }^{1}(\Omega )}-\left( \partial _{t}\rho (\phi _{1})\,\mathbf{u%
}_{1},\mathbf{w}\right) +\left( \partial _{t}(\rho (\phi _{1})\mathbf{u}%
_{2}),\mathbf{w}\right) -\left( \rho (\phi _{2})\partial _{t}\mathbf{u}_{2},%
\mathbf{w}\right) \\
& \quad +\left( \rho (\phi _{1})(\mathbf{u}_{1}\cdot \nabla )\mathbf{u},%
\mathbf{w}\right) +\left( \rho (\phi _{1})(\mathbf{u}\cdot \nabla )\mathbf{u}%
_{2},\mathbf{w}\right) +\left( (\rho (\phi _{1})-\rho (\phi _{2}))(\mathbf{u}%
_{2}\cdot \nabla )\mathbf{u}_{2},\mathbf{w}\right) \\
& \quad -\frac{\rho _{1}-\rho _{2}}{2}\left( (\nabla \mu _{1}\cdot \nabla )%
\mathbf{u},\mathbf{w}\right) -\frac{\rho _{1}-\rho _{2}}{2}\left( (\nabla
\Theta \cdot \nabla )\mathbf{u}_{2},\mathbf{w}\right) +(\nu (\phi _{1})D%
\mathbf{u},\nabla \mathbf{w}) \\
& \quad +\left( (\nu (\phi _{1})-\nu (\phi _{2}))D\mathbf{u}_{2},\nabla
\mathbf{w}\right) =(\mu _{1}\nabla \Phi ,\mathbf{w})+(\Theta \nabla \phi
_{2},\mathbf{w}),
\end{split}
\label{vel_s2}
\end{equation}%
for any $\mathbf{w}\in H_{0,\sigma }^{1}(\Omega )$, in $(0,T-\tau _{3})$, and
\begin{equation}
\partial _{t}\Phi +\mathbf{u}_{1}\cdot \nabla \Phi +\mathbf{u}\cdot \nabla
\phi _{2}=\Delta \Theta ,\quad \Theta =F^{\prime }(\phi _{1})-F^{\prime
}(\phi _{2})-J\ast \Phi \quad \text{a.e. in }\Omega \times (0,T-\tau _{3}).
\label{phi_s2}
\end{equation}%
%
%
%
%
%
%
%
%
%
%
%
%
%
%
%
%
%
%
%
%
%
%
%
%
%
%
%
%
%
%
%
%
%
%
%
%
%
%
%
%
%
%
As next step, we take $\mathbf{w}=\mathbf{u}$ in \eqref{vel_s2} and apply
the chain rule formula in \cite[Lemma 5.3]{Frigeri} with $\widehat{\rho }%
=\rho (\phi _{1})$ and $\mathbf{u}$ on the interval $(0,T-\tau _{3})$.
Clearly, we have $\widehat{\rho }\in H^{1}(0,T-\tau _{3};L^{2}(\Omega ))$
and $\mathbf{u}\in L^{\infty }(0,T-\tau _{3};L_{\sigma }^{2}(\Omega ))\cap
L^{2}(0,T-\tau _{3};H_{0,\sigma }^{1}(\Omega ))$. We now claim that $%
\partial _{t}(\widehat{\rho }\mathbf{u})\in L^{2}(0,T;H_{0,\sigma
}^{1}(\Omega )^{\prime })$. In fact, by definition, we have
\begin{equation*}
\partial _{t}(\widehat{\rho }\mathbf{u})=\partial _{t}({\rho }(\phi _{1})%
\mathbf{u}_{1})-\partial _{t}({\rho }(\phi _{1})\mathbf{u}_{2}),
\end{equation*}%
where $\partial _{t}({\rho }(\phi _{1})\mathbf{u}_{1})\in
L^{2}(0,T;H_{0,\sigma }^{1}(\Omega )^{\prime })$ by the first part of the
proof, and $\partial _{t}({\rho }(\phi _{1})\mathbf{u}_{2})\in
L^{2}(0,T;H_{0,\sigma }^{1}(\Omega )^{\prime })$ by \eqref{W:reg:phi_1} and $%
\mathbf{u}_{2}\in L^{2}(0,\infty ;H_{0,\sigma }^{2}(\Omega ))\cap
H^{1}(0,\infty ;L_{\sigma }^{2}(\Omega ))$. Therefore, by using \cite[Lemma
5.3]{Frigeri}, the chain rule
\begin{equation*}
\langle \partial _{t}(\rho (\phi _{1})\mathbf{u}),\mathbf{u}\rangle
_{H_{0,\sigma }^{1}(\Omega )}=\frac{1}{2}\frac{\mathrm{d}}{\mathrm{d}t}%
\int_{\Omega }\rho (\phi )_{1}|\mathbf{u}|^{2}\,\mathrm{d}x+\frac{1}{2}%
\int_{\Omega }\partial _{t}\rho (\phi _{1})|\mathbf{u}|^{2}\,\mathrm{d}x
\end{equation*}%
holds almost everywhere in $(0,T-\tau _{3})$. Also, we observe that
\begin{equation*}
-\left( \partial _{t}\rho (\phi _{1})\mathbf{u}_{1},\mathbf{u}\right)
+\left( \partial _{t}(\rho (\phi _{1})\mathbf{u}_{2}),\mathbf{u}\right)
=-\int_{\Omega }\partial _{t}\rho (\phi _{1})|\mathbf{u}|^{2}\,\mathrm{d}%
x+(\rho (\phi _{1})\partial _{t}\mathbf{u}_{2},\mathbf{u}).
\end{equation*}%
Thus, exploiting the above relations and \eqref{U:rel1}-\eqref{U:rel3}, we
find (cf. \eqref{U-uni})
\begin{equation}
\begin{split}
& \frac{\mathrm{d}}{\mathrm{d}t}\int_{\Omega }\frac{\rho (\phi _{1})}{2}|%
\mathbf{u}|^{2}\,\mathrm{d}x+\int_{\Omega }\nu (\phi _{1})|D\mathbf{u}|^{2}\,%
\mathrm{d}x \\
& =-\int_{\Omega }(\rho (\phi _{1})-\rho (\phi _{2})\partial _{t}\mathbf{u}%
_{2}\cdot \mathbf{u}\,\mathrm{d}x-\int_{\Omega }\rho (\phi _{1})(\mathbf{u}%
\cdot \nabla )\mathbf{u}_{2}\cdot \mathbf{u}\,\mathrm{d}x-\int_{\Omega
}(\rho (\phi _{1})-\rho (\phi _{2}))(\mathbf{u}_{2}\cdot \nabla )\mathbf{u}%
_{2}\cdot \mathbf{u}\,\mathrm{d}x \\
& \quad -\int_{\Omega }(\nu (\phi _{1})-\nu (\phi _{2}))D\mathbf{u}%
_{2}:\nabla \mathbf{u}\,\mathrm{d}x-\frac{\rho _{1}-\rho _{2}}{2}%
\int_{\Omega }\Theta \Delta \mathbf{u}_{2}\cdot \mathbf{u}\,\mathrm{d}x-%
\frac{\rho _{1}-\rho _{2}}{2}\int_{\Omega }\Theta \nabla \mathbf{u}%
_{2}:\nabla \mathbf{u}\,\mathrm{d}x \\
& \quad -\int_{\Omega }\Phi (\nabla J\ast \phi _{1})\cdot \mathbf{u}\,%
\mathrm{d}x-\int_{\Omega }\phi _{2}(\nabla J\ast \Phi )\cdot \mathbf{u}\,%
\mathrm{d}x.
\end{split}
\label{W:uni1}
\end{equation}%
The rest of the argument follows by repeating line by line the proof of the
continuous dependence estimate for \textquotedblleft separated" strong
solutionsx given in Section \ref{Uniqueness}. As a
result, we obtain the following differential inequality
\begin{equation*}
\frac{\mathrm{d}}{\mathrm{d}t}\left( \int_{\Omega }\frac{\rho (\phi _{1})}{2}%
|\mathbf{u}|^{2}\,\mathrm{d}x+\frac{1}{2}\Vert \Phi \Vert _{L^{2}(\Omega
)}^{2}\right) \leq K(t)\left( \int_{\Omega }\frac{\rho (\phi
_{1})}{2}|\mathbf{u}|^{2}\,\mathrm{d}x+\frac{1}{2}\Vert \Phi \Vert
_{L^{2}(\Omega )}^{2}\right)
\end{equation*}%
almost everywhere in $(0,T-\tau _{3})$, where $K$ is defined as in \eqref{contdepconst}.

By the regularity of the strong solution $\mathbf{u}_{2}$, it immediately
follows that $\widetilde{K}\in L^{1}(0,T-\tau _{3})$. Then, we conclude from
the Gronwall lemma that $(\mathbf{u}_{1}(t),\phi _{1}(t))=(\mathbf{u}%
_{2}(t),\phi _{2}(t))$ on $[0,T-\tau _{3}]$, and thereby $(\mathbf{u}%
(t),\phi (t))=(\widetilde{\mathbf{u}}(t),\widetilde{\phi }(t))$ on $[\tau
_{3},T]$. So, setting $\Pi (t)=\widetilde{\Pi }(t)$ on $[\tau _{3},\infty )$%
, we have that $(\mathbf{u},\Pi ,\phi )$ is a \textquotedblleft separated"
strong solution on $[\tau ,\infty )\subset \lbrack \tau _{3},\infty )$.
\medskip

{\color{black} In the last part, we demonstrate that any weak solution
converges to an equilibrium, i.e., a minimum of the nonlocal Helmholtz free
energy \eqref{nonloc-en}. To this end, we first observe from the previous
part and Theorem \ref{ExistCahn}-(ii) that
\begin{equation}
\begin{cases}
\mathbf{u}\in BUC([1,\infty );H_{0,\sigma }^{1}(\Omega ))\cap L_{\mathrm{%
uloc}}^{2}([1,\infty );H_{0,\sigma }^{2}(\Omega ))\cap H_{\mathrm{uloc}%
}^{1}([1,\infty );L_{\sigma }^{2}(\Omega )), \\
\phi \in L^{\infty }(1,\infty ;L^{\infty }(\Omega ))\text{ such that }%
\sup_{t\in \lbrack 1,+\infty )}\Vert \phi (t)\Vert _{L^{\infty }(\Omega
)}\leq 1-\delta , \\
\phi \in C_{\mathrm{w}}([1,\infty );H^{1}(\Omega ))\cap L_{\mathrm{uloc}%
}^{q}([1,\infty );W^{1,p}(\Omega )),\quad q=\frac{2p}{p-2},\quad p\in
(2,\infty ), \\
\partial _{t}\phi \in L^{\infty }(1,\infty ;H^{1}(\Omega )^{\prime })\cap
L^{2}(1,\infty ;L^{2}(\Omega )), \\
\mu \in BUC([1,\infty );H^{1}(\Omega ))\cap L_{\mathrm{uloc}%
}^{2}([1,\infty );H^{2}(\Omega ))\cap H_{\mathrm{uloc}}^{1}([1,\infty
);L^{2}(\Omega )).
\end{cases}
\label{CE:reg}
\end{equation}%
In addition, the energy identity
\begin{equation}
E(\mathbf{u}(t),\phi (t))+\int_{1}^{t}\left\Vert \sqrt{\nu (\phi (s))}D%
\mathbf{u}(s)\right\Vert _{L^{2}(\Omega )}^{2}+\Vert \nabla \mu (s)\Vert
_{L^{2}(\Omega ))}^{2}\,\mathrm{d}s=E(\mathbf{u}(1),\phi (1))  \label{CE:ee}
\end{equation}%
holds for every $t\geq 1$. Thanks to the separation property, the
classical theory for second-order parabolic semilinear equations (cf. \cite[%
Corollary 5.6]{GGG2017} and the references therein) entails that there
exists $\gamma \in (0,1)$ such that
\begin{equation}
\phi \in BUC([2,\infty ),C^{\gamma }(\overline{\Omega })).
\label{Holder-reg}
\end{equation}%
Now we define the $\omega $-limit set of $(\mathbf{u},\phi )$ as
\begin{equation*}
\omega (\mathbf{u},\phi )=\left\{ (\mathbf{v},\Phi )\in L_{\sigma
}^{2}(\Omega )\times L^{\infty }(\Omega ):\exists \,t_{n}\nearrow \infty
\text{ such that }(\mathbf{u}(t_{n}),\phi (t_{n}))\rightarrow (\mathbf{v}%
,\Phi )\text{ in }L_{\sigma }^{2}(\Omega )\times L^{\infty }(\Omega
)\right\} .
\end{equation*}%
In light of \eqref{CE:reg} and \eqref{Holder-reg}, it follows that $\omega (%
\mathbf{u},\phi )$ is non-empty, compact and connected in $L_{\sigma
}^{2}(\Omega )\times L^{\infty }(\Omega )$. Also, we observe that any $(%
\mathbf{v},\Phi )\in \omega (\mathbf{u},\phi )$ is such that $\Vert \Phi \Vert
_{C(\overline{\Omega })}\leq 1-\delta $. }

Next, we claim that
\begin{equation}
\omega (\mathbf{u},\phi )\subset \left\{ (\mathbf{0},\phi _{\infty }):\phi
_{\infty }\in C^{\gamma }(\overline{\Omega })\text{ solves \eqref{Stat-CH}}%
\right\} .  \label{Omega-car}
\end{equation}%
Arguing as in \cite[Section 3]{AGG2022}, subtracting the Helmholtz free
energy equation (cf. \eqref{EE-nCH})
\begin{equation}
\mathcal{E}_{\mathrm{nloc}}(\phi (t))+\int_{1}^{t}\Vert \nabla \mu (s)\Vert
_{L^{2}(\Omega )}^{2}\,\mathrm{d}\tau +\int_{1}^{t}\int_{\Omega }\phi \,%
\mathbf{u}\cdot \nabla \mu \,\mathrm{d}x\,\mathrm{d}s=\mathcal{E}_{\mathrm{%
nloc}}(\phi (1)),\quad \forall \,t\in \lbrack 1,\infty ),  \label{Hee}
\end{equation}%
from \eqref{CE:ee} we have, for all $t\in (1,\infty )$,
\begin{equation}
E_{\mathrm{kin}}(\mathbf{u}(t),\phi (t))+\int_{1}^{t}\left\Vert \sqrt{\nu
(\phi (s))}D\mathbf{u}(s)\right\Vert _{L^{2}(\Omega )}^{2}\,\mathrm{d}s=E_{%
\mathrm{kin}}(\mathbf{u}(1),\phi (1))+\int_{1}^{t}\int_{\Omega }\phi \,%
\mathbf{u}\cdot \nabla \mu \,\mathrm{d}x\,\mathrm{d}\tau .  \label{CE:ee-kin}
\end{equation}%
Let us set $\varepsilon >0$. We observe from \eqref{CE:ee} that $\mathbf{u}%
\in L^{2}(0,\infty ;H_{0,\sigma }^{1}(\Omega ))$ and $\nabla \mu \in
L^{2}(0,\infty ;L^{2}(\Omega ;\mathbb{R}^{2}))$, there exists $T>0$ such
that $\Vert \mathbf{u}(T)\Vert _{L^{2}(\Omega )}\leq \varepsilon $ and $%
\Vert \nabla \mu \Vert _{L^{2}((T,\infty ;L^{2}(\Omega ))}\leq \varepsilon $%
. Then, we infer that
\begin{align*}
\max_{t\in \lbrack T,\infty )}\int_{\Omega }\frac{1}{2}\rho (\phi (t))|%
\mathbf{u}(t)|^{2}\,\mathrm{d}x& +\nu _{\ast }\int_{T}^{\infty }\Vert D%
\mathbf{u}(s)\Vert _{L^{2}(\Omega )}^{2}\,\mathrm{d}s \\
& \leq \int_{\Omega }\frac{1}{2}\rho (\phi (T))|\mathbf{u}(T)|^{2}\,\mathrm{d%
}x+\int_{T}^{\infty }\Vert \phi (s)\Vert _{L^{\infty }(\Omega )}\Vert
\mathbf{u}(s)\Vert _{L^{2}(\Omega )}\Vert \nabla \mu (s)\Vert _{L^{2}(\Omega
)}\,\mathrm{d}s \\
& \leq \frac{\rho ^{\ast }}{2}\Vert \mathbf{u}(T)\Vert _{L^{2}(\Omega )}^{2}+%
\frac{\nu _{\ast }}{2}\int_{T}^{\infty }\Vert D\mathbf{u}(s)\Vert
_{L^{2}(\Omega )}^{2}\,\mathrm{d}s+C\int_{T}^{\infty }\Vert \nabla \mu
(s)\Vert _{L^{2}(\Omega )}^{2}\,\mathrm{d}s \\
& \leq \frac{\rho ^{\ast }}{2}\varepsilon ^{2}+\frac{\nu _{\ast }}{2}%
\int_{T}^{\infty }\Vert D\mathbf{u}(s)\Vert _{L^{2}(\Omega )}^{2}\,\mathrm{d}%
s+C\varepsilon ^{2},
\end{align*}%
which gives that
\begin{equation*}
\max_{t\in \lbrack T,\infty )}\Vert \mathbf{u}(t)\Vert _{L^{2}(\Omega )}\leq
2C\varepsilon ,
\end{equation*}%
where $C$ is independent of $T$ and $\varepsilon $. Thus, $\mathbf{u}%
(t)\rightarrow \mathbf{0}$ as $t\nearrow \infty $.

Let us now consider $t_n \nearrow\infty$ and let $(\mathbf{u}%
(t_n),\phi(t_n)\rightarrow (\mathbf{0},\phi_\infty)$ in $L_\sigma^2(\Omega)%
\times L^\infty(\Omega)$ as $n\rightarrow \infty$. We now set $(\mathbf{u}%
_n(t),\phi_n(t))=(\mathbf{u}(t+t_n),\phi(t+t_n))$ for $t \in [1,\infty)$.
Clearly, $\mathbf{u}_n(t)\rightarrow \mathbf{0}$ in $L^2_\sigma(\Omega)$ as $%
n\rightarrow \infty$. Also, since $\mathbf{u}_n$ is uniformly bounded in $L^\infty(0,\infty; L^2_\sigma(\Omega))\cap L^2(0,\infty;
H^1_\sigma(\Omega))$, and exploiting Theorem \ref{ExistCahn}, \eqref{CE:reg}%
, \eqref{Holder-reg} and \eqref{Hee}, it is easy to deduce that
\begin{align*}
\| \phi_n\|_{L^\infty(0,T; C(\overline{\Omega}))}\leq 1-\delta, \quad \|
\phi_n\|_{L^2(0,T; H^1(\Omega))}\leq C, \quad \| \partial_t
\phi_n\|_{L^2(0,T; H^1(\Omega)^{\prime})}\leq C, \quad \| \mu_n\|_{L^2(0,T;
H^1(\Omega))}\leq C
\end{align*}
for some $C$ independent of $n$ and for any $T>0$, where $%
\mu_n(t)=\mu(t+t_n)$. Then, $\phi_n$ converges to $\phi^{\prime}$ strongly
in $L^2(0,T; L^2(\Omega))$ for any $T>0$ and $\mu_n$ converges to $%
\mu^{\prime}$ weakly in $L^2(0,T; H^1(\Omega))$ for any $T>0$. It follows
that $\phi^{\prime}$ is a weak solution to \eqref{nCH} in the sense of
Theorem \ref{ExistCahn} with chemical potential $\mu^{\prime}$,
divergence-free drift $\mathbf{v}=\mathbf{0}$, and initial datum $%
\phi^{\prime}(0)= \phi_\infty$. In addition, we have $E_{\mathrm{nloc}%
}(\phi_n(t))\rightarrow E_{\mathrm{nloc}}(\phi^{\prime}(t))$ for almost
every $t \in [1,\infty)$ as $n\rightarrow \infty$. However, since $\mathbf{u}
\in L^2(0,\infty;H^1_{0,\sigma}(\Omega))$ and $\nabla \mu \in L^2(0,\infty;
L^2(\Omega;\mathbb{R}^2))$, $\phi \mathbf{u}\cdot \nabla \mu \in
L^1(0,\infty; L^1(\Omega))$. Then, the limit $E_\infty:= \lim_{t\rightarrow
\infty}E_{\mathrm{nloc}}(\phi(t))$ exists and is unique. Therefore, we
infer that $E_{\mathrm{nloc}}(\phi^{\prime}(t))= E_\infty$ almost everywhere
in $[1,\infty)$. We conclude from the energy equality of $\phi^{\prime}$
that $\nabla \mu^{\prime}=0$ for almost every $t \in [1,\infty)$, and
thereby $\partial_t \phi^{\prime}(t)=0$ for almost every $t \in [1,\infty)$.
As such, $\phi^{\prime}(t) \equiv \phi_\infty$ for all $t \in [1,\infty)$
and
\begin{equation*}
F^{\prime}(\phi_\infty)-J\ast \phi_\infty=\mu_\infty \quad \text{ in }%
\Omega, \quad \text{for some } \mu_\infty \in \mathbb{R}.
\end{equation*}
This proves \eqref{Omega-car}. We are left to show that the whole weak solution converges to $(\mathbf{0}, \phi_\infty)$ as $t$ goes to $+\infty$.
We know that, thanks to \eqref{CE:ee}, the limit energy $E_\infty$ is constant on $\omega (\mathbf{u},\phi )$. Thus, we deduce from \eqref{ENERGYID} that,
for all $t>0$,
\begin{equation*}
E_\infty +\int_{t}^{+\infty }\left\Vert \sqrt{\nu (\phi (s))}D%
\mathbf{u}(s)\right\Vert _{L^{2}(\Omega )}^{2}+\Vert \nabla \mu (s)\Vert
_{L^{2}(\Omega ))}^{2}\,\mathrm{d}s=E(\mathbf{u}(t),\phi (t)) 
\end{equation*}
from which we deduce (see \eqref{TOTALEN})
\begin{equation}
\begin{split}
\int_{t}^{+\infty }\Vert \nabla \mu (s)\Vert_{L^{2}(\Omega ))}^{2}\,\mathrm{d}s&\leq E(\mathbf{u}(t),\phi (t)) - E_\infty \\
&= E(\mathbf{u}(t),\phi (t)) - E_{\mathrm{nloc}}(\phi_\infty)\\
&=\frac{1}{2}\int_{\Omega }\rho (\phi )|\mathbf{u}|^{2} + E_{\mathrm{nloc}}(\phi(t)) - E_{\mathrm{nloc}}(\phi_\infty).
\end{split}
\label{ENERGYINEQ}
\end{equation}
To conclude, we now need the real analyticity of the potential $F$ in order
to apply a suitable version of the {\L}ojasiewicz-Simon inequality (see, for instance \cite[Lemma 2.20]{GG}).
This amounts to say that there is $\theta\in (0,1/2]$ and $T_0>0$ sufficiently large such that
\begin{equation}
\vert E_{\mathrm{nloc}}(\phi (t)) - E_{\mathrm{nloc}}(\phi_\infty)\vert^{1-\theta} \leq C\Vert \mu - \overline{\mu}\Vert_{L^{2}(\Omega ))} 
\leq C\Vert \nabla\mu\Vert_{L^{2}(\Omega ))},
\quad \forall\,t\geq T_0,
\label{LJ1}
\end{equation}
for some $C>0$. Therefore, we get from \eqref{ENERGYINEQ} and \eqref{LJ1} that (see also \eqref{RHO-Jtilde})
\begin{equation*}
\vert E(\mathbf{u}(t),\phi (t)) - E_\infty\vert^{1-\theta} \leq C (\Vert \mathbf{u} \Vert_{L^{2}(\Omega ))} + \Vert \nabla\mu\Vert_{L^{2}(\Omega ))}), \quad\forall\, t\geq T_0. 
\end{equation*}
We can now argue as in \cite[Sec.~6]{AGG2022} to infer that  
$(\mathbf{u}(t),\phi (t))$ converges to $(\mathbf{0},\phi _{\infty})$ in $L_{\sigma }^{2}(\Omega )\times L^{\infty }(\Omega )$ as $t \to +\infty$.

\section{Proof of Theorem \protect\ref{Uniq:strong:H}: Improved continuous
dependence estimate for matched densities}

\label{Uniqueness:model H}

We consider two sets of initial data $(\mathbf{u}_{0}^{1},\phi _{0}^{1})$
and $(\mathbf{u}_{0}^{2},\phi _{0}^{2})$ which satisfy the assumptions of
Theorem \ref{MR:strong}, respectively, with constant density $\rho =\rho
_{1}=\rho _{2}>0$ (i.e., we consider the nonlocal Model H). We denote by $(%
\mathbf{u}_{1},\Pi _{1},\phi _{1})$ and $(\mathbf{u}_{2},\Pi _{2},\phi _{2})$
the corresponding strong solutions provided by Theorem \ref{MR:strong}. Let
us set
$\mathbf{u}=\mathbf{u}_{1}-\mathbf{u}_{2}$, $P=\Pi _{1}-\Pi _{2}$, $\Phi =\phi _{1}-\phi _{2}$, $\Theta =\mu _{1}-\mu _{2}=F^{\prime }(\phi
_{1})-F^{\prime }(\phi _{2})-J\ast \Phi$, which solve%
\begin{equation}
\begin{split}
&\rho \partial _{t}\mathbf{u}+\rho \mathrm{div}\,\left( \mathbf{u}_{1}\otimes
\mathbf{u}\right) +\rho \mathrm{div}\,\left( \mathbf{u}\otimes \mathbf{u}%
_{2}\right) -\mathrm{div}\,(\nu (\phi _{1})D\mathbf{u})-\mathrm{div}\,((\nu
(\phi _{1})-\nu (\phi _{2}))D\mathbf{u}_{2})+\nabla P \\
&\quad =\mu _{1}\nabla \Phi +\Theta \nabla \phi _{2}, \\[5pt]
&\partial _{t}\Phi +\mathbf{u}_{1}\cdot \nabla \Phi +\mathbf{u}\cdot \nabla
\phi _{2}=\Delta \Theta ,%
\end{split}
\label{U:modelH:2}
\end{equation}%
almost everywhere in $\Omega \times (0,\infty )$. Multiplying %
\eqref{U:modelH:2}$_{1}$ by $\mathbf{A}^{-1}\mathbf{u}$ and %
\eqref{U:modelH:2}$_{2}$ by $\mathcal{N}(\Phi -\overline{\Phi })$ (notice
that, by the conservation of mass, $\overline{\Phi }$ is constant),
integrating over $\Omega $ and adding the resulting equations together, we find the
identity
\begin{equation}
\begin{split}
& \frac{\mathrm{d}}{\mathrm{d}t}\left( \frac{\rho }{2}\Vert \mathbf{u}\Vert
_{\sharp }^{2}+\frac{1}{2}\Vert \Phi -\overline{\Phi }\Vert _{\ast
}^{2}\right) +(\Theta ,\Phi -\overline{\Phi })+(\nu (\phi _{1})D\mathbf{u}%
,\nabla \mathbf{A}^{-1}\mathbf{u}) \\
& =\rho (\mathbf{u}_{1}\otimes \mathbf{u},\nabla \mathbf{A}^{-1}\mathbf{u}%
)+\rho (\mathbf{u}\otimes \mathbf{u}_{2},\nabla \mathbf{A}^{-1}\mathbf{u}%
)-\left( (\nu (\phi _{1})-\nu (\phi _{2}))D\mathbf{u}_{2},\nabla \mathbf{A}%
^{-1}\mathbf{u}\right) \\
& \quad -(\mathbf{u}_{1}\cdot \nabla \Phi ,\mathcal{N}(\Phi -\overline{\Phi }%
))-(\mathbf{u}\cdot \nabla \phi _{2},\mathcal{N}(\Phi -\overline{\Phi }%
))+(\mu _{1}\nabla \Phi ,\mathbf{A}^{-1}\mathbf{u})+(\Theta \nabla \phi _{2},%
\mathbf{A}^{-1}\mathbf{u}).
\end{split}%
\end{equation}%
Arguing as in \cite[proof of Theorem 3.1]{GMT2019}, we observe that
\begin{align*}
\left( \nu (\phi _{1})D\mathbf{u},\nabla \mathbf{A}^{-1}\mathbf{u}\right) &
=\left( \nu (\phi _{1})\nabla \mathbf{u},D\mathbf{A}^{-1}\mathbf{u}\right)
=-\left( \mathbf{u},\mathrm{div}\,(\nu (\phi _{1})D\mathbf{A}^{-1}\mathbf{u}%
)\right) \\
& =-\left( \mathbf{u},\nu ^{\prime }(\phi _{1})D\mathbf{A}^{-1}\mathbf{u}%
\nabla \phi _{1}\right) -\frac{1}{2}\left( \mathbf{u},\nu (\phi _{1})\Delta
\mathbf{A}^{-1}\mathbf{u}\right) \\
& =-\left( \mathbf{u},\nu ^{\prime }(\phi _{1})D\mathbf{A}^{-1}\mathbf{u}%
\nabla \phi _{1}\right) +\frac{1}{2}\left( \mathbf{u},\nu (\phi _{1})\mathbf{%
u}\right) -\frac{1}{2}\left( \mathbf{u},\nu (\phi _{1})\nabla \pi \right) ,
\end{align*}%
where the artificial pressure $\pi \in L^{\infty }(0,T;H_{(0)}^{1}(\Omega ))$
is associated to the Stokes problem $-\Delta \mathbf{A}^{-1}\mathbf{U}%
+\nabla \pi =\mathbf{u}$ in $\Omega \times (0,\infty )$. Thanks to the above
relation and by \eqref{mu-diff} and \eqref{U:rel3}, we obtain the
differential inequality
\begin{equation}
\begin{split}
& \frac{\mathrm{d}}{\mathrm{d}t}\left( \frac{\rho }{2}\Vert \mathbf{u}\Vert
_{\sharp }^{2}+\frac{1}{2}\Vert \Phi -\overline{\Phi }\Vert _{\ast
}^{2}\right) +\frac{\nu _{\ast }}{2}\Vert \mathbf{u}\Vert _{L^{2}(\Omega
)}^{2}+\frac{3\alpha }{4}\Vert \phi \Vert _{L^{2}(\Omega )}^{2} \\
& \leq C\Vert \phi -\overline{\phi }\Vert _{\ast }^{2}+\left\vert \overline{%
\phi ^{1}}-\overline{\phi ^{2}}\right\vert \left( \Vert F^{\prime }(\phi
^{1})\Vert _{L^{1}(\Omega )}+\Vert F^{\prime }(\phi ^{2})\Vert
_{L^{1}(\Omega )}\right) +\rho (\mathbf{u}_{1}\otimes \mathbf{u},\nabla
\mathbf{A}^{-1}\mathbf{u}) \\
& \quad -\rho (\mathbf{u}\otimes \mathbf{u}_{2},\nabla \mathbf{A}^{-1}%
\mathbf{u})-\left( (\nu (\phi _{1})-\nu (\phi _{2}))D\mathbf{u}_{2},\nabla
\mathbf{A}^{-1}\mathbf{u}\right) +\left( \mathbf{u},\nu ^{\prime }(\phi
_{1})D\mathbf{A}^{-1}\mathbf{u}\nabla \phi _{1}\right) \\
& \quad +\frac{1}{2}\left( \mathbf{u},\nu (\phi _{1})\nabla \pi \right) -(%
\mathbf{u}_{1}\cdot \nabla \Phi ,\mathcal{N}(\Phi -\overline{\Phi }))-(%
\mathbf{u}\cdot \nabla \phi _{2},\mathcal{N}(\Phi -\overline{\Phi })) \\
& \quad -(\Phi (\nabla J\ast \phi _{1}),\mathbf{A}^{-1}\mathbf{u})-(\phi
_{2}(\nabla J\ast \Phi ),\mathbf{A}^{-1}\mathbf{u}).
\end{split}
\label{U:strong-dif}
\end{equation}%
By using \eqref{LADY} and \eqref{stoke}, we have
\begin{align*}
\left\vert \rho \int_{\Omega }\mathbf{u}_{1}\otimes \mathbf{u}:\nabla
\mathbf{A}^{-1}\mathbf{u}\,\mathrm{d}x\right\vert & \leq C\Vert \mathbf{u}%
_{1}\Vert _{L^{4}(\Omega )}\Vert \mathbf{u}\Vert _{L^{2}(\Omega )}\Vert
\nabla \mathbf{A}^{-1}\mathbf{u}\Vert _{L^{4}(\Omega )} \\
& \leq \frac{\nu _{\ast }}{24}\Vert \mathbf{u}\Vert _{L^{2}(\Omega
)}^{2}+C\Vert \mathbf{u}_{1}\Vert _{L^{4}(\Omega )}^{4}\Vert \mathbf{u}\Vert
_{\sharp }^{2}
\end{align*}%
and
\begin{align*}
\left\vert \rho \int_{\Omega }\mathbf{u}\otimes \mathbf{u}_{2}:\nabla
\mathbf{A}^{-1}\mathbf{u}\,\mathrm{d}x\right\vert & \leq \Vert \mathbf{u}%
_{2}\Vert _{L^{4}(\Omega )}\Vert \mathbf{u}\Vert _{L^{2}(\Omega )}\Vert
\nabla \mathbf{A}^{-1}\mathbf{u}\Vert _{L^{4}(\Omega )} \\
& \leq \frac{\nu _{\ast }}{24}\Vert \mathbf{u}\Vert _{L^{2}(\Omega
)}^{2}+C\Vert \mathbf{u}_{2}\Vert _{L^{4}(\Omega )}^{4}\Vert \mathbf{u}\Vert
_{\sharp }^{2}.
\end{align*}%
In a similar way, recalling the assumption \ref{nu}, we find
\begin{align*}
\left\vert \int_{\Omega }\left( \nu (\phi _{1})-\nu (\phi _{2})\right) D%
\mathbf{u}_{2}:\nabla \mathbf{A}^{-1}\mathbf{u}\,\mathrm{d}x\right\vert &
\leq C\Vert \Phi \Vert _{L^{2}(\Omega )}\Vert D\mathbf{u}_{2}\Vert
_{L^{4}(\Omega )}\Vert \nabla \mathbf{A}^{-1}\mathbf{u}\Vert _{L^{4}(\Omega
)} \\
& \leq \frac{\alpha }{16}\Vert \Phi \Vert _{L^{2}(\Omega )}^{2}+\frac{\nu
_{\ast }}{24}\Vert \mathbf{u}\Vert _{L^{2}(\Omega )}^{2}+C\Vert D\mathbf{u}%
_{2}\Vert _{L^{4}(\Omega )}^{4}\Vert \mathbf{u}\Vert _{\sharp }^{2}
\end{align*}%
and
\begin{align*}
\left\vert \int_{\Omega }\nu^{\prime }(\phi _{1})\mathbf{u}\cdot \left( D%
\mathbf{A}^{-1}\mathbf{u}\right) \nabla \phi _{1}\,\mathrm{d}x\right\vert &
\leq C\Vert \mathbf{u}\Vert _{L^{2}(\Omega )}\Vert \nabla \mathbf{A}^{-1}%
\mathbf{u}\Vert _{L^{4}(\Omega )}\Vert \nabla \phi _{1}\Vert _{L^{4}(\Omega
)} \\
& \leq \frac{\nu _{\ast }}{24}\Vert \mathbf{u}\Vert _{L^{2}(\Omega
)}^{2}+C\Vert \nabla \phi _{1}\Vert _{L^{4}(\Omega )}^{4}\Vert \mathbf{u}%
\Vert _{\sharp }^{2}.
\end{align*}%
Exploiting now Lemma \ref{press}, we obtain
\begin{align*}
\left\vert \frac{1}{2}\int_{\Omega }\nu (\phi _{1})\mathbf{u}\cdot \nabla
\pi \,\mathrm{d}x\right\vert & =\left\vert \frac{1}{2}\int_{\Omega }\nu
^{\prime }(\phi _{1})\mathbf{u}\cdot \nabla \phi _{1}\pi \,\mathrm{d}%
x\right\vert \\
& \leq C\Vert \mathbf{u}\Vert _{L^{2}(\Omega )}\Vert \nabla \phi _{1}\Vert
_{L^{4}(\Omega )}\Vert \pi \Vert _{L^{4}(\Omega )} \\
& \leq C\Vert \mathbf{u}\Vert _{L^{2}(\Omega )}\Vert \nabla \phi _{1}\Vert
_{L^{4}(\Omega )}\Vert \nabla \mathbf{A}^{-1}\mathbf{u}\Vert _{L^{2}(\Omega
)}^{\frac{1}{2}}\Vert \mathbf{u}\Vert _{L^{2}(\Omega )}^{\frac{1}{2}} \\
& \leq \frac{\nu _{\ast }}{24}\Vert \mathbf{u}\Vert _{L^{2}(\Omega
)}^{2}+C\Vert \nabla \phi _{1}\Vert _{L^{4}(\Omega )}^{4}\Vert \mathbf{u}%
\Vert _{\sharp }^{2}.
\end{align*}%
Next, arguing exactly as in \eqref{U:weak-u}, we get
\begin{equation*}
\left\vert \int_{\Omega }\mathbf{u}_{1}\cdot \nabla \Phi \cdot \mathcal{N}%
(\Phi -\overline{\Phi }))\,\mathrm{d}x\right\vert \leq \frac{\alpha }{16}%
\Vert \Phi \Vert _{L^{2}(\Omega )}^{2}+C\Vert \mathbf{u}_{1}\Vert
_{L^{4}(\Omega )}^{4}\Vert \Phi -\overline{\Phi }\Vert _{\ast
}^{2}+C\left\vert \overline{\Phi }\right\vert ^{2}.
\end{equation*}%
Since $\Vert \phi _{2}\Vert _{L^{\infty }(\Omega \times (0,\infty ))}\leq 1$%
, we infer that
\begin{equation*}
\left\vert \int_{\Omega }\mathbf{u}\cdot \nabla \phi _{2}\,\mathcal{N}(\Phi -%
\overline{\Phi }))\,\mathrm{d}x\right\vert =\left\vert \int_{\Omega }\phi
_{2}\mathbf{u}\cdot \nabla \mathcal{N}(\Phi -\overline{\Phi }))\,\mathrm{d}%
x\right\vert \leq \frac{\nu _{\ast }}{24}\Vert \mathbf{u}\Vert
_{L^{2}(\Omega )}^{2}+C\Vert \Phi -\overline{\Phi }\Vert _{\ast }^{2}.
\end{equation*}%
Lastly, by \ref{J-ass} and $\Vert \phi _{i}\Vert _{L^{\infty }(\Omega
\times (0,\infty ))}\leq 1$ for $i=1,2$, we deduce that
\begin{equation*}
\left\vert \int_{\Omega }\Phi (\nabla J\ast \phi _{1})\cdot \mathbf{A}^{-1}%
\mathbf{u}\,\mathrm{d}x\right\vert \leq \Vert \nabla J\ast \phi _{1}\Vert
_{L^{\infty }(\Omega )}\Vert \Phi \Vert _{L^{2}(\Omega )}\Vert \mathbf{A}%
^{-1}\mathbf{u}\Vert _{L^{2}(\Omega )}\leq \frac{\alpha }{16}\Vert \Phi
\Vert _{L^{2}(\Omega )}^{2}+C\Vert \mathbf{u}\Vert _{\sharp }^{2}
\end{equation*}%
and
\begin{equation*}
\left\vert \int_{\Omega }\phi _{2}(\nabla J\ast \Phi )\cdot \mathbf{A}^{-1}%
\mathbf{u}\,\mathrm{d}x\right\vert \leq \Vert \phi _{2}\Vert _{L^{\infty
}(\Omega )}\Vert \nabla J\ast \Phi \Vert _{L^{2}(\Omega )}\Vert \mathbf{A}%
^{-1}\mathbf{u}\Vert _{L^{2}(\Omega )}\leq \frac{\alpha }{16}\Vert \Phi
\Vert _{L^{2}(\Omega )}^{2}+C\Vert \mathbf{u}\Vert _{\sharp }^{2}.
\end{equation*}%
Combining \eqref{U:strong-dif} with above inequalities, we are led to
\begin{align}
& \frac{\mathrm{d}}{\mathrm{d}t}\left( \frac{\rho }{2}\Vert \mathbf{u}\Vert
_{\sharp }^{2}
+\frac{1}{2}\Vert \Phi -\overline{\Phi }\Vert _{\ast
}^{2}\right) 
 \leq \Lambda_1(t)\left( \frac{\rho }{2}\Vert \mathbf{u}\Vert _{\sharp }^{2}+\frac{1%
}{2}\Vert \Phi -\overline{\Phi }\Vert _{\ast }^{2}\right)
+ \Lambda_2(t)\left\vert
\overline{\Phi }\right\vert
+C\left\vert \overline{\phi }\right\vert ^{2},  \label{vv1bis}
\end{align}%
where
\begin{equation*}
\Lambda_1(t):=C\left( 1+\Vert \mathbf{u}_{1}(t)\Vert _{L^{4}(\Omega )}^{4}+\Vert
\mathbf{u}_{2}(t)\Vert _{L^{4}(\Omega )}^{4}+\Vert D\mathbf{u}_{2}(t)\Vert
_{L^{4}(\Omega )}^{4}+\Vert \nabla \phi _{1}(t)\Vert _{L^{4}(\Omega
)}^{4}\right)
\end{equation*}
and
\begin{equation*}
\Lambda_2(t):=  \Vert F^{\prime }(\phi _{1}(t))\Vert
_{L^{1}(\Omega )}+\Vert F^{\prime }(\phi _{2}(t))\Vert _{L^{1}(\Omega )}.
\end{equation*}
Owing to \eqref{regg}, it is easily seen that $\Lambda_j\in L^{1}(0,T)$ for $j=1,2$.
Thus, it follows from the Gronwall lemma that \eqref{weakstrong1} holds. The
proof is concluded.

\section{Proof of Theorem \protect\ref{stability}: Matched versus unmatched density}

\label{last} Let us fix $T>0$. Consider $(\mathbf{u},\Pi ,\phi )$ and $(%
\mathbf{u}_{H},\Pi _{H},\phi _{H})$ the strong solutions to the nonlocal AGG
model with density $\rho (\phi )$ and to the nonlocal Model H with constant
density $\overline{\rho }>0$ (i.e. \eqref{syst2}-\eqref{bic}) with $%
\overline{\rho }=\rho _{1}=\rho _{2}$, respectively. We assume that both $(%
\mathbf{u},\Pi ,\phi )$ and $(\mathbf{u}_{H},\Pi _{H},\phi _{H})$ originate
from the same initial datum $(\mathbf{u}_{0},\phi _{0})$. Therefore, setting $\mathbf{%
v}=\mathbf{u}-\mathbf{u}_{H}$, $Q=\Pi -\Pi _{H}$, $\Phi =\phi -\phi _{H}$, we have
\begin{equation}
\label{phih}
\begin{split}
&\left( \frac{\rho _{1}+\rho _{2}}{2}\right) \partial _{t}\mathbf{v}+\left(
\frac{\rho _{1}-\rho _{2}}{2}\phi \right) \partial _{t}\mathbf{u}+\left(
\frac{\rho _{1}+\rho _{2}}{2}-\overline{\rho }\right) \partial _{t}\mathbf{u}%
_{H}+\rho (\phi )(\mathbf{u}\cdot \nabla )\mathbf{u}-\overline{\rho }(%
\mathbf{u}_{H}\cdot \nabla ) \\
&\quad -\left( \frac{\rho _{1}-\rho _{2}}{2}\right) \left( (\nabla \mu \cdot \nabla
)\mathbf{u}\right) -\mathrm{div}\,(\nu (\phi )D\mathbf{v})-\mathrm{div}%
\,((\nu (\phi )-\nu (\phi _{H})D\mathbf{u}_{H})+\nabla Q \\
&=\mu \nabla \phi -\mu _{H}\nabla \phi _{H}, \\[5pt]
&\partial _{t}\Phi +\mathbf{u}\cdot \nabla \Phi +\mathbf{v}\cdot \nabla \phi
_{H}=\Delta M,%
\end{split}
\end{equation}
almost everywhere in $\Omega \times (0,T)$ where $M=\mu -\mu _{H}=F^{\prime
}(\phi )-F^{\prime }(\phi _{H})-J\ast \Phi $. Multiplying \eqref{phih}$_{1}$
by $\mathbf{A}^{-1}\mathbf{V}$ and integrating over $\Omega $, we find
\begin{equation}
\begin{split}
& \left( \frac{\rho _{1}+\rho _{2}}{4}\right) \frac{\mathrm{d}}{\mathrm{d}t}%
\Vert \mathbf{v}\Vert _{\sharp }^{2}+(\nu (\phi )D\mathbf{v},\nabla \mathbf{A%
}^{-1}\mathbf{v})=-\frac{\rho _{1}-\rho _{2}}{2}(\phi \partial _{t}\mathbf{u}%
,\mathbf{A}^{-1}\mathbf{v}) \\
& \quad -\left( \frac{\rho _{1}-\rho _{2}}{2}-\overline{\rho }\right)
(\partial _{t}\mathbf{u}_{H},\mathbf{A}^{-1}\mathbf{v})-((\rho (\phi )(%
\mathbf{u}\cdot \nabla )\mathbf{u}-\overline{\rho }(\mathbf{u}_{H}\cdot
\nabla )\mathbf{u}_{H}),\mathbf{A}^{-1}\mathbf{v}) \\
& \quad +\frac{\rho _{1}-\rho _{2}}{2}((\nabla \mu \cdot \nabla )\mathbf{u},%
\mathbf{A}^{-1}\mathbf{v})-((\nu (\phi )-\nu (\phi _{H}))D\mathbf{u}%
_{H},\nabla \mathbf{A}^{-1}\mathbf{v}) \\
& \quad -(\Phi (\nabla J\ast \phi ),\mathbf{A}^{-1}\mathbf{v})-(\phi
_{H}(\nabla J\ast \Phi ),\mathbf{A}^{-1}\mathbf{v}).
\end{split}
\label{stab:1}
\end{equation}%
Observe now that (cf. \cite[Section 3]{GMT2019})
\begin{equation}
(\nu (\phi )D\mathbf{v},\nabla \mathbf{A}^{-1}\mathbf{v})=-\left( \mathbf{v}%
,\nu ^{\prime }(\phi )D\mathbf{A}^{-1}\mathbf{v}\nabla \phi \right) +\frac{1%
}{2}\left( \mathbf{v},\nu (\phi )\mathbf{v}\right) +\frac{1}{2}\left(
\mathbf{v},\nu ^{\prime }(\phi )\tilde{\Pi}\nabla \phi \right) ,
\label{stab:nu}
\end{equation}%
where the artificial pressure is determined by the Stokes problem $-\Delta
\mathbf{A}^{-1}\mathbf{v}+\nabla \tilde{\Pi}=\mathbf{v}$ almost everywhere in $\Omega
\times (0,T)$. On the other hand, multiplying \eqref{phih}$_{2}$ by $%
\mathcal{N}\Phi $ (notice that $\overline{\Phi }\equiv 0$ by the
conservation of mass, since the two solutions originate from the same initial
data) and integrating over $\Omega $, we obtain (cf. \eqref{mu-diff})
\begin{equation}
\frac{1}{2}\frac{\mathrm{d}}{\mathrm{d}t}\Vert \Phi \Vert _{\ast }^{2}+\frac{%
3\alpha }{4}\Vert \Phi \Vert _{L^{2}(\Omega )}^{2}\leq C\Vert \Phi \Vert
_{\ast }^{2}+(\Phi \mathbf{u},\nabla \mathcal{N}\Phi )+(\phi _{H}\mathbf{v}%
,\nabla \mathcal{N}\varphi ).  \label{stab:2}
\end{equation}%
Here $C$ stands for a generic positive constant which may depend on given quantities and which may vary even within the same line.
Adding \eqref{stab:1} and \eqref{stab:2} together, and exploiting \eqref{stab:nu}%
, we end up with
\begin{equation}
\begin{split}
& \frac{\mathrm{d}}{\mathrm{d}t}\left( \left( \frac{\rho _{1}+\rho _{2}}{4}%
\right) \Vert \mathbf{v}\Vert _{\sharp }^{2}+\frac{1}{2}\Vert \Phi \Vert
_{\ast }^{2}\right) +\frac{\nu _{\ast }}{2}\Vert \mathbf{v}\Vert
_{L^{2}(\Omega )}^{2}+\frac{3\alpha }{4}\Vert \Phi \Vert _{L^{2}(\Omega
)}^{2} \\
& \leq -\frac{\rho _{1}-\rho _{2}}{2}(\phi \partial _{t}\mathbf{u},\mathbf{A}%
^{-1}\mathbf{v})-\left( \frac{\rho _{1}-\rho _{2}}{2}-\overline{\rho }%
\right) (\partial _{t}\mathbf{u}_{H},\mathbf{A}^{-1}\mathbf{v}) \\
& \quad -((\rho (\phi )(\mathbf{u}\cdot \nabla )\mathbf{u}-\overline{\rho }(%
\mathbf{u}_{H}\cdot \nabla )\mathbf{u}_{H}),\mathbf{A}^{-1}\mathbf{v})+\frac{%
\rho _{1}-\rho _{2}}{2}((\nabla \mu \cdot \nabla )\mathbf{u},\mathbf{A}^{-1}%
\mathbf{v}) \\
& \quad -((\nu (\phi )-\nu (\phi _{H}))D\mathbf{u}_{H},\nabla \mathbf{A}^{-1}%
\mathbf{v})+\left( \mathbf{v},\nu ^{\prime }(\phi )D\mathbf{A}^{-1}\mathbf{v}%
\nabla \phi \right) -\frac{1}{2}\left( \mathbf{v},\nu ^{\prime }(\phi )%
\tilde{\Pi}\nabla \phi \right) \\
& \quad +(\Phi \mathbf{u},\nabla \mathcal{N}\Phi )+(\phi _{H}\mathbf{v}%
,\nabla \mathcal{N}\Phi )-(\Phi (\nabla J\ast \phi ),\mathbf{A}^{-1}\mathbf{v%
})-(\phi _{H}(\nabla J\ast \Phi ),\mathbf{A}^{-1}\mathbf{v}).
\end{split}
\label{stab:3}
\end{equation}%
Since $\Vert \phi \Vert _{L^{\infty }(\Omega \times (0,T))}\leq 1$, we have
\begin{equation*}
\left\vert \frac{\rho _{1}-\rho _{2}}{2}\int_{\Omega }\phi \partial _{t}%
\mathbf{u}\cdot \mathbf{A}^{-1}\mathbf{v}\,\right\vert \leq C\left\vert
\frac{\rho _{1}-\rho _{2}}{2}\right\vert ^{2}\Vert \partial _{t}\mathbf{u}%
\Vert _{L^{2}(\Omega )}^{2}+C\Vert \mathbf{v}\Vert _{\sharp }^{2},
\end{equation*}%
and
\begin{equation*}
\left\vert \left( \frac{\rho _{1}-\rho _{2}}{2}-\overline{\rho }\right)
\int_{\Omega }\partial _{t}\mathbf{u}_{H}\cdot \mathbf{A}^{-1}\mathbf{v}\,%
\mathrm{d}x\right\vert \leq C\Vert \mathbf{v}\Vert _{\sharp
}^{2}+C\left\vert \frac{\rho _{1}-\rho _{2}}{2}-\overline{\rho }\right\vert
^{2}\Vert \partial _{t}\mathbf{u}_{H}\Vert _{L^{2}(\Omega )}^{2}.
\end{equation*}%
Integrating by parts, we find
\begin{align*}
& -\int_{\Omega }(\rho (\phi )(\mathbf{u}\cdot \nabla )\mathbf{u}-\overline{%
\rho }(\mathbf{u}_{H}\cdot \nabla )\mathbf{u}_{H})\cdot \mathbf{A}^{-1}%
\mathbf{v}\,\mathrm{d}x \\
& =-\int_{\Omega }\rho (\phi )(\mathbf{v}\cdot \nabla )\mathbf{u}\cdot
\mathbf{A}^{-1}\mathbf{v}\,\mathrm{d}x-\int_{\Omega }(\rho (\phi )-\overline{%
\rho })(\mathbf{u}_{H}\cdot \nabla )\mathbf{u}_{H}\cdot \mathbf{A}^{-1}%
\mathbf{v}\,\mathrm{d}x-\int_{\Omega }\rho (\phi )(\mathbf{u}_{H}\cdot
\nabla )\mathbf{v}\cdot \mathbf{A}^{-1}\mathbf{v}\,\mathrm{d}x \\
& =-\int_{\Omega }\rho (\phi )(\mathbf{v}\cdot \nabla )\mathbf{u}\cdot
\mathbf{A}^{-1}\mathbf{v}\,\mathrm{d}x-\int_{\Omega }(\rho (\phi )-\overline{%
\rho })(\mathbf{u}_{H}\cdot \nabla )\mathbf{u}_{H}\cdot \mathbf{A}^{-1}%
\mathbf{v}\,\mathrm{d}x+\int_{\Omega }\rho (\phi )(\mathbf{u}_{H}\cdot
\nabla )\mathbf{A}^{-1}\mathbf{v}\cdot \mathbf{v}\,\mathrm{d}x \\
& \quad +\frac{\rho _{1}-\rho _{2}}{2}\int_{\Omega }\left( \nabla \phi \cdot
\mathbf{u}_{H}\right) \left( \mathbf{v}\cdot \mathbf{A}^{-1}\mathbf{v}%
\right) \,\mathrm{d}x.
\end{align*}%
Then, recalling that $\mathbf{u}_{H}\in L^{\infty }(0,\infty ;H_{0,\sigma
}^{1}(\Omega ))$, we obtain
\begin{align*}
\left\vert \int_{\Omega }\rho (\phi )(\mathbf{v}\cdot \nabla )\mathbf{u}%
\cdot \mathbf{A}^{-1}\mathbf{v}\,\mathrm{d}x\right\vert & \leq C\Vert
\mathbf{v}\Vert _{L^{2}(\Omega )}\Vert \nabla \mathbf{u}\Vert _{L^{4}(\Omega
)}\Vert \mathbf{A}^{-1}\mathbf{v}\Vert _{L^{4}(\Omega )} \\
& \leq \frac{\nu _{\ast }}{28}\Vert \mathbf{v}\Vert _{L^{2}(\Omega
)}^{2}+C\Vert \nabla \mathbf{u}\Vert _{L^{4}(\Omega )}^{2}\Vert \mathbf{v}%
\Vert _{\sharp }^{2}
\end{align*}%
and
\begin{align*}
\left\vert -\int_{\Omega }(\rho (\phi )-\overline{\rho })(\mathbf{u}%
_{H}\cdot \nabla )\mathbf{u}_{H}\cdot \mathbf{A}^{-1}\mathbf{v}\,\mathrm{d}%
x\right\vert & \leq \Vert \rho (\phi )-\overline{\rho }\Vert _{L^{\infty
}(\Omega )}\Vert \mathbf{u}_{H}\Vert _{L^{4}(\Omega )}\Vert \nabla \mathbf{u}%
_{H}\Vert _{L^{2}(\Omega )}\Vert \mathbf{A}^{-1}\mathbf{v}\Vert
_{L^{4}(\Omega )} \\
& \leq C\left( \left\vert \frac{\rho _{1}-\rho _{2}}{2}\right\vert
+\left\vert \frac{\rho _{1}+\rho _{2}}{2}-\overline{\rho }\right\vert
\right) ^{2}+C\Vert \mathbf{v}\Vert _{\sharp }^{2}.
\end{align*}%
On the other hand, by \eqref{LADY}, \eqref{W14bis}, \eqref{W14tris} and %
\eqref{stoke}, we infer that
\begin{align*}
\left\vert \int_{\Omega }\rho (\phi )(\mathbf{u}_{H}\cdot \nabla )\mathbf{A}%
^{-1}\mathbf{v}\cdot \mathbf{v}\,\mathrm{d}x\right\vert & \leq C\Vert
\mathbf{u}_{H}\Vert _{L^{4}(\Omega )}\Vert \nabla \mathbf{A}^{-1}\mathbf{v}%
\Vert _{L^{4}(\Omega )}\Vert \mathbf{v}\Vert _{L^{2}(\Omega )} \\
& \leq \frac{\nu _{\ast }}{28}\Vert \mathbf{v}\Vert _{L^{2}(\Omega
)}^{2}+C\Vert \mathbf{v}\Vert _{\sharp }^{2}
\end{align*}%
and
\begin{align*}
\left\vert \int_{\Omega }\left( \nabla \phi \cdot \mathbf{u}_{H}\right)
\left( \mathbf{v}\cdot \mathbf{A}^{-1}\mathbf{v}\right) \,\mathrm{d}%
x\right\vert & \leq \left\vert \frac{\rho _{1}-\rho _{2}}{2}\right\vert
\Vert \mathbf{u}_{H}\Vert _{L^{4}(\Omega )}\Vert \mathbf{v}\Vert
_{L^{2}(\Omega )}\Vert \mathbf{A}^{-1}\mathbf{v}\Vert _{L^{\infty }(\Omega
)}\Vert \nabla \phi \Vert _{L^{4}(\Omega )} \\
& \leq C\left\vert \frac{\rho _{1}-\rho _{2}}{2}\right\vert \Vert \mathbf{v}%
\Vert _{L^{2}(\Omega )}^{\frac{3}{2}}\Vert \mathbf{v}\Vert _{\sharp }^{\frac{%
1}{2}}\Vert \nabla \phi \Vert _{L^{4}(\Omega )} \\
& \leq \frac{\nu _{\ast }}{28}\Vert \mathbf{v}\Vert _{L^{2}(\Omega
)}^{2}+C\Vert \nabla \phi \Vert _{L^{4}(\Omega )}^{4}\Vert \mathbf{v}\Vert
_{\sharp }^{2},
\end{align*}%
as well as
\begin{align*}
\left\vert \int_{\Omega }(\nu (\phi )-\nu (\phi _{H}))D\mathbf{u}_{H}:\nabla
\mathbf{A}^{-1}\mathbf{v}\,\mathrm{d}x\right\vert & \leq C\Vert \Phi \Vert
_{L^{2}(\Omega )}\Vert D\mathbf{u}_{H}\Vert _{L^{4}(\Omega )}\Vert \nabla
\mathbf{A}^{-1}\mathbf{v}\Vert _{L^{4}(\Omega )} \\
& \leq C\Vert \Phi \Vert _{L^{2}(\Omega )}\Vert D\mathbf{u}_{H}\Vert
_{L^{4}(\Omega )}\Vert \mathbf{v}\Vert _{\sharp }^{\frac{1}{2}}\Vert \mathbf{%
v}\Vert _{L^{2}(\Omega )}^{\frac{1}{2}} \\
& \leq \frac{\alpha }{12}\Vert \Phi \Vert _{L^{2}(\Omega )}^{2}+\frac{\nu
_{\ast }}{28}\Vert \mathbf{v}\Vert _{L^{2}(\Omega )}^{2}+C\Vert D\mathbf{u}%
_{H}\Vert _{L^{4}(\Omega )}^{4}\Vert \mathbf{v}\Vert _{\sharp }^{2}.
\end{align*}%
In a similar way, by using assumption \ref{nu}, we also get
\begin{align*}
\left\vert \int_{\Omega }\nu ^{\prime }(\phi )\mathbf{v}\cdot (D\mathbf{A}%
^{-1}\mathbf{v}\nabla \phi )\,\mathrm{d}x\right\vert & \leq C\Vert \mathbf{v}%
\Vert _{L^{2}(\Omega )}\Vert \nabla \mathbf{A}^{-1}\mathbf{v}\Vert
_{L^{4}(\Omega )}\Vert \nabla \phi \Vert _{L^{4}(\Omega )} \\
& \leq \frac{\nu _{\ast }}{28}\Vert \mathbf{v}\Vert _{L^{2}(\Omega
)}^{2}+C\Vert \nabla \phi \Vert _{L^{4}(\Omega )}^{4}\Vert \mathbf{v}\Vert
_{\sharp }^{2}.
\end{align*}%
Since $\mathbf{u}\in L^{\infty }(0,\infty ;H_{0,\sigma }^{1}(\Omega ))$, it
is easily seen that
\begin{align*}
\left\vert \frac{\rho _{1}-\rho _{2}}{2}\int_{\Omega }(\nabla \mu \cdot
\nabla )\mathbf{u}\cdot \mathbf{A}^{-1}\mathbf{v}\,\mathrm{d}x\right\vert &
\leq \left\vert \frac{\rho _{1}-\rho _{2}}{2}\right\vert \Vert \nabla \mu
\Vert _{L^{4}(\Omega )}\Vert \nabla \mathbf{u}\Vert _{L^{2}(\Omega )}\Vert
\mathbf{A}^{-1}\mathbf{v}\Vert _{L^{4}(\Omega )} \\
& \leq C\Vert \mathbf{v}\Vert _{\sharp }^{2}+C\left\vert \frac{\rho
_{1}-\rho _{2}}{2}\right\vert ^{2}\Vert \nabla \mu \Vert _{L^{4}(\Omega
)}^{2}.
\end{align*}%
Now, exploiting Lemma \ref{press}, we get
\begin{align*}
\left\vert \int_{\Omega }\nu ^{\prime }(\phi )\mathbf{v}\cdot \nabla \phi
\tilde{\Pi}\,\mathrm{d}x\right\vert & \leq C\Vert \mathbf{V}\Vert
_{L^{2}(\Omega )}\Vert \nabla \phi \Vert _{L^{4}(\Omega )}\Vert \tilde{\Pi}%
\Vert _{L^{4}(\Omega )} \\
& \leq C\Vert \mathbf{v}\Vert _{L^{2}(\Omega )}\Vert \nabla \phi \Vert
_{L^{4}(\Omega )}\Vert \nabla \mathbf{A}^{-1}\mathbf{v}\Vert _{L^{2}(\Omega
)}^{\frac{1}{2}}\Vert \mathbf{v}\Vert _{L^{2}(\Omega )}^{\frac{1}{2}} \\
& \leq \frac{\nu _{\ast }}{28}\Vert \mathbf{v}\Vert _{L^{2}(\Omega
)}^{2}+C\Vert \nabla \phi \Vert _{L^{4}(\Omega )}^{4}\Vert \mathbf{v}\Vert
_{\sharp }^{2}.
\end{align*}%
Finally, as in the proof of Theorem \ref{Uniq:strong:H}, we have
\begin{equation*}
\left\vert \int_{\Omega }\Phi \mathbf{u}\cdot \nabla \mathcal{N}\Phi \,%
\mathrm{d}x\right\vert \leq \frac{\alpha }{12}\Vert \Phi \Vert
_{L^{2}(\Omega )}^{2}+C\Vert \mathbf{u}\Vert _{L^{4}(\Omega )}^{4}\Vert \Phi
\Vert _{\ast }^{2}
\end{equation*}%
and
\begin{equation*}
\left\vert \int_{\Omega }\phi _{H}\mathbf{v}\cdot \nabla \mathcal{N}\Phi \,%
\mathrm{d}x\right\vert \leq \frac{\nu _{\ast }}{28}\Vert \mathbf{v}\Vert
_{L^{2}(\Omega )}^{2}+C\Vert \Phi \Vert _{\ast }^{2},
\end{equation*}%
as well as
\begin{equation*}
\left\vert \int_{\Omega }\Phi (\nabla J\ast \phi )\cdot \mathbf{A}^{-1}%
\mathbf{v}\,\mathrm{d}x\right\vert +\left\vert \int_{\Omega }\phi
_{H}(\nabla J\ast \Phi )\cdot \mathbf{A}^{-1}\mathbf{v}\,\mathrm{d}%
x\right\vert \leq \frac{\alpha }{12}\Vert \Phi \Vert _{L^{2}(\Omega
)}^{2}+C\Vert \mathbf{v}\Vert _{\sharp }^{2}.
\end{equation*}%
Combining the above estimates, we arrive at
\begin{align*}
& \frac{\mathrm{d}}{\mathrm{d}t}\left( \left( \frac{\rho _{1}+\rho _{2}}{4}%
\right) \Vert \mathbf{v}\Vert _{\sharp }^{2}+\frac{1}{2}\Vert \Phi \Vert
_{\ast }^{2}\right) +\frac{\nu _{\ast }}{4}\Vert \mathbf{v}\Vert
_{L^{2}(\Omega )}^{2}+\frac{\alpha }{2}\Vert \Phi \Vert _{L^{2}(\Omega )}^{2}
\\
& \quad \leq R_{1}\left( \left( \frac{\rho _{1}+\rho _{2}}{4}\right) \Vert
\mathbf{v}\Vert _{\sharp }^{2}+\frac{1}{2}\Vert \Phi \Vert _{\ast
}^{2}\right) +R_{2}\left( \left\vert \frac{\rho _{1}-\rho _{2}}{2}%
\right\vert ^{2}+\left\vert \frac{\rho _{1}+\rho _{2}}{2}-\overline{\rho }%
\right\vert ^{2}\right) ,
\end{align*}%

where
\begin{align*}
R_{1}& :=C\left( 1+\Vert \nabla \mathbf{u}\Vert _{L^{4}(\Omega )}^{2}+\Vert
\nabla \mathbf{u}_{H}\Vert _{L^{4}(\Omega )}^{4}+\Vert \nabla \phi \Vert
_{L^{4}(\Omega )}^{4}\right) , \\
R_{2}& :=C\left( 1+\Vert \partial _{t}\mathbf{u}_{H}\Vert _{L^{2}(\Omega
)}^{2}+\Vert \partial _{t}\mathbf{u}\Vert _{L^{2}(\Omega )}^{2}+\Vert \nabla
\mu \Vert _{L^{4}(\Omega )}^{2}\right) .
\end{align*}%
Notice that $C$ depends on the norm of the initial
data and the time $T$. An application of the Gronwall lemma yields
\begin{equation*}
\Vert \mathbf{v}(t)\Vert _{\sharp }^{2}+\Vert \Phi (t)\Vert _{\ast }^{2}\leq
\dfrac{\left( \left\vert \frac{\rho _{1}-\rho _{2}}{2}\right\vert
^{2}+\left\vert \frac{\rho _{1}+\rho _{2}}{2}-\overline{\rho }\right\vert
^{2}\right) }{\min \{\frac{\rho _{1}+\rho _{2}}{4},\frac{1}{2}\}}\int_{0}^{t}%
\mathrm{e}^{\int_{s}^{t}R_{1}(\tau )\, \d\tau } R_{2}(s)\,\mathrm{d}%
s,\quad \forall \,t\in \lbrack 0,T].
\end{equation*}%
Therefore, in light of \eqref{regg}, the above inequality implies the
desired conclusion (\ref{stab}). The proof is finished.

\appendix

\section{Global \textquotedblleft separated" solutions to \eqref{nCH} with smooth divergence free drift}

\label{Appendix} \setcounter{equation}{0}

In this Appendix, we establish the existence of global \textit{regular}
solutions to the nonlocal Cahn-Hilliard equation with \textit{smooth}
divergence-free drift. More precisely, we aim to construct solutions to %
\eqref{nCH} satisfying the separation property for all times.

\begin{theorem}
\label{REG-APPR-nCH} Let the assumptions \ref{Omega}-\ref{m-ass} hold
and let $T>0$ be given. If $\mathbf{u}\in \mathcal{D}(0,T;C_{0,\sigma
}^{\infty }(\Omega ;\mathbb{R}^{2}))$ and $\phi _{0}\in H^{1}(\Omega )\cap
L^{\infty }(\Omega )$ with $\Vert \phi _{0}\Vert _{L^{\infty }(\Omega )}<1$
and $|\overline{\phi _{0}}|<1$, then there exists a solution $\phi$ to \eqref{nCH}
such that 
\begin{equation}
\begin{cases}
\phi \in L^{\infty }(0,T;H^{1}(\Omega )\cap L^{\infty }(\Omega )):\quad
\sup_{t\in \lbrack 0,T]}\Vert \phi (t)\Vert _{L^{\infty }(\Omega )}\leq
1-\delta , \\
\phi \in L^{q}(0,T;W^{1,p}(\Omega )),\quad q=\frac{2p}{p-2},\quad \forall
\,p\in (2,\infty ), \\
\partial _{t}\phi \in L^{\infty }(0,T;H^{1}(\Omega )^{\prime })\cap
L^{2}(0,T;L^{2}(\Omega )), \\
\mu \in C([0,T];H^{1}(\Omega ))\cap L^{2}(0,T;H^{2}(\Omega ))\cap
H^{1}(0,T;L^{2}(\Omega )),
\end{cases}
\label{Reg-approx-nch}
\end{equation}%
where $\delta \in (0,1)$ depends on  $T$, $\mathbf{u}$ and $\phi
_{0}$. Any solution satisfying the above properties is a strong solution, that is,
\begin{equation}
\begin{split}
&\partial _{t}\phi +\mathbf{u}\cdot \nabla \phi =\Delta \mu ,\quad \mu
=-J\ast \phi +F^{\prime }(\phi ), \quad \text{ a.e. in } \Omega \times (0,T)\\
& \partial _{\mathbf{n}}\mu =0, \quad \text{ a.e. on } \partial \Omega \times (0,T), \quad
 \phi (\cdot,0)=\phi _{0}, \quad \text{ a.e. in } \Omega .
 \end{split}
 \label{CH-strong_sense}
 \end{equation}

\end{theorem}

\begin{proof}

Let us first introduce the Yosida approximation of the singular potential $F$%
. For any $\lambda >0$, we define $F_{\lambda }:\mathbb{R}\rightarrow
\mathbb{R}$ such that $F_{\lambda }(s)=\frac{\lambda }{2}|A_{\lambda
}s|^{2}+F(J_{\lambda }(s))$ where $J_{\lambda }=\left( I+\lambda F^{\prime
}\right) ^{-1}$ and $A_{\lambda }=\frac{1}{\lambda }(I-J_{\lambda })$. We
report the following main properties (see \cite{Brezis} and \cite[Section 3]%
{GGG2017}):

\begin{enumerate}
\item[(a1)] \label{F1} for any $\lambda >0$, $F_{\lambda }\in C_{\text{loc}%
}^{2,1}(\mathbb{R})$ such that $F_{\lambda }(0)=F_{\lambda }^{\prime }(0)=0$;

\item[(a2)] \label{F2} for any $0<\lambda ^{\sharp }\leq 1$, there exists $%
C_{\sharp }>0$ such that
\begin{equation}
F_{\lambda }(s)\geq \frac{1}{4\lambda ^{\sharp }}s^{2}-C_{\sharp },\quad
\forall \,s\in \mathbb{R},\ \forall \,\lambda \in (0,\lambda ^{\sharp }];
\end{equation}

\item[(a3)] \label{F3} $F_{\lambda }$ is convex with
\begin{equation*}
F_{\lambda }^{\prime \prime }(s)\geq \frac{\alpha }{1+\alpha },\quad \forall
\,s\in \mathbb{R};
\end{equation*}

\item[(a4)] \label{F4} for any $\lambda >0$, $F_{\lambda }^{\prime }$ is
Lipschitz on $\mathbb{R}$ with constant $\frac{1}{\lambda }$;

\item[(a5)] \label{F5} as $\lambda \rightarrow 0$, $F_{\lambda
}(s)\rightarrow F(s)$ for all $s\in \mathbb{R}$, $|F_{\lambda }^{\prime
}(s)|\rightarrow |F^{\prime }(s)|$ for $s\in (-1,1)$ and $F_{\lambda
}^{\prime }$ converges uniformly to $F^{\prime }$ on any compact subset of $%
(-1,1)$; furthermore, $|F_{\lambda }^{\prime }(s)|\rightarrow +\infty $ for
every $|s|\geq 1$.
\end{enumerate}

\noindent Let us now fix $\lambda ^{\star }$ to be positive and sufficiently
small. We will choose $\lambda ^{\star }$ will be defined later on. We
claim that, for any $\lambda \in (0,\lambda ^{\star })$, there exists a
function $\phi _{\lambda }$ such that%
\begin{equation}
\phi _{\lambda }\in L^{\infty }(0,T;L^{2}(\Omega ))\cap
L^{2}(0,T;H^{1}(\Omega ))\cap H^{1}(0,T;H^{1}(\Omega )^{\prime }),
\label{Reg-lambda}
\end{equation}%
%
%
%
%
%
%
%
%
%
%
%
%
%
%
%
%
%
%
%
%
%
%
%
%
%
%
%
%
%
%
%
%
%
%
%
%
%
%
%
%
which satisfies the variational formulation
\begin{equation}
\langle \partial _{t}\phi _{\lambda },v\rangle -(\phi _{\lambda }\,\mathbf{u}%
,\nabla v)+(\nabla \mu _{\lambda },\nabla v)=0,\quad \forall \,v\in
H^{1}(\Omega ),\ \text{a.e. in }(0,T),  \label{w-nCH-l1}
\end{equation}%
where $\mu _{\lambda }=F_{\lambda }^{\prime }(\phi _{\lambda })-J\ast \phi
_{\lambda }\in L^{2}(0,T;H^{1}(\Omega ))$, as well as $\phi (\cdot ,0)=\phi
_{0}(\cdot )$ in $\Omega $. The proof of the existence of the approximating
solution $\phi _{\lambda }$ is carried out by the Galerkin scheme. The
argument is rather standard and we refer the reader to \cite{DPGG2018,
GGG2017}. \medskip

\noindent \textbf{Conservation of mass and energy estimates.} First, taking $%
v=1$ in \eqref{w-nCH-l1}, we obtain that $\overline{\phi _{\lambda }}(t)=%
\overline{\phi _{0}}$ for all $t\in \lbrack 0,T]$. Since $\Vert \phi
_{0}\Vert _{L^{\infty }(\Omega )}<1$ by assumption, we clearly infer that $|%
\overline{\phi _{\lambda }}(t)|=|\overline{\phi _{0}}|<1$ for all $t\in
\lbrack 0,T]$. Next, we define the energy functional $\mathcal{E}_{\lambda
}:L^{2}(\Omega )\rightarrow \mathbb{R}$ as follows
\begin{equation*}
\mathcal{E}_{\lambda }(u):=\int_{\Omega }F_{\lambda }(u)\,\mathrm{d}x-\frac{1%
}{2}\int_{\Omega }\left( J\ast u\right) \,u\,\mathrm{d}x.
\end{equation*}%
In light of (a1) and (a4), it is easily seen that $|F_{\lambda }(s)|\leq
\frac{1}{\lambda }s^{2}$ for all $s\in \mathbb{R}$. In turn, this gives that
\begin{equation}
\mathcal{E}_{\lambda }(u)\leq \left( \frac{1}{\lambda }+\frac{\Vert J\Vert
_{W^{1,1}(\mathbb{R}^{2})}}{2}\right) \Vert u\Vert _{L^{2}(\Omega )}^{2},
\label{ener-abov}
\end{equation}%
thus $\mathcal{E}_{\lambda }$ is well defined in $L^{2}(\Omega
)$. Moreover, by the assumption on the kernel ${J}$ in \ref{J-ass} and
(a2), for any $\lambda <\lambda ^{\star }$, we have
\begin{equation}
\begin{split}
\mathcal{E}_{\lambda }(u)& \geq \frac{1}{4\lambda ^{\star }}\Vert u\Vert
_{L^{2}(\Omega )}^{2}-C_{\lambda ^{\star }}|\Omega |-\frac{1}{2}\Vert J\ast
u\Vert _{L^{2}(\Omega )}\Vert u\Vert _{L^{2}(\Omega )} \\
& \geq \left( \frac{1}{4\lambda ^{\star }}-\frac{\Vert J\Vert _{W^{1,1}(%
\mathbb{R}^{2})}}{2}\right) \Vert u\Vert _{L^{2}(\Omega )}^{2}-C_{\lambda
^{\star }}|\Omega |.
\end{split}
\label{ener}
\end{equation}%
Hence, setting $\lambda ^{\star }\leq \left[ 4\left( 1+\frac{\Vert J\Vert
_{W^{1,1}(\mathbb{R}^{2})}}{2}\right) \right] ^{-1}$, we are led to
\begin{equation}
\mathcal{E}_{\lambda }(u)\geq \Vert u\Vert _{L^{2}(\Omega )}^{2}-C_{b},\quad
\forall \,u\in L^{2}(\Omega ),\quad \forall \, \lambda \in (0,\lambda ^{\star}],  \label{ener2}
\end{equation}%
where $C_{b}>0$ is a constant independent of $\lambda $ as well as any other constant in the sequel unless it is explicitly pointed out.
Let us now take $v=\mu _{\lambda }$ in \eqref{w-nCH-l1}. By using %
\eqref{Reg-lambda}, \eqref{ener-abov}, \cite[Proposition 4.2]{CKRS2007} and
the definition of $\mu _{\lambda }$, we obtain
\begin{equation*}
\frac{\mathrm{d}}{\mathrm{d}t}\mathcal{E}_{\lambda }(\phi _{\lambda })+\Vert
\nabla \mu _{\lambda }\Vert _{L^{2}(\Omega )}^{2}+\int_{\Omega }\phi
_{\lambda }\,\mathbf{u}\cdot \nabla \mu _{\lambda }\,\mathrm{d}x=0.
\end{equation*}%
Thanks to \eqref{ener2}, we easily get
\begin{align*}
\left\vert \int_{\Omega }\phi _{\lambda }\,\mathbf{u}\cdot \nabla \mu
_{\lambda }\,\mathrm{d}x\right\vert & \leq \Vert \mathbf{u}\Vert _{L^{\infty
}(\Omega )}\Vert \phi _{\lambda }\Vert _{L^{2}(\Omega )}\Vert \nabla \mu
_{\lambda }\Vert _{L^{2}(\Omega )} \\
& \leq \frac{1}{2}\Vert \nabla \mu _{\lambda }\Vert _{L^{2}(\Omega )}^{2}+%
\frac{1}{2}\Vert \mathbf{u}\Vert _{L^{\infty }(\Omega )}^{2}\left( C_{b}+%
\mathcal{E}_{\lambda }(\phi _{\lambda })\right) .
\end{align*}%
Then, we find
\begin{equation}
\frac{\mathrm{d}}{\mathrm{d}t}\mathcal{E}_{\lambda }(\phi _{\lambda })+\frac{%
1}{2}\Vert \nabla \mu _{\lambda }\Vert _{L^{2}(\Omega )}^{2}\leq \frac{1}{2}%
\Vert \mathbf{u}\Vert _{L^{\infty }(\Omega )}^{2}\left( C_{b}+\mathcal{E}%
_{\lambda }(\phi _{\lambda })\right) .  \label{ener3}
\end{equation}%
In light of the general properties of the Yosida approximation of a convex
function (see, in particular, \cite[Proposition 1.8, Chapter IV]{SHOW1997}),
we recall that $F_{\lambda }(s)$ is increasing in $\lambda $ towards $F(s)$
for all $s\in \mathbb{R}$. Since $\Vert \phi _{0}\Vert _{L^{\infty }(\Omega
)}$, we infer that $\mathcal{E}_{\lambda }(\phi _{0})\leq \mathcal{E}(\phi
_{0})<\infty $. Therefore, it follows from the Gronwall lemma applied to %
\eqref{ener3} that
\begin{equation}
\mathcal{E}_{\lambda }(\phi _{\lambda }(t))\leq \left( \mathcal{E}(\phi
_{0})+\frac{C_{b}}{2}\int_{0}^{T}\Vert \mathbf{u}(\tau )\Vert _{L^{\infty
}(\Omega )}^{2}\,\mathrm{d}\tau \right) \mathrm{exp} \left(\int_{0}^{T}\mathbf{u}
(\tau )\Vert _{L^{\infty }(\Omega )}^{2}\,\mathrm{d}\tau \right),\quad \forall
\,t\in \lbrack 0,T].  \label{ener4}
\end{equation}%
Combining \eqref{ener2} with \eqref{ener4}, and integrating \eqref{ener3} on $%
[0,T]$, we obtain
\begin{equation}
\begin{split}
& \max_{t\in \lbrack 0,T]}\Vert \phi _{\lambda }(t)\Vert _{L^{2}(\Omega
)}^{2}+\frac{1}{2}\int_{0}^{T}\Vert \nabla \mu _{\lambda }(\tau )\Vert
_{L^{2}(\Omega )}^{2}\,\mathrm{d}\tau \\
&\leq 2\left( C_{b}+\mathcal{E}(\phi _{0})+\frac{C_{b}}{2}%
\int_{0}^{T}\Vert \mathbf{u}(\tau )\Vert _{L^{\infty }(\Omega )}^{2}\,%
\mathrm{d}\tau \right) 
 \mathrm{exp}\left(\int_{0}^{T}%
\mathbf{u}(\tau )\Vert _{L^{\infty }(\Omega )}^{2}\,\mathrm{d}\tau \right).
\end{split}
\label{ener5}
\end{equation}

Next, recalling that $F_{\lambda }^{\prime }$ is Lipschitz and $\phi
_{\lambda }(t)\in H^{1}(\Omega )$ for almost every $t\in \lbrack 0,T]$, it
follows from \cite{MM} that $\nabla F_{\lambda }^{\prime }(\phi _{\lambda
}(t))=F_{\lambda }^{\prime \prime }(\phi _{\lambda }(t))\nabla \phi
_{\lambda }(t)$ for almost every $x\in \Omega $ and $t\in \lbrack 0,T]$. By
definition of $\mu _{\lambda }$ and (a3), we clearly have
\begin{equation}
\left( \frac{\alpha }{1+\alpha }\right) ^{2}\int_{\Omega }|\nabla \phi
_{\lambda }|^{2}\,\mathrm{d}x\leq 2\Vert \nabla \mu _{\lambda }\Vert
_{L^{2}(\Omega )}^{2}+2\Vert J\Vert _{W^{1,1}(\mathbb{R}^{2})}^{2}\Vert \phi
_{\lambda }\Vert _{L^{2}(\Omega )}^{2}.  \label{ener6}
\end{equation}%
%
%
%
%
%
%
%
%
%
%
%
%
%
%
%
%
%
%
%
%
%
%
%
%
%
%
%
%
%
%
%
%
%
%
%
%
%
%
%
%
%
Let us now control $\overline{\mu _{\lambda }}$. Multiplying $\mu _{\lambda
} $ by $\phi _{\lambda }-\overline{\phi _{\lambda }}$ and integrating over $%
\Omega $, we find
\begin{equation*}
\int_{\Omega }F_{\lambda }^{\prime }(\phi _{\lambda })\left( \phi _{\lambda
}-\overline{\phi _{\lambda }}\right) \,\mathrm{d}x=\int_{\Omega }\mu
_{\lambda }\left( \phi _{\lambda }-\overline{\phi _{\lambda }}\right) \,%
\mathrm{d}x+\int_{\Omega }J\ast \phi _{\lambda }\left( \phi _{\lambda }-%
\overline{\phi _{\lambda }}\right) \,\mathrm{d}x.
\end{equation*}%
Observing that $(\overline{\mu _{\lambda }},\phi _{\lambda }-\overline{\phi }%
_{\lambda })=0$, we infer from the properties of $J$, the Poincar\'{e}
inequality, \eqref{energy} and the conservation of mass that
\begin{equation*}
\int_{\Omega }F_{\lambda }^{\prime }(\phi _{\lambda })\left( \phi _{\lambda
}-\overline{\phi }_{\lambda }\right) \,\mathrm{d}x=\int_{\Omega }\left( \mu
_{\lambda }-\overline{\mu _{\lambda }}\right) \left( \phi _{\lambda }-%
\overline{\phi _{\lambda }}\right) \,\mathrm{d}x+\int_{\Omega }J\ast \phi
_{\lambda }\left( \phi _{\lambda }-\overline{\phi _{\lambda }}\right) \,%
\mathrm{d}x\leq C\left( 1+\Vert \nabla \mu _{\lambda }\Vert _{L^{2}(\Omega
)}\right) .
\end{equation*}%
Now, we recall from \cite[Proof of Theorem 3.4]{GGG2017} (which is inspired
by \cite{MZ}) that
\begin{equation}
\Vert F_{\lambda }^{\prime }(\phi _{\lambda })\Vert _{L^{1}(\Omega )}\leq
C^{1}\left\vert \int_{\Omega }F_{\lambda }^{\prime }(\phi _{\lambda })\big(%
\phi _{\lambda }-\overline{\phi _{\lambda }}\big)\,\mathrm{d}x\right\vert
+C^{2},  \label{MZ-2}
\end{equation}%
where $C^{j}$, $j=1,2$, are positive constants that only depend on $F$, $\Omega $ and $\overline{\phi
_{0}}$. Then, combining
the above estimates with \eqref{MZ-2}, we have
\begin{equation}
\begin{split}
|\overline{\mu _{\lambda }}|& \leq \frac{1}{|\Omega |}\left( \int_{\Omega
}|F_{\lambda }^{\prime }(\phi _{\lambda })|\,\mathrm{d}x+\left\vert
\int_{\Omega }J\ast \phi _{\lambda }\,\mathrm{d}x\right\vert \right) \\
& \leq \frac{C^{1}}{|\Omega |}\left\vert \int_{\Omega }F_{\lambda }^{\prime
}(\phi _{\lambda })\big(\phi _{\lambda }-\overline{\phi _{\lambda }}\big)\,%
\mathrm{d}x\right\vert +\frac{C^{2}}{|\Omega |}+\frac{C\Vert J\Vert _{L^{1}(%
\mathbb{R}^{2})}}{|\Omega |}\Vert \phi _{\lambda }\Vert _{L^{2}(\Omega )} \\
& \leq C\left( 1+\Vert \nabla \mu _{\lambda }\Vert _{L^{2}(\Omega )}\right) ,
\end{split}
\label{mu5}
\end{equation}%
where $C$ depends on $F$, $\Omega $, $\overline{\phi _{0}}$, $\mathcal{E}%
(\phi _{0})$ and $\Vert \mathbf{u}\Vert _{L^{2}(0,T;L^{\infty }(\Omega ))}$.
On the other hand, concerning $\partial _{t}\phi _{\lambda }$, it is
immediate to check that
\begin{equation}
\Vert \partial _{t}\phi _{\lambda }\Vert _{H^{1}(\Omega )^{\prime }}\leq
\Vert \phi _{\lambda }\Vert _{L^{2}(\Omega )}\Vert \mathbf{u}\Vert
_{L^{\infty }(\Omega )}+\Vert \nabla \mu _{\lambda }\Vert _{L^{2}(\Omega )}.
\label{ener9}
\end{equation}%
%
%
%
%
%
%
%
%
%
%
%
%
%
%
%
%
%
%
%
%
%
%
%
%
%
%
%
%
%
%
%
%
%
%
%
%
%
%
%
%
%
Therefore, owing to \eqref{ener5}, \eqref{ener6}, \eqref{mu5} and \eqref{ener9}%
, we obtain
\begin{equation}
\Vert \phi _{\lambda }\Vert _{L^{\infty }(0,T;L^{2}(\Omega ))}+\Vert \nabla
\phi _{\lambda }\Vert _{L^{2}(0,T;L^{2}(\Omega ))}+\Vert \partial _{t}\phi
_{\lambda }\Vert _{L^{2}(0,T;H^{1}(\Omega )^{\prime })}+\Vert \mu _{\lambda
}\Vert _{L^{2}(0,T;H^{1}(\Omega ))}\leq C.  \label{energy}
\end{equation}%
In addition, we get by comparison that
\begin{equation}
\Vert F_{\lambda }^{\prime }(\phi _{\lambda })\Vert _{L^{2}(0,T;H^{1}(\Omega
))}\leq C.  \label{energy2}
\end{equation}%
\smallskip

\noindent \textbf{Sobolev estimates.}
We derive higher-order Sobolev estimates following the argument used in \cite%
{DPGG2018}. To this aim, we introduce the difference quotient $\partial
_{t}^{h}f(t)=h^{-1}\left( f(t+h)-f(t)\right) $ and the shift $%
S^{h}f(t)=f(t+h)$ for $0<t<T-h$. Subtracting now the weak formulation %
\eqref{w-nCH-l1} evaluated at time $t$ from the one at time $t+h$, dividing
by $h$ and choosing $\mathcal{N}\partial _{t}^{h}\phi _{\lambda }$ as test
function, we obtain
\begin{equation}
\frac{1}{2}\frac{\mathrm{d}}{\mathrm{d}t}\Vert \partial _{t}^{h}\phi
_{\lambda }\Vert _{\ast }^{2}-\int_{\Omega }S^{h}\phi _{\lambda }\,\partial
_{t}^{h}\mathbf{u}\cdot \nabla \mathcal{N}\partial _{t}\phi _{\lambda }\,%
\mathrm{d}x-\int_{\Omega }\partial _{t}^{h}\phi _{\lambda }\,\mathbf{u}\cdot
\nabla \mathcal{N}\partial _{t}^{h}\phi _{\lambda }\,\mathrm{d}%
x+\int_{\Omega }\partial _{t}^{h}\mu _{\lambda }\partial _{t}^{h}\phi
_{\lambda }\,\mathrm{d}x=0.  \label{differ}
\end{equation}%
By (a3), we have
\begin{equation}
\begin{split}
\int_{\Omega }\partial _{t}^{h}\mu _{\lambda }\partial _{t}^{h}\phi
_{\lambda }\,\mathrm{d}x& =\int_{\Omega }\frac{1}{h}\left( F_{\lambda
}^{\prime }(S^{h}\phi _{\lambda })-F_{\lambda }^{\prime }(\phi _{\lambda
})\right) \partial _{t}^{h}\phi _{\lambda }\,\mathrm{d}x-\int_{\Omega
}\left( J\ast \partial _{t}^{h}\phi _{\lambda }\right) \partial _{t}^{h}\phi
_{\lambda }\,\mathrm{d}x \\
& \geq \frac{\alpha }{1+\alpha }\Vert \partial _{t}^{h}\phi _{\lambda }\Vert
_{L^{2}(\Omega )}^{2}-\int_{\Omega }\left( J\ast \partial _{t}^{h}\phi
_{\lambda }\right) \partial _{t}^{h}\phi _{\lambda }\,\mathrm{d}x.
\end{split}
\label{mun}
\end{equation}%
Also, we observe that (cf. \eqref{mu-diff})
\begin{equation}
\begin{split}
(J\ast \partial _{t}^{h}\phi _{\lambda },\partial _{t}^{h}\phi _{\lambda })&
\leq \Vert \nabla J\ast \partial _{t}^{h}\phi _{\lambda }\Vert
_{L^{2}(\Omega )}\Vert \nabla \mathcal{N}\partial _{t}^{h}\phi _{\lambda
}\Vert _{L^{2}(\Omega )} \\
& \leq \frac{\alpha }{4(1+\alpha )}\Vert \partial _{t}^{h}\phi _{\lambda
}\Vert _{L^{2}(\Omega )}^{2}+C\Vert \partial _{t}^{h}\phi _{\lambda }\Vert
_{\ast }^{2}.
\end{split}%
\end{equation}%
On the other hand, we infer from \eqref{H_2}, \eqref{LADY} and \eqref{ener5}
that
\begin{align*}
\left\vert \int_{\Omega }S^{h}\phi _{\lambda }\,\partial _{t}^{h}\mathbf{u}%
\cdot \nabla \mathcal{N}\partial _{t}\phi _{\lambda }\,\mathrm{d}%
x\right\vert & \leq \Vert \partial _{t}^{h}\mathbf{u}\Vert _{L^{4}(\Omega
)}\Vert S^{h}\phi _{\lambda }\Vert _{L^{2}(\Omega )}\Vert \nabla \mathcal{N}%
\partial _{t}^{h}\phi _{\lambda }\Vert _{L^{4}(\Omega )} \\
& \leq \frac{\alpha }{4(1+\alpha )}\Vert \partial _{t}^{h}\phi _{\lambda
}\Vert _{L^{2}(\Omega )}^{2}+C\Vert \partial _{t}^{h}\mathbf{u}\Vert
_{L^{4}(\Omega )}^{2}
\end{align*}%
and
\begin{equation*}
\left\vert \int_{\Omega }\partial _{t}^{h}\phi _{\lambda }\,\mathbf{u}\cdot
\nabla \mathcal{N}\partial _{t}^{h}\phi _{\lambda }\,\mathrm{d}x\right\vert
\leq \frac{\alpha }{4(1+\alpha )}\Vert \partial _{t}^{h}\phi _{\lambda
}\Vert _{L^{2}(\Omega )}^{2}+C\Vert \mathbf{u}\Vert _{L^{\infty }(\Omega
)}^{2}\Vert \partial _{t}^{h}\phi _{\lambda }\Vert _{\ast }^{2}.
\end{equation*}%
Therefore, we derive from \eqref{differ} that
\begin{equation}  \label{increm}
\frac{1}{2}\frac{\mathrm{d}}{\mathrm{d}t}\Vert \partial _{t}^{h}\phi
_{\lambda }\Vert _{\ast }^{2}+\frac{\alpha }{4(1+\alpha )}\Vert \partial
_{t}^{h}\phi _{\lambda }\Vert _{L^{2}(\Omega )}^{2}\leq C\left( 1+\Vert
\mathbf{u}\Vert _{L^{\infty }(\Omega )}^{2}\right) \Vert \partial
_{t}^{h}\phi _{\lambda }^{k}\Vert _{\ast }^{2}+C\Vert \partial _{t}^{h}%
\mathbf{u}\Vert _{L^{4}(\Omega )}^{2}.
\end{equation}%
Recalling that $\overline{\phi _{\lambda }}(t)\equiv \overline{\phi _{0}}$,
and thereby $\overline{\partial _{t}\phi }(t)\equiv 0$, for all $t\in
\lbrack 0,T]$, we infer from \eqref{w-nCH-l1} that
\begin{equation}
\frac{1}{2}\frac{\mathrm{d}}{\mathrm{d}t}\Vert \phi _{\lambda }-{\phi }%
_{0}\Vert _{\ast }^{2}-\int_{\Omega }\phi _{\lambda }\,\mathbf{u}\cdot
\nabla \mathcal{N}(\phi _{\lambda }-{\phi }_{0})\,\mathrm{d}x+\int_{\Omega
}\mu \left( \phi _{\lambda }-\phi _{0}\right) \,\mathrm{d}x=0.
\label{contr-0h}
\end{equation}%
Observing that
\begin{equation*}
\int_{\Omega }\left( F_{\lambda }^{\prime }(\phi _{\lambda })-F_{\lambda
}^{\prime }(\phi _{0})\right) \left( \phi _{\lambda }-\phi _{0}\right) \,%
\mathrm{d}x\geq 0,
\end{equation*}%
and exploiting \eqref{ener5}, we obtain
\begin{align*}
-\int_{\Omega }\mu \left( \phi _{\lambda }-{\phi }_{0}\right) \,\mathrm{d}x&
=-\int_{\Omega }F_{\lambda }^{\prime }(\phi _{\lambda })\left( \phi
_{\lambda }-{\phi }_{0}\right) \,\mathrm{d}x+\int_{\Omega }J\ast \phi
_{\lambda }\left( \phi _{\lambda }-{\phi }_{0}\right) \,\mathrm{d}x \\
& \leq \int_{\Omega }F_{\lambda }^{\prime }(\phi _{0})\left( \phi _{\lambda
}-{\phi }_{0}\right) \,\mathrm{d}x+\int_{\Omega }J\ast \phi _{\lambda
}\left( \phi _{\lambda }-{\phi }_{0}\right) \,\mathrm{d}x \\
& \leq \Vert \nabla F_{\lambda }^{\prime }(\phi _{0})\Vert _{L^{2}(\Omega
)}\Vert \nabla \mathcal{N}(\phi _{\lambda }-{\phi }_{0})\Vert _{L^{2}(\Omega
)}+\Vert \nabla J\Vert _{L^{1}(\mathbb{R}^{2})}\Vert \phi _{\lambda }\Vert
_{L^{2}(\Omega )}\Vert \nabla \mathcal{N}(\phi _{\lambda }-{\phi }_{0})\Vert
_{L^{2}(\Omega )} \\
& \leq C(1+\Vert \nabla F_{\lambda }^{\prime }(\phi _{0})\Vert
_{L^{2}(\Omega )})\Vert \phi _{\lambda }-{\phi }_{0}\Vert _{\ast }.
\end{align*}%
Similarly, we have
\begin{equation*}
\left\vert \int_{\Omega }\phi _{\lambda }\,\mathbf{u}\cdot \nabla \mathcal{N}%
(\phi _{\lambda }-{\phi }_{0})\,\mathrm{d}x\right\vert \leq \Vert \mathbf{u}%
\Vert _{L^{\infty }(\Omega )}\Vert \phi _{\lambda }\Vert _{L^{2}(\Omega
)}\Vert \phi _{\lambda }-{\phi }_{0}\Vert _{\ast }\leq C\Vert \mathbf{u}%
\Vert _{L^{\infty }(\Omega )}\Vert \phi _{\lambda }-{\phi }_{0}\Vert _{\ast
}.
\end{equation*}%
In conclusion, integrating \eqref{contr-0h} in $(0,t)$ for $t\in (0,T)$, we
find
\begin{equation}
\frac{1}{2}\Vert \phi _{\lambda }(t)-{\phi }_{0}\Vert _{\ast }^{2}\leq
C\left( 1+\Vert \nabla F_{\lambda }^{\prime }(\phi _{0})\Vert _{L^{2}(\Omega
)}+\Vert \mathbf{u}\Vert _{L^{\infty }(0,T;L^{\infty }(\Omega ))}\right)
\int_{0}^{t}\Vert \phi _{\lambda }^{k}(s)-{\phi }_{0,k}\Vert _{\ast }\,%
\mathrm{d}s  \label{increm-2}
\end{equation}%
and a well-known version of the Gronwall lemma (see \cite[Lemma A.5]{Brezis})
implies that
\begin{equation*}
\frac{1}{2}\Vert \phi _{\lambda }(t)-{\phi }_{0}\Vert _{\ast }\leq t\left(
1+\Vert \nabla F_{\lambda }^{\prime }(\phi _{0})\Vert _{L^{2}(\Omega
)}+\Vert \mathbf{u}\Vert _{L^{\infty }(0,T;L^{\infty }(\Omega ))}\right)
,\quad \forall \,t\in (0,T).
\end{equation*}%
In order to obtain an uniform estimate in $\lambda $, we are left to control
$\Vert \nabla F_{\lambda }^{\prime }(\phi _{0})\Vert _{L^{2}(\Omega )}$. To
this aim, we first recall from \cite[Lemma 3.10]{GGG2017} that
\begin{equation*}
F_{\lambda }^{\prime \prime }(s)=\frac{1}{\lambda }\left[ 1-\frac{1}{%
1+\lambda F^{\prime \prime }(J_{\lambda }(s))}\right] ,\quad \forall \,s\in
(-1,1),
\end{equation*}%
where $J_{\lambda }$ is the resolvent operator. In light of \cite[Chapter
IV, Proposition 1.7]{SHOW1997}, $J_{\lambda }(s)\rightarrow s$ for all $s\in
(-1,1)$, which entails that $F_{\lambda }^{\prime \prime }(s)\rightarrow
F^{\prime \prime }(s)$ for all $s\in (-1,1)$. Furthermore, since $J_{\lambda
}(0)=0$ (cf. $F^{\prime }(0)=0$) and $J_{\lambda }$ is a contraction, $%
J_{\lambda }$ is bounded on compact subset of $(-1,1)$ independently of $%
\lambda $. Observing that $F_{\lambda }^{\prime \prime }(s)\leq F^{\prime
\prime }(J_{\lambda }(s))$, it follows that $F_{\lambda }^{\prime \prime
}(s) $ is also bounded on compact subset of $(-1,1)$ independently of $%
\lambda $. In particular, since $\Vert \phi _{0}\Vert _{L^{\infty }(\Omega
)}<1$, we have that $\Vert F_{\lambda }^{\prime \prime }(\phi _{0})\Vert
_{L^{\infty }(\Omega )}\leq C_{F}$, where $C_{F}$ is independent of $\lambda
$. By Lebesgue's dominated convergence theorem, we infer that
\begin{equation}
\lim_{\lambda \rightarrow 0}\Vert \nabla F_{\lambda }^{\prime }(\phi
_{0})\Vert _{L^{2}(\Omega )}=\lim_{\lambda \rightarrow 0}\Vert F_{\lambda
}^{\prime \prime }(\phi _{0})\nabla \phi _{0}\Vert _{L^{2}(\Omega )}=\Vert
F^{\prime \prime }(\phi _{0})\nabla \phi _{0}\Vert _{L^{2}(\Omega )}\leq
C_{F}\Vert \phi _{0}\Vert _{H^{1}(\Omega )}.  \label{F''-conv}
\end{equation}%
Therefore, choosing $t=h$ in \eqref{increm-2} and exploiting \eqref{F''-conv}
and $\mathbf{u}\in C_0^\infty((0,T);C_{0,\sigma }^{\infty }(\Omega ;\mathbb{R}%
^{2})$), we conclude that $\Vert \partial _{t}^{h}\phi _{\lambda }(0)\Vert _{\ast }\leq C$. Now, an application of Gronwall's lemma to \eqref{increm} entails that
\begin{equation*}
\begin{split}
& \max_{t\in \lbrack 0,T-h]}\Vert \partial _{t}^{h}\phi _{\lambda }(t)\Vert
_{\ast }^{2}+\int_{0}^{T}\Vert \partial _{t}^{h}\phi _{\lambda }(\tau )\Vert
_{L^{2}(\Omega )}^{2}\,\mathrm{d}\tau \\
& \quad \leq C\left( \Vert \partial _{t}^{h}\phi _{\lambda }(0)\Vert _{\ast
}^{2}+\int_{0}^{T}\Vert \partial _{t}^{h}\mathbf{u}(\tau )\Vert
_{L^{4}(\Omega )}^{2}\,\mathrm{d}\tau \right) \mathrm{exp} \left(CT+C\int_{0}^{T}\Vert \mathbf{u}(\tau )\Vert _{L^{\infty }(\Omega )}^{2}\,%
\mathrm{d}\tau  \right).
\end{split}%
\end{equation*}%
Recalling the inequality $\Vert \partial _{t}^{h}%
\mathbf{u}\Vert _{L^{2}(0,T-h;L^{4}(\Omega ))}\leq \Vert \partial _{t}%
\mathbf{u}\Vert _{L^{2}(0,T;L^{4}(\Omega ))}$, we conclude that
\begin{equation*}
\Vert \partial _{t}^{h}\phi _{\lambda }\Vert _{L^{\infty }(0,T;H^{1}(\Omega
)^{\prime })}+\Vert \partial _{t}^{h}\phi _{\lambda }\Vert
_{L^{2}(0,T;L^{2}(\Omega ))}\leq C,
\end{equation*}%
where $C$ is also independent of $h$. Passing then to the limit as
$h\rightarrow 0$, this gives
\begin{equation}
\Vert \partial _{t}\phi _{\lambda }\Vert _{L^{\infty }(0,T;H^{1}(\Omega
)^{\prime })}+\Vert \partial _{t}\phi _{\lambda }\Vert
_{L^{2}(0,T;L^{2}(\Omega ))}\leq C. \label{dtt}
\end{equation}%
Next, by comparison in %
\eqref{w-nCH-l1} and using \eqref{mu5}, we easily obtain that
\begin{equation*}
\Vert \mu _{\lambda }\Vert _{L^{\infty }(0,T;H^{1}(\Omega ))}+\Vert
F_{\lambda }^{\prime }(\phi _{\lambda })\Vert _{L^{\infty }(0,T;H^{1}(\Omega
))}\leq C.
\end{equation*}%
In addition, by elliptic regularity, we have
\begin{equation*}
\Vert \mu _{\lambda }\Vert _{H^{2}(\Omega )}\leq C\left( \Vert \partial
_{t}\phi _{\lambda }\Vert _{L^{2}(\Omega )}+\Vert \mathbf{u}\cdot \nabla
\phi _{\lambda }\Vert _{L^{2}(\Omega )}\right) \leq C\left( \Vert \partial
_{t}\phi _{\lambda }\Vert _{L^{2}(\Omega )}+\Vert \mathbf{u}\Vert
_{L^{\infty }(\Omega )}\Vert \nabla \phi _{\lambda }\Vert _{L^{2}(\Omega
)}\right) .
\end{equation*}%
Thus, thanks to \eqref{energy} and \eqref{dtt}, we also infer that
\begin{equation}
\Vert \mu _{\lambda }\Vert _{L^{2}(0,T;H^{2}(\Omega ))}\leq C,  \label{h2}
\end{equation}%
Finally, recalling that $F_{\lambda
}^{\prime \prime }(\phi _{\lambda })\nabla \phi _{\lambda }=\nabla \mu
_{\lambda }+\nabla J\ast \phi _{\lambda }$ almost everywhere in $\Omega
\times (0,T)$, we find (cf. also \eqref{Lp})
\begin{equation}
\Vert \nabla \phi _{\lambda }\Vert _{L^{p}(\Omega )}\leq C\left( 1+\Vert
\nabla \mu _{\lambda }\Vert _{L^{p}(\Omega )}\right) .  \label{mu6}
\end{equation}%
Then, by making use of the interpolation inequality $\Vert u\Vert
_{L^{q}(0,T;L^{p}(\Omega ))}\leq C\Vert u\Vert _{L^{\infty
}(0,T;L^{2}(\Omega ))}\Vert u\Vert _{L^{2}(0,T;H^{1}(\Omega ))}$, where $q=%
\frac{2p}{p-2}$ and $p\in (2,\infty )$, and by using \eqref{mu6} and %
\eqref{h2}, we are led to
\begin{equation}
\Vert \phi _{\lambda }\Vert _{L^{q}(0,T;W^{1,p}(\Omega ))}\leq C,\quad q=%
\frac{2p}{p-2},\quad \forall \,p\in (2,\infty ).  \label{Lq-Lp}
\end{equation}%
\smallskip

\noindent \textbf{Passage to the limit and further regularities.} Thanks to
the above estimates \eqref{energy}-\eqref{energy2}, \eqref{dtt}-\eqref{h2}
and to the convergence properties (a5), we deduce by standard compactness
arguments and by passing to the limit as $\lambda \rightarrow 0$ in %
\eqref{w-nCH-l1} that there exist $\phi \in L^{\infty }(0,T;H^{1}(\Omega
))\cap L^{q}(0,T;W^{1,p}(\Omega ))$, where $q=\frac{2p}{p-2}$ and $p\in
(2,\infty )$, such that $\partial _{t}\phi \in L^{\infty }(0,T;H^{1}(\Omega
)^{\prime })\cap L^{2}(0,T;L^{2}(\Omega ))$ and $\mu \in L^{\infty
}(0,T;H^{1}(\Omega ))\cap L^{2}(0,T;H^{2}(\Omega ))$, which satisfy
the problem \eqref{CH-strong_sense}.
Furthermore, by a classical argument for singular potentials (see,
e.g., the proof of \cite[Theorem 3.15]{GGG2017}), we deduce that $\phi \in
L^{\infty }(\Omega \times (0,T))$ such that $|\phi|<1$ almost
everywhere in $\Omega \times (0,T)$.
By comparison in \eqref{CH-strong_sense}$_1$, we infer that $F^{\prime }(\phi
)\in L^{\infty }(0,T;H^{1}(\Omega ))$. In light of assumption \ref{h3},
arguing as in \cite{GGG2022} we find $F^{\prime \prime }(\phi )\in
L^{\infty }(0,T;L^{p}(\Omega ))$ for all $p\in \lbrack 2,\infty )$. Owing to
this regularity, we can recast the argument in \cite[%
Section 4.1]{GGG2022} for the advective case by observing that the
corresponding drift term vanishes once again (i.e., $\mathcal{Z}=0,$ cf. %
\eqref{Z-term}, see the proof of Theorem \ref{ExistCahn}, (ii)). This yields
the existence of a constant $\delta >0$ such that $\Vert \phi \Vert
_{L^{\infty }(\Omega \times (0,T))}\leq 1-\delta $. To conclude this proof,
we are left to show an estimate for $\partial _{t}\mu $. We observe that
\begin{equation}
\partial _{t}^{h}\mu =\partial _{t}^{h}\phi \left( \int_{0}^{1}F^{\prime
\prime }(sS^{h}\phi +(1-s)\phi )\,\mathrm{d}s\right) -J\ast \partial
_{t}^{h}\phi ,\quad 0<t\leq T-h.  \label{a}
\end{equation}%
By the separation property, $\Vert sS^{h}\phi +(1-s)\phi )\Vert
_{L^{\infty }(\Omega \times (0,T-h))}\leq 1-\delta $ for all $s\in (0,1)$.
Then, by the properties of $J$ and exploiting that $\Vert \partial
_{t}^{h}\phi \Vert _{L^{2}(0,T-h;L^{2}(\Omega ))}\leq \Vert \partial
_{t}\phi \Vert _{L^{2}(0,T;L^{2}(\Omega ))}$, we obtain that
$
\Vert \partial
_{t}^{h}\mu ^{k}\Vert _{L^{2}(0,T-h;L^{2}(\Omega ))}\leq C,
$
where $C>0$ is independent of $h$. This implies that $\partial _{t}\mu \in L^{2}(0,T;L^{2}(\Omega ))$.
The proof is now completed.
\end{proof}

\medskip\noindent
{\bf Acknowledgments.} M. Grasselli and A. Poiatti have been partially funded by MIUR-PRIN research grant n. 2020F3NCPX.
M. Grasselli and A. Giorgini are members of Gruppo Nazionale per l'Analisi Ma\-te\-ma\-ti\-ca, la Probabilit\`{a} e le loro Applicazioni (GNAMPA), Istituto Nazionale di Alta Matematica (INdAM).

\end{document}